\documentclass[12pt]{amsart}
\usepackage[english]{babel}
\usepackage{amsmath, amsthm, amssymb}
\usepackage{enumitem}
\setlistdepth{10}\renewlist{itemize}{itemize}{10}\newcount\aux\aux0\loop\advance\aux1\setlist[itemize,\the\aux]{label={$\bullet$}}\ifnum\aux<10\repeat
\setlist[enumerate, 1]{label={\textup{(\roman*)}}}
\setlist[enumerate, 2]{label={\textup{(\alph*)}}}
\setlist{itemsep=0pt, topsep=\smallskipamount, leftmargin=1.5em, listparindent=1em}
\usepackage{geometry}
\usepackage{hyperref}

\theoremstyle{plain}
\newtheorem{thm}{Theorem}[section]
\newtheorem*{bigtheorem}{Theorem}
\newtheorem{lemma}[thm]{Lemma}
\newtheorem{prop}[thm]{Proposition}
\newtheorem{cor}[thm]{Corollary}

\theoremstyle{definition}
\newtheorem{defn}[thm]{Definition}
\newtheorem{example}[thm]{Example}

\newtheorem{algo}[thm]{Algorithm}
\newtheorem*{convention}{Convention}

\theoremstyle{remark}
\newtheorem{rem}[thm]{Remark}
\newtheorem*{claim}{Claim}

\def\uv#1{``#1''}

\let\zet\Z
\newcommand{\Q}{\mathbb{Q}}\let\kve\Q
\let\er\R

\let\phi\varphi

\let\epsilon\varepsilon

\def\zav#1{\left(#1\right)}
\def\set#1{\left\{#1\right\}}

\def\inv#1{#1^{-1}}

\def\nequiv{\not\equiv}
\let\OK\Ok
\newcommand{\mtrx}[1]{\left(\begin{matrix}#1\end{matrix}\right)}

\def\quatalg#1#2{{#1\overwithdelims()#2}}

\def\II{\mathbb I}
\def\KK{\mathbb K}
\def\pe{\mathbb P}
\DeclareMathOperator{\nrd}{nrd}
\DeclareMathOperator{\trd}{trd}
\DeclareMathOperator{\discrd}{discrd}
\DeclareMathOperator{\disc}{disc}
\DeclareMathOperator{\Mat}{M}
\DeclareMathOperator{\chr}{char}
\DeclareMathOperator{\im}{Im}
\let\ker\undefined\DeclareMathOperator{\ker}{ker}
\DeclareMathOperator{\Nm}{Nm}
\DeclareMathOperator{\Tr}{Tr}

\let\injto\hookrightarrow
\def\ii{\mathbf i}  
\def\j{\mathbf j}
\def\k{\mathbf k}
\def\h{\mathbf h}
\def\lcol#1#2{{(#1:#2)_{\text{\sffamily\upshape L}}}}
\def\rcol#1#2{{(#1:#2)_{\text{\sffamily\upshape R}}}}
\def\lord{\mathcal O_{\text{\sffamily\upshape L}}}
\def\rord{\mathcal O_{\text{\sffamily\upshape R}}}

\DeclareMathOperator{\Cl}{Cl}
\def\NCl{\Cl^+}
\DeclareMathOperator{\Cls}{Cls}

\title[Generalizing Hurwitz's proof of the four-square theorems]{Generalizing Hurwitz's quaternionic proof of Lagrange's and Jacobi's four-square theorems}
\author[]{Mat\v{e}j Dole\v{z}\'{a}lek}
\address{Charles University, Faculty of Mathematics and Physics, Department of Algebra,
Sokolovsk\'{a} 83, 186 75 Praha 8, Czech Republic}
\email{matej@gimli.ms.mff.cuni.cz}
\subjclass[2020]{11E12, 11E25, 11R52, 11R80}
\keywords{Universal quadratic form, quaternion order, class number, totally real number field}
\thanks{We acknowledge support by Charles University project GAUK No. 134824.}

\def\orderlisturl{\url{https://gimli.ms.mff.cuni.cz/~matej/perceptive-orders/}}

\begin{document}

\begin{abstract}
A~proof of Lagrange's and Jacobi's four-square theorem due to Hurwitz utilizes orders in a quaternion algebra over the rationals. Seeking a generalization of this technique to orders over number fields, we identify two key components: an order with a good factorization theory and the condition that all orbits under the action of the group of elements of norm $1$ acting by multiplication intersect the suborder corresponding to the quadratic form to be studied.
We use recent results on class numbers of quaternion orders and then find all suborders satisfying the orbit condition. Subsequently, we obtain universality and formulas for the number of representations by the corresponding quadratic forms.
We also present a quaternionic proof of Götzky's four-square theorem.
\end{abstract}

\maketitle

\tableofcontents

\section*{Introduction}

The study of quadratic forms is a long thread woven into the history of number theory, from the theorems on sums of two and four squares of Fermat and Lagrange respectively, all the way to the $290$-theorem of Bhargava and Hanke and beyond. Lagrange's four-square theorem is perhaps the single most apt prototype for the history of this area: it states that the positive definite quadratic form $t^2+x^2+y^2+z^2$ is \emph{universal} over $\zet$, that is to say that it expresses all positive integers. Universal forms have enjoyed much interest throughout history -- Dickson \cite{dickson} identified all universal \emph{diagonal} forms in four variables over $\zet$, while others like Siegel \cite{siegel} broadened the scope to consider quadratic forms over the rings of integers of number fields. Over the integers at least, the study of universal forms may be seen as effectively solved by the celebrated $290$-theorem due to Bhargava and Hanke \cite{bhargava-hanke}, which remarkably states that a positive definite form over $\zet$ is universal if and only if it expresses each of $29$ so-called \emph{critical integers}, the largest of which is $290$ (hence the name).

On the broader front over number fields, mostly \emph{totally real} number fields, the topic has enjoyed much interest in recent years. Various authors have studied diverse aspect of the matter, such as providing bounds on the minimal number of variables in a universal form over a given number field, relations to continued fractions in the case of quadratic fields or considering when a quadratic form with coefficients from $\zet$ may be universal over a number field \cite{blomer-kala1, blomer-kala2, cech-etal, chan-kim-raghavan, earnest-khosravani, kala1, kala2, BMkim, MHkim, kala-svoboda, krasensky-tinkova-zemkova, yatsyna}. Overall, it appears that universal forms in a small number of variables are rare. To a reader interested in the topic, we may recommend a survey paper \cite{kala-survey} by Kala.

This article is more specifically interested in the intersection of the study of quadratic forms with that of \emph{quaternions}, which came into prominence in 19th century based on geometric and algebraic motivations, though they may be seen as having been anticipated in Euler's four-square identity. Originally, Lagrange proved his four-square theorem elementarily using a descent technique, but Hurwitz \cite{hurwitz} provided a proof using what he called the \uv{integer quaternions}, though they are more commonly known as \emph{Hurwitz quaternions} after him or the \emph{Hurwitz order}. Further, Hurwitz simultaneously used the technique to prove Jacobi's four-square theorem, an extension of Lagrange's theorem which states not only that $t^2+x^2+y^2+z^2=n$ has an integer solution for all $n>0$, but that the exact number of these solutions is
\[
    8\sum_{4\nmid d\mid n}d.
\]
(Jacobi's original proof was analytic in nature). Soon after, quaternions over (totally real) number fields were also considered, providing results on sums of four squares in these number fields \cite{kirmse}.

\bigskip

The goal of this article is to generalize Hurwitz's approach to other quadratic forms over number fields and to find all situations where it may succeed, under a somewhat conservative notion of what constitutes \uv{a Hurwitz-like approach} we introduce, called \emph{perceptivity} (see Definition~\ref{def:perceptivity}). As such, it may be useful to overview and motivate Hurwitz's proof of the four-square theorems.

Hurwitz's approach starts with the observation that the sum of four squares may be expressed by the \emph{reduced norm} of quaternions $\mathbf q\overline{\mathbf q}$, where
\[
    \mathbf q = t+x\ii+y\j+z\k \mapsto \overline{\mathbf q} = t-x\ii-y\j-z\k
\]
(see Section~\ref{chap:prelims} for the notation of quaternion algebras) and that this norm is also multiplicative. It would then be natural to hope to use the ring $\zet\oplus\zet\ii\oplus\zet\j\oplus\zet\k$, the so-called \emph{Lipschitz order}, but complications arise due to this not being a (left or right) principal ideal domain, since e.g. its left ideal generated by $1+\ii$ and $1+\j$ is not principal. To fix this problem, Hurwitz used the order $\zet\oplus\zet\ii\oplus\zet\j\oplus\zet\frac{1+\ii+\j+\k}2$ that now bears his name, in which the coordinates of $1$, $\ii$, $\j$, $\k$ are not only allowed to all be integers, but to also (simultaneously) lie in $\zet+\frac12$. It turns out that the Hurwitz order is a (left and right) principal ideal domain (in fact it admits a left- and right-sided analogue of the Euclidean algorithm). Using this, Hurwitz proved that the reduced norm in the Hurwitz order expresses every positive integer (we will prove a straightforward generalization of this in Corollary~\ref{cor:nrd-universal}).

But since the objective is to prove universality of the reduced norm in the Lipschitz order, not the Hurwitz order, it is necessary to carry over the universality result from the better behaved larger order to the smaller order of interest. For this, one realizes that for any $\mathbf q$ from the Hurwitz order, there is a quaternion $\mathbf u$ of reduced norm $1$ such that $\mathbf u\mathbf q$ lies in the Lipschitz order. In other words, the action by these $\mathbf u$ with multiplication from the left partitions the Hurwitz order into orbits, each of which intersects the Lipschitz order.

To arrive at Jacobi's four-square theorem, Hurwitz's approach is again to first work in the Hurwitz order and count representations there by building a weaker, non-commutative version of \uv{unique }factorization into irreducible elements, and then using the action described above and calculating the sizes of intersections of individual orbits with the Lipschitz order. Thus he arrives at a formula for number of representations in the Lipschitz order by essentially weighing the formula for the Hurwitz order.

\bigskip

Our plan at generalizing Hurwitz's approach will then be to start with an order $\mathcal H$ (over a number field) that is a principal ideal domain -- to play a role analogous to the Hurwitz order -- and count representation by reduced norm there. We will overview a well-known approach to this by counting factorizations in $\mathcal H$ in Section~\ref{chap:factorizations}. Next, in Section~\ref{chap:orbits} we will endeavor to find all suborders $\mathcal G\subseteq\mathcal H$ such that for any $\mathbf q\in \mathcal H$, there is a $\mathbf u\in\mathcal H$ with reduced norm $1$ such that $\mathbf u\mathbf q\in \mathcal G$. A~finite enumeration of orders that are principal ideal domains (i.e. candidates for $\mathcal H$) follows from the work of Kirschmer and Lorch \cite{kirschmer-lorch}, who enumerated all orders in definite quaternion algebras over number fields with class number $1$ (see Section~\ref{chap:prelims} and Theorem~\ref{thrm:kirschmer-list}). This was preceded by partial results on this problem, when Kirschmer and Voight \cite{kirschmer-voight} provided such an enumeration restricted to so-called \emph{Eichler orders} and when Brzezinski \cite{brzezinski} did so for orders over $\zet$.
The main result of this article, namely Theorem~\ref{thrm:our-list} combined with the theorems of Subsection~\ref{sec:jacobi}, will be as follows:
\begin{bigtheorem}
Up to isomorphism, there are 111 orders in definite quaternion algebras over totally real number fields that are perceptive in some maximal order. Consequently, the norm form of each of them is universal and admits an explicit formula for the number of representation of a chosen totally positive element given by one of Theorems~\ref{thrm:kind-max-formula}, \ref{thrm:kind-q-formula}, \ref{thrm:kind-p-formula}, \ref{thrm:kind-q2-formula}, \ref{thrm:kind-p2-formula}, \ref{thrm:kind-qq-formula}, \ref{thrm:kind-pq-formula}, \ref{thrm:kind-p3-formula}, \ref{thrm:kind-q3-formula} and \ref{thrm:kind-q4-formula}.
\end{bigtheorem}
The list of these $111$ orders along with an implementation in Magma \cite{magma} that we used is available electronically from \orderlisturl.

Another source of inspiration for this article are certain works of Deutsch, who used quaternions and related notions of geometry of numbers to prove various results on universality or representation by several specific quadratic forms over number fields \cite{deutsch-on-gotzky, deutsch2, deutsch3, deutsch4, deutsch5, deutsch6, deutsch7}. Notably, \cite{deutsch-on-gotzky} contains a proof of universality of $t^2+x^2+y^2+z^2$ over $\kve(\sqrt5)$ using quaternions, which we will improve upon in Section~\ref{chap:gotzky} by giving a quaternionic proof of Götzky's four-square theorem, which gives a precise formula for the number of representations by this form, akin to Jacobi's four-square theorem.

\section*{Acknowledgments}
I~am grateful to V\'{i}t\v{e}zslav Kala for his helpful advice. I~would also like to thank John Voight for answering a few questions about the state of research regarding quaternion orders of small class number.

\section{Preliminaries}
\label{chap:prelims}

In this section, we give a brief overview of the theoretical baseline for this article.
For further background, see \cite{voight}.

The secondary role of this section is that of a repository of miscellaneous smaller lemmata.

\subsection{Number fields, quadratic forms}
Throughout the article, we consider totally real number fields $K$ with their rings of integers $\OK$. The field $K$ is equipped with $d$ real embedding $\sigma_1,\dots,\sigma_d:K\to\er$, where $d$ is the degree of $K$, and we say an $\alpha\in K$ is \emph{totally positive} (denoted $\alpha\succ0$) if $\sigma_i(\alpha)>0$ for all $i$. The subsets of totally positive elements in $K$ and $\OK$ are $K^+$ and $\OK^+$ respectively. We denote the field \emph{norm} and field \emph{trace} of $\alpha$ as $\Nm_{K/\kve}(\alpha) =\prod_{i=1}^d\sigma_i(\alpha)$ and $\Tr_{K/\kve}(\alpha)=\sum_{i=1}^d\sigma_i(\alpha)$ respectively; we will drop the subscript $K/\kve$ when it is clear from the context.

We say that a quadratic form over $K$ is \emph{totally positive definite} if it attains totally positive values at every point aside from $0$. Lastly, a quadratic form $Q$ in $r$ variables over $\OK$ is said to \emph{represent} an $\alpha\in\OK$ if there are $\beta_1,\dots,\beta_r\in\OK$ such that $Q(\beta_1,\dots,\beta_r)=\alpha$, and a totally positive definite $Q$ is \emph{universal}, if it represents all elements of $\OK^+$.

\subsection{Algebras with involutions, quaternion algebras}
\label{sec:algebras}

We say an algebra $\mathcal A$ over a field $F$ is equipped with a \emph{standard involution} $\mathbf x\mapsto \overline{\mathbf x}$ if
\[
\overline 1=1,\qquad \overline{(\overline{\mathbf x})}=\mathbf x\qquad,\overline{(\mathbf x\mathbf y)} = \overline{\mathbf y}\,\overline{\mathbf x}\qquad\text{and}\qquad \mathbf x\overline{\mathbf x}\in F
\]
is satisfied for all $\mathbf x, \mathbf y\in\mathcal A$. Then, we define the \emph{reduced norm} and \emph{reduced trace} as
\[
    \nrd(\mathbf x) := \mathbf x\overline{\mathbf x},\qquad \trd(\mathbf x) := \mathbf x+\overline{\mathbf x}.
\]
These are a multiplicative and an additive map $\mathcal A\to F$ respectively.
We may also observe that any $\mathbf x\in\mathcal A$ satisfies
\[
    \mathbf x^2-\mathbf x\trd(\mathbf x)+\nrd(\mathbf x) = \mathbf x^2-\mathbf x(\mathbf x+\overline{\mathbf x})+\mathbf x\overline{\mathbf x} = 0.
\]
Hence every element of an $\mathcal A$ satisfies a quadratic equation over $F$. As a consequence, any subspace spanned by $1$ and $\mathbf x$ is a subalgebra of $\mathcal A$.

\begin{lemma}
\label{lem:switch-the-order}
Let $\mathcal O$ be a subring of an $F$-algebra $\mathcal A$ with a standard involution. If $R$ is a subring of $F$ such that $\trd(\mathcal O)\subseteq R$, then for any $\mathbf x,\mathbf y\in \mathcal O$, the $R$-submodule $\mathcal B:=R+R\mathbf x+R\mathbf y+R\mathbf x\mathbf y$ is a subring of $\mathcal A$.
\end{lemma}
\begin{proof}
    The only nontrivial part is to verify that $\mathcal B$ is closed under multiplication, which may be done by simply checking that all possible products of the four generating elements $1$, $\mathbf x$, $\mathbf y$, $\mathbf x\mathbf y$ lie in $\mathcal B$. The products involving $1$ are trivial and the products of an element with itself follow from
    \[
        \mathbf x^2 = \underbrace{\trd(\mathbf x)}_{\in R}\mathbf x-\underbrace{\nrd(\mathbf x)}_{\in R} \in R+R\mathbf x,
    \]
    where we use that fact that $\trd(\mathcal O)\subseteq R$.
    Next we have $\mathbf x\cdot\mathbf y\in \mathcal B$ by construction and
    then we calculate
    \begin{align*}
        \mathbf y\mathbf x &= \trd(\mathbf y\mathbf x)-\overline{\mathbf x}\,\overline{\mathbf y} = \trd(\mathbf y\mathbf x) - (\trd(\mathbf x)-\mathbf x)(\trd(\mathbf y)-\mathbf y) = \\
        &= \underbrace{\trd(\mathbf y\mathbf x) - \trd(\mathbf x)\trd(\mathbf y)}_{\in R}+\underbrace{\trd(\mathbf x)}_{\in R}\mathbf y+\underbrace{\trd(\mathbf y)}_{\in R}\mathbf x - \mathbf x\mathbf y,
    \end{align*}
    which lies in $\mathcal B$. Then we see that
    \[
        \mathbf x\cdot\mathbf x\mathbf y \in (R+R\mathbf x)\mathbf y = R\mathbf y+R\mathbf x\mathbf y \subseteq \mathcal B
    \]
    and similarly $\mathbf x\mathbf y\cdot\mathbf y \in\mathcal B$. Lastly
    \[
        \mathbf x\mathbf y\cdot\mathbf x = \mathbf x(\mathbf y\mathbf x) \in \mathbf x(R+R\mathbf x+R\mathbf y+R\mathbf x\mathbf y),
    \]
    which lies in $\mathcal B$, because we've already checked all products where $\mathbf x$ is the left multiplicand. Analogously, we obtain $\mathbf y\cdot\mathbf x\mathbf y\in\mathcal B$.
\end{proof}

\begin{lemma}
\label{lem:threedim-subalg}
Let $B$ be a four-dimensional $k$-algebra with a standard involution and let $A$ be a three-dimensional subalgebra of $B$. Then whenever $1,\mathbf x,\mathbf y$ is a basis of $A$, we may choose $\tilde{\mathbf x}\in \mathbf x+k$, $\tilde{\mathbf y}\in \mathbf y+k$ such that $\tilde{\mathbf x}\tilde{\mathbf y} = 0$ and $\nrd(\tilde{\mathbf x})=\nrd(\tilde{\mathbf y}) = 0$.
\end{lemma}
\begin{proof}
Since $1,\mathbf x,\mathbf y$ is a basis of $A$, we have
\[
    \mathbf x\mathbf y = c_1+c_2\mathbf x+c_3\mathbf y
\]
for some $c_1,c_2,c_3\in k$. Denoting $\tilde{\mathbf x}:= \mathbf x-c_3$ and $\tilde{\mathbf y} := \mathbf y-c_2$, we then obtain
\[
    \tilde{\mathbf x}\tilde{\mathbf y} = c_1-c_2c_3 =: c \in k,
\]
and we wish to obtain $c=0$.
Since we only shifted in the direction of the other basis element, we see that $1,\tilde{\mathbf x}, \tilde{\mathbf y}$ is still a basis of $A$. If now $c\neq 0$, it would mean that both $\tilde{\mathbf x}$, $\tilde{\mathbf y}$ are invertible and hence
\[
    \tilde{\mathbf y} = c\inv{\tilde{\mathbf x}} = c\nrd(\tilde{\mathbf x})^{-1}\overline{\zav{\tilde{\mathbf x}}} \in k+k\tilde{\mathbf x},
\]
which is a contradiction. So $\tilde{\mathbf x}\tilde{\mathbf y}=0$ as we wanted.

Now, suppose for the sake of contradiction that $\nrd(\tilde{\mathbf x})\neq0$, then $\tilde{\mathbf x}$ is invertible, so $\tilde{\mathbf x}\tilde{\mathbf y} = 0$ implies $\tilde{\mathbf y} = 0$, which is absurd, since $\tilde{\mathbf y}$ belongs to some basis. $\nrd(\tilde{\mathbf y})=0$ follows analogously.
\end{proof}

A~\emph{quaternion algebra} over a field $F$ of $\chr F\neq2$ is a four-dimensional $F$-algebra of the form
\[
    \quatalg{a,b}F := F \oplus F\ii\oplus F\j\oplus F\k,\quad a,b\in F^\times
\]
with multiplication given by
\[
    \ii^2=a, \qquad \j^2=b, \qquad \k=\ii\j=-\j\ii.
\]
Alternatively, quaternion algebras can be characterized as four-dimensional \emph{central simple algebras}. We leave out the explicit definition of quaternion algebras in characteristic $2$, since we will work mainly with algebras over number fields.

We may equip $\mathcal A:=\quatalg{a,b}F$ with an anti-involution
\[
    \mathbf q = t+x\ii+y\j+z\k \mapsto \overline{\mathbf q}:=t-x\ii-y\j-z\k.
\]
Due to $\mathbf q\overline{\mathbf q} = t^2-ax^2-by^2+abz^2$, this is a standard involution.

A~notable case of a quaternion algebra is the algebra $\quatalg{1,1}F$, which is isomorphic to the ring $\Mat_2(F)$ of $2\times2$ matrices over $F$. A~quaternion algebra over $F$ is always either a division algebra -- then we say $\mathcal A$ is a \emph{definite} quaternion algebra -- or it is isomorphic to $\Mat_2(F)$ and we say $\mathcal A$ is an \emph{indefinite} quaternion algebra.

Right ideals in a ring of $n\times n$ matrices over $F$ correspond bijectively to subspaces $L\subset F^n$ via $L\mapsto \set{\mathbf a\in\Mat_n(F)\mid \im\mathbf a\subset L}$, so in particular for $\Mat_2(F)$, we get that apart from the two trivial ideals, all non-trivial ideal correspond to lines in $F^2$. In particular, $\Mat_2(F)$ is a (right) PID. Since $\Mat_2(F)$ is isomorphic to its opposite ring via the standard involution, we also have an analogous result for left ideals.

\subsection{Orders, ideals, completions}

When $\mathcal A$ is a quaternion algebra over a number field $K$, we may consider $\OK$-lattices in it. An \emph{$\OK$-order}, or just \emph{order} for short, is an $\OK$-lattice that is simultaneously a subring of $\mathcal A$. On the other hand, starting with an arbitrary lattice $L$, we obtain its \emph{left order}
\[
    \lord(L) := \set{\mathbf x\in\mathcal A\mid \mathbf xL\subseteq L}.
\]
This is always an order, and analogously, we define the \emph{right order} $\rord(L)$.
For two lattices $L$ and $M$, we similarly define their \emph{left colon lattice}
\[
    \lcol LM := \set{\mathbf x\in \mathcal A\mid \mathbf x M\subseteq L}.
\]
This is a lattice, and we may note $\lord(L) = \lcol LL$. Analogously, we define the \emph{right colon lattice}
\(
    \rcol LM
\).

For an order $\mathcal H$ in a quaternion algebra $\mathcal A$, we denote its set of elements of reduced norm $1$ as
\[
    \mathcal H^1 := \set{\mathbf u\in\mathcal H\mid\nrd(\mathbf u)=1}.
\]
When $\mathcal A$ is a definite quaternion algebra over a totally real number field, then $\mathcal H^1$ is guaranteed to be finite.

\begin{lemma}
\label{lem:its-an-order}
Let $\mathcal G\subseteq\mathcal H$ be orders.
\begin{enumerate}
\item If $\mathcal I$ is a two-sided ideal of $\mathcal H$, then $\mathcal G+\mathcal I$ is an order. As a special case, we may take $\mathcal I:=\mathfrak a\mathcal H$ for any ideal $\mathfrak a$ of $\OK$.
\item If $\mathbf z\in\mathcal H$, then $\mathcal G+\mathcal G\mathbf z$ is an order. Similarly for $\mathcal G+\mathbf z\mathcal G$.
\end{enumerate}
\end{lemma}
\begin{proof}
\begin{enumerate}
\item The only non-trivial part is to verify that $\mathcal G+\mathcal I$ is closed under multiplication. Letting $\mathbf q_1,\mathbf q_2\in \mathcal G$, $\mathbf a_1, \mathbf a_2\in \mathcal I$, we obtain
\[
(\mathbf q_1+\mathbf a_1)(\mathbf q_2+\mathbf a_2) = \underbrace{\mathbf q_1\mathbf q_2}_{\in\mathcal G} + \underbrace{\mathbf a_1\mathbf q_2}_{\in\mathcal I}+\underbrace{\mathbf q_1\mathbf a_2}_{\in\mathcal I}+\underbrace{\mathbf a_1\mathbf a_2}_{\in\mathcal I} \in \mathcal G+\mathcal I.
\]
\item Again, we only have to prove $\mathcal G+\mathcal G\mathbf z$ is closed under multiplication. Since it is generated as an $\OK$-module by elements of the forms $\mathbf q$ and $\mathbf q\mathbf z$ for $\mathbf q\in\mathcal G$, it suffices to verify for that the product of any two such elements lies in $\mathcal G+\mathcal G\mathbf z$ again. Clearly $\mathbf q_1\mathbf q_2\in\mathcal G$ for $\mathbf q_1,\mathbf q_2\in\mathcal G$ and
\[
\mathbf q_1(\mathbf q_2\mathbf z) = \underbrace{(\mathbf q_1\mathbf q_2)}_{\in\mathcal G}\mathbf z~\in \mathcal G\mathbf z.
\]
Next, similarly to the proof of Lemma~\ref{lem:switch-the-order}, we use the fact that
\[
    \mathbf y\mathbf x \in \OK+\OK\mathbf x+\OK\mathbf y+\OK\mathbf x\mathbf y
\]
for any $\mathbf x,\mathbf y\in\mathcal H$ to get
\begin{align*}
    (\mathbf q_1\mathbf z)\mathbf q_2 &\in \mathbf q_1(\OK+\OK\mathbf q_2 +\OK\mathbf z+\OK\mathbf q_2\mathbf z) =\\&= \underbrace{(\OK\mathbf q_1+\OK\mathbf q_1\mathbf q_2)}_{\subseteq\mathcal G} + \underbrace{(\OK\mathbf q_1+\OK\mathbf q_1\mathbf q_2)}_{\subseteq\mathcal G}\mathbf z~\subseteq \mathcal G+\mathcal G\mathbf z,\\
    (\mathbf q_1\mathbf z)(\mathbf q_2\mathbf z) &= (\mathbf q_1\mathbf z\mathbf q_2)\mathbf z~\in (\mathcal G +\mathcal G\mathbf z)\mathbf z~= \mathcal G\mathbf z+\mathcal G\mathbf z^2 \subseteq\\&\subseteq \mathcal G\mathbf z+\mathcal G(\OK+\OK\mathbf z) \subseteq\mathcal G+\mathcal G\mathbf z.
\end{align*}

The result for $\mathcal G+\mathbf z\mathcal G$ is proved analogously.
\qedhere
\end{enumerate}
\end{proof}

\emph{Maximal orders} (meaning inclusion-maximal orders) always exist in a quaternion algebra over a number field, playing a role analogous to the ring of integers of a number field. Unlike in the commutative setting however, there may be many distinct maximal orders.

\begin{prop}[{\cite[Corollary 10.5.4]{voight}}]
\label{prp:maximals-in-matrix-ring}
If $R$ is a PID, $F$ its fraction field and $\mathcal A = \Mat_2(F)$, then any maximal $R$-order $\mathcal H\subset\mathcal A$ is isomorphic to $\Mat_2(R)$.
\end{prop}

\subsection{Localizations, completions}
\label{subsec:loc-comp}

For any prime ideal $\mathfrak p$ of $\OK$, we may consider the localization $\mathcal O_{K,(\mathfrak p)}$ and its $\mathfrak p$-adic completion $\mathcal O_{K,\mathfrak p}$. Further, we may take $K_{\mathfrak p}:=K\otimes_{\OK}\mathcal O_{K,\mathfrak p}$. We may also apply $-\otimes_{K}K_{\mathfrak p}$ to a quaternion algebra $\mathcal A$ over $K$ to obtain $\mathcal A_{\mathfrak p}$, a quaternion algebra over $K_{\mathfrak p}$, and then apply $-\otimes_{\OK}\mathcal O_{K,\mathfrak p}$ to any $\OK$-lattice $L\subset\mathcal A$ to obtain $L_{\mathfrak p}\subset\mathcal A_{\mathfrak p}$.

The usefulness of localizations and completions stems from the ability to only check certain properties locally -- notably, being an order and being an ideal of an order may be checked locally. Here, let us summarize the contents of Corollary 9.4.4, Lemma 9.4.6 and Theorem 9.5.1 of \cite[Chapter 9]{voight}:
\begin{thm}[Local-global dictionary]
Let $\mathcal A$ be a quaternion algebra over $K$ and let $M,N\subseteq V$ be $\OK$-lattices.
\begin{enumerate}
    \item $M\subseteq N$, if and only if $M_{\mathfrak p}\subseteq N_{\mathfrak p}$ for all $\mathfrak p$. In particular, $M=N$ if and only if $M_{\mathfrak p}=N_{\mathfrak p}$ for all $\mathfrak p$.
    \item If $M$ is fixed, then $N\mapsto (N_{\mathfrak p}\mid \text{prime $\mathfrak p$ of $\OK$})$ yields a bijection between $\OK$-lattices in $\mathcal A$ and collections of $\mathcal O_{K,\mathfrak p}$-lattices in $\mathcal A_{\mathfrak p}$ that only differ from $(M_{\mathfrak p}\mid \text{prime $\mathfrak p$ of $\OK$})$ in finitely many positions.
\end{enumerate}
\end{thm}
A~more informal wording of part (ii) is that if we start with a lattice $M$ and change $M_{\mathfrak p}$ at finitely many $\mathfrak p$, the result uniquely determines an $\OK$-lattice again.

A~further strength of localizations and completions comes from the fact that over a completion $K_{\mathfrak p}$ of a number field, there is only one division quaternion algebra up to isomorphism. This allows us to say that a quaternion algebra over $K_{\mathfrak p}$ is either the matrix ring, or \uv{the} division algebra.

This division quaternion algebra $\mathcal A_{\mathfrak p}$ may be given explicitly: if $\chr(\OK/\pi)\neq2$, the (principal) ideal $\mathfrak p \mathcal O_{K,\mathfrak p}$ of $\mathcal O_{K,\mathfrak p}$ is generated by some $\pi$ and $d\in \mathcal O_{K,\mathfrak p}$ is chosen such that it becomes a quadratic non-residue in $\mathcal O_{K,\mathfrak p}/\mathfrak p\mathcal O_{K,\mathfrak p}\simeq\OK/\pi$, then $\quatalg{d,\pi}{K_{\mathfrak p}}\simeq\mathcal A_{\mathfrak p}$ (see \cite[Chapter 13]{voight} for further details). As with the definition of quaternion algebras themselves, characteristic $2$ brings technical complications, but the result about uniqueness still holds.

Moreover, the division quaternion algebra over $K_{\mathfrak p}$ has a single, unique maximal order $\mathcal H_{\mathfrak p}$. This order has only one maximal (left or right) ideal $P$, which is a two-sided ideal and satisfies $P^2 = \pi\mathcal H_{\mathfrak p}$ \cite[Theorem 13.3.10]{voight}.

\subsection{Indices, discriminants}
\label{sec:inddisc}

When $L$ and $M$ are two $\OK$-lattices in an $K$-vector space $V$, we define their \emph{$\OK$-index} $[L:M]_{\OK}$ as the fractional ideal generated by $\det(\delta)\in K$ for $K$-linear endomorphisms $\delta:V\to V$ such that $\delta(L)\subseteq M$. Of note is the case when $M\subseteq L$, then we may consider that $L/M$ is a torsion $\OK$-module. Appealing to the structure theorem for finitely generated modules over a Dedekind ring (\cite[Theorem 10.3.10]{broue}), we then decompose $L/M$ as some direct sum of cyclic modules, i.e.
\[
    L/M \simeq \OK/I_1\oplus\cdots\oplus \OK/I_n.
\]
Then it holds that $[L:M]_{\OK} = I_1\cdots I_n$, so in particular if $M\subseteq L$, then $M=L$ if and only if $[L:M]_{\OK} = \OK$. Another consequence of this is that $[L:M]_{\OK}\cdot L \subseteq M$. Further, we have $\Nm_{K/\kve}([L:M]_{\OK}) = [L:M]_{\zet}$.
Indices also commute with completions, i.e. $([M:N]_{\OK})_{\mathfrak p} = [M_{\mathfrak p}:N_{\mathfrak p}]_{\mathcal O_{K,\mathfrak p}}$

For any $\OK$-lattice $L$ in a quaternion algebra $\mathcal A$ over $K$, its \emph{discriminant} $\disc(L)$ is the fractional ideal of $\OK$ generated by all $\det(\trd(\mathbf x_i\mathbf x_j))_{i,j=1,\dots,4}$ as $(\mathbf x_1,\dots,\mathbf x_4)$ runs through all quadruplets of elements of $L$. It turns out that $\disc(L)$ is always a square of an ideal of $\OK$, and this ideal is called the \emph{reduced discriminant} $\discrd(L)$ (this might not hold for all lattices over more general rings). It may be explicitly constructed as the ideal generated by all
\[
    \trd((\mathbf x_1\mathbf x_2-\mathbf x_2\mathbf x_1)\overline{\mathbf x_3})
\]
as $(\mathbf x_1,\mathbf x_2,\mathbf x_3)$ runs through all triplets of elements of $L$. See \cite[Chapter 15]{voight} for more details.

\begin{prop}
\label{prp:discs-and-indices}
Let $\mathcal A$ be a definite quaternion algebra  over $K$.
\begin{enumerate}
\item For $\OK$-lattices $L,M\subset\mathcal A$, it holds that $\discrd(M) = [L:M]_{\OK}\discrd(L)$.
\item For an $\OK$-lattice $L\subseteq \mathcal A$ and $\mathbf q\in\mathcal A$, it holds that $[L:\mathbf q L]_{\OK} = {[L:L\mathbf q]}_{\OK} = \nrd(\mathbf q)^2$.
\end{enumerate}
\end{prop}
Further, the discriminant $\disc\mathcal A$ of the quaternion algebra $\mathcal A$ is defined as the product of all those primes $\mathfrak p$ of $\OK$ for which $\mathcal A_{\mathfrak p}$ is a division algebra (or, rather, the division algebra); there always only finitely many such primes.
Notably, whenever $\mathcal H$ is a maximal $\OK$-order in $\mathcal A$, then $\discrd\mathcal H = \disc \mathcal A$ (\cite[Theorem 23.2.9]{voight}).

\subsection{Ideal class theory}
\label{sec:ideal-classes}

We use $\Cl K$ to refer to the \emph{ideal class group} of $K$ and $\NCl K$ to the \emph{narrow class group}. Both are always finite and $\#\Cl K=1$ occurs if and only if $\OK$ is a PID, while $\#\NCl K=1$ occurs if and only if $\OK$ is a PID in which all totally positive units are squares (of units).

In the non-commutative setting of quaternion orders, we lose the group structure, leaving only an \emph{ideal class set}.

\begin{defn}
\label{def:cls}
Let $\mathcal A$ be a definite quaternion algebra over $K$. Amongst $\OK$-lattices in $\mathcal A$, we say $L$ and $M$ are \emph{(right) equivalent}, which we denote $L\sim_{\text{\sffamily R}}M$, if $L = \mathbf x M$ for some $\mathbf x\in \mathcal A^\times$; let us denote the equivalence classes of this relation as $[L]_{\text{\sffamily R}}$. For an $\OK$-order $\mathcal H\subset \mathcal A$, we then define its \emph{(right) class set} as
\[
    \Cls\mathcal H := \set{[L]_{\text{\sffamily R}}\mid \text{$L$ is a lattice with $\rord(L)=\mathcal H$}}.
\]
The cardinality of $\Cls\mathcal H$ is called the \emph{(right) class number} of $\mathcal H$.
\end{defn}
The left class set is defined analogously, though they are always in bijection due to the isomorphism $\mathcal H\simeq\mathcal H^{\text{op}}$ given by the standard involution, hence we will usually omit the designation of right or left.
Note that
when $\mathcal H$ is maximal, $\#\Cls\mathcal H=1$ is equivalent to $\mathcal H$ being a right PID (and thus also a left PID).
Similarly to the class group of number fields, the class set of a quaternion order is always guaranteed to be finite. This is due to an analogue of Minkowski's bound (see \cite[Section 17.7]{voight} for more details).

Kirschmer and Lorch \cite{kirschmer-lorch} obtained an enumeration of all orders in definite quaternion algebras with class number $1$, up to isomorphism, which we will use in Section~\ref{chap:orbits}. In this context, it will be useful to note some implications this has on any superorders as well as the base number field $K$.

\begin{prop}[{\cite[Exercise 17.3]{voight}, \cite[Remark 6.3]{kirschmer-lorch}}]
\label{prp:cls-implications}
Let $\mathcal G\subseteq\mathcal H$ be $\OK$-orders in a definite quaternion algebra over $K$. Then  $\#\Cls \mathcal G\geq\#\Cls\mathcal H\geq\#\NCl K$.
\end{prop}

\section{Factorization in principal ideal quaternion orders}
\label{chap:factorizations}

In this section, we will examine factorizations into irreducible elements in quaternion orders. These are not unique, but we will show that when the order is a principal ideal domain, uniqueness up to certain manipulations holds, which allows us to count elements of a given reduced norm based on these factorizations.

Note that even though a quaternion order $\mathcal H$ is a non-commutative ring, we know that the standard involution preserves orders and maps their left ideal to right ideals and vice versa. Hence the notions of a left PID and a right PID coincide, which is why we will call them just PIDs.

\begin{convention}
Throughout this section, let us fix a totally real number field $K$ and let $\OK$ be its ring of integers; further, let us presume that $K$ has narrow class number $1$, i.e. that it is a PID and that all of its totally positive units are squares.
We will be considering a maximal $\OK$-order $\mathcal H$ that is a PID, in a definite quaternion algebra $\mathcal A$ over $K$.
\end{convention}

As discussed in Subsection~\ref{sec:ideal-classes}, a maximal order $\mathcal H$ is a PID if and only if it has class number $1$, and by Proposition~\ref{prp:cls-implications}, this forces $K$ to have narrow class number $1$. Further, a non-maximal order cannot be a PID, e.g. since it will always contain some principal ideals of a large order, which then cannot be principal in this smaller order. Hence the above presumptions are justified.

\subsection{Irreducible elements}

Let us start by identifying the irreducible elements of $\mathcal H$, that is those that cannot be written as a product of two non-invertible elements -- we will show they are exactly the quaternions of prime reduced norm, and for every totally positive prime element of $\OK$, there exists a quaternion of that reduced norm in $\mathcal H$.

The elementary proofs we provide here will be later somewhat superseded by the local counting arguments of Subsection~\ref{sec:irredcount}

\begin{lemma}
\label{lem:finite-field-isotropy}
Over a finite field $k$, any quadratic form $Q$ in three or more variables is isotropic, i.e. it represents $0$.
\end{lemma}
\begin{proof}
    In odd characteristic, this is proved in \cite[62:1b]{omeara}. For the statement including characteristic 2, see \cite[Exercise 12.6]{voight}.´
\end{proof}

\begin{prop}
    \label{prp:prime-norms}
    Let $\pi$ be a totally positive prime element of $\OK$. Then there exists a quaternion $\mathbf p\in \mathcal H$ of reduced norm $\pi$.
\end{prop}
\begin{proof}
    Reducing the quadratic form $(\mathcal H, \nrd)$ modulo $\pi$, we obtain a quaternary quadratic form over the finite field $\OK/\pi\OK$. By the Lemma~\ref{lem:finite-field-isotropy}, this has an isotropic vector, which corresponds to a quaternion $\mathbf q\in\mathcal H$ with $\nrd(\mathbf q)\in \pi\OK$ but $\mathbf q\notin \pi\mathcal H$.

    Since $\mathcal H$ is a PID, there exists a generator $\mathbf p$ of the left ideal $\mathcal H\mathbf p=\mathcal H\mathbf q+\mathcal H\pi$. On one hand, calculating in $\mathcal H/\pi\mathcal H$ immediately yields $\nrd(\mathbf p)\equiv 0\pmod \pi$. On the other hand, $\nrd(\mathbf p)\mid \nrd(\pi) = \pi^2$, so $\nrd(\mathbf p)$ must be either $\pi$ or $\pi^2$ up to multiplication by units of $\OK$.

    If $\nrd(\mathbf p)$ were associated to $\pi^2$, it would mean $\pi = \mathbf a\mathbf p$ for some $\mathbf a\in\mathcal H$ with $\nrd(\mathbf a)\in\OK^\times$, hence $\mathbf p\in\mathcal H\pi$. This would in turn imply $\mathbf q\in\mathcal H\mathbf p\subseteq \mathcal H\pi$, a contradiction.

    So $\nrd(\mathbf p)=\epsilon\pi$ for some $\epsilon\in\OK^\times$. Clearly, $\epsilon$ must be totally positive, so because $\OK$ has narrow class number $1$, we may express it as $\epsilon=\epsilon_0^2$. Then $\nrd(\epsilon_0^{-1}\mathbf p) = \pi$ proves the proposition.
\end{proof}

\begin{cor}
    \label{cor:nrd-universal}
    $(\mathcal H,\nrd)$ is a universal quadratic form over $\OK$.
\end{cor}
\begin{proof}
    Any totally positive element of $\OK$ can be expressed as a product of totally positive primes and totally positive units. Each of the primes is expressed by the Proposition. Further, totally positive units are squares, which are the reduced norms of elements of $\OK$, hence totally positive units are also expressed. The Corollary then follows by the multiplicativity of $\nrd$.
\end{proof}

\begin{prop}
\label{prp:irreducibles}
A~quaternion $\mathbf p\in\mathcal H$ is irreducible in $\mathcal H$, if and only if $\nrd(\mathbf p)$ is a (totally positive) prime element of $\OK$.
\end{prop}
\begin{proof}
If $\nrd(\mathbf p)$ is a prime, the irreducibility of $\mathbf p$ is immediate by considering the reduced norm, since an element is invertible in $\mathcal H$ if and only if its reduced norm is invertible in $\OK$.

On the other hand, suppose that $\mathbf p$ is irreducible. Its reduced norm then cannot be invertible, hence it is divisible by some prime $\pi\in\OK$. Just as in the proof of Proposition~\ref{prp:prime-norms}, we may then take $\mathcal H\mathbf p +\mathcal H\pi =: \mathcal H\mathbf q$ with $\nrd(\mathbf q)=\pi$. This will then imply that $\mathbf q$ is irreducible and $\mathbf p = \mathbf a\mathbf q$ for some $\mathbf a\in\mathcal H$, forcing $\mathbf a$ to be invertible. Hence $\epsilon:=\nrd(\mathbf a)$ is invertible in $\OK$, so $\nrd(\mathbf p) = \epsilon\pi$ is a prime in $\OK$.
\end{proof}

\subsection{Factorizing into irreducibles}
\label{sec:factorizations}

Here, we present several results on factorization into irreducible elements in $\mathcal H$. Because of non-commutativity, we cannot hope for an entirely unique factorization. We can however prove some weaker analogues which will later suffice for counting elements of a given reduced norm.

While we work over number fields, the proofs do not differ significantly from the situation in the Hurwitz order over $\kve$ (cf. \cite[Chapter 5]{conway-smith} or \cite[Section 11.5]{voight}).

\begin{defn}
\label{def:unit-migration}
For a given quaternion $\mathbf q\in\mathcal H$, let us call a finite sequence $\mathbf a_1,\dots,\mathbf a_n$ a \emph{factorization} of $\mathbf q$ if $\mathbf q = \mathbf a_1\cdots \mathbf a_n$. Let us further say that two factorizations $\mathbf a_1,\dots,\mathbf a_n$ and $\mathbf b_1,\dots,\mathbf b_m$ differ \emph{by unit migration}, if $n=m$ and there exist quaternions $\mathbf u_1,\dots,\mathbf u_{n-1}\in\mathcal H^1$ such that
\[
    \mathbf b_1 = \mathbf a_1\inv{\mathbf u_1}, \quad \mathbf b_2 = \mathbf u_1\mathbf a_2\inv{\mathbf u_2}, \quad\dots,\quad \mathbf b_{n-1} = \mathbf u_{n-2}\mathbf a_{n-1}\inv{\mathbf u_{n-1}}, \quad \mathbf b_n = \mathbf u_{n-1}\mathbf a_n.
\]
\end{defn}

We will always consider factorizations where reduced norms of the multiplicands follow a given factorization in $\OK$. Before we do this with a factorization into a product of irreducible elements, i.e. elements with prime reduced norms, let us start with a factorization whose reduced norms are prime powers, or, slightly more generally, pairwise coprime.

\begin{prop}
\label{prp:coprime-factorization}
Let $\mathbf q\in\mathcal H$ have $\nrd(\mathbf q) = \alpha\in\OK^+$ and let $\alpha = \alpha_1\cdots\alpha_n$ be a factorization into a product of pairwise coprime elements of $\OK^+$. Then:
\begin{enumerate}
    \item There is a factorization $\mathbf q = \mathbf a_1\cdots\mathbf a_n$ with $\nrd(\mathbf a_i)=\alpha_i$.
    \item This factorization is unique up to unit migration.
\end{enumerate}
\end{prop}
\begin{proof}
    To first prove existence, let us consider $\mathcal H\mathbf q + \mathcal H\alpha_n =: \mathcal H\mathbf a_n$. Computing in $\mathcal H/\alpha_n\mathcal H$, we see that $\nrd(\mathbf a_n)\in\alpha_n\OK$. On the other hand, $\nrd(\mathbf a_n)$ must divide both $\nrd(\alpha_n)=\alpha_n^2$ and $\nrd(\mathbf q)=\alpha$. Since $\alpha_1,\dots,\alpha_n$ are pairwise coprime, $\alpha_n$ is the greatest common divisor of $\alpha$ and $\alpha_n^2$. Hence $\nrd(\mathbf a_n)$ must be associated to $\alpha_n$ and so after possibly changing $\mathbf a_n$ by a suitable unit from $\OK^\times$, we may take $\nrd(\mathbf a_n)=\alpha_n$. The proof of existence of the desired factorization then proceeds by induction on $n$.

    For the uniqueness, suppose $\mathbf q = \mathbf a_1\cdots\mathbf a_n = \mathbf b_1\cdots\mathbf b_n$ are two such factorizations. The previous paragraph implies that $\mathcal H\mathbf a_n = \mathcal H\mathbf q + \mathcal H\alpha_n = \mathcal H\mathbf b_n$, hence $\mathbf b_n = \mathbf u \mathbf a_n$ for some $\mathbf u\in\mathcal H$ and considering the reduced norms shows $\nrd(\mathbf u)=1$. We may then denote $\mathbf u_{n-1}:={\mathbf u}$, simplify the equality $\mathbf a_1\cdots\mathbf a_n = \mathbf b_1\cdots\mathbf b_n$ to
    \[
        \mathbf a_1\cdots\mathbf a_{n-2}(\mathbf a_{n-1}\inv{\mathbf u_{n-1}}) = \mathbf b_1\cdots\mathbf b_{n-2}\mathbf b_{n-1}
    \]
    and proceed by induction.
\end{proof}

Now we factorize a quaternion $\mathbf q$ whose reduced norm is a power of some prime, say $\pi\in\OK^+$. It may happen that $\mathbf q\in\pi\mathcal H$. Since $\pi$ commutes with everything in $\mathcal H$, such a occurrence would introduce an overcount when we count the factorizations later, so it is useful to distinguish and separate this.

\begin{defn}
    Let us say a quaternion $\mathbf q\in\mathcal H$ is \emph{primitive} if $\mathbf q\notin\alpha\mathcal H$ for all non-units $\alpha\in\OK\setminus\OK^\times$.
\end{defn}

First we prepare a weaker, non-commutative analogue of Euclid's lemma:

\begin{lemma}
    \label{lem:pseudo-euclid}
    Let $\pi\in\OK^+$ be a prime, $\mathbf p\in\mathcal H$ a quaternion of reduced norm $\pi$ and $\mathbf a,\mathbf b\in\mathcal H$ arbitrary. Then $\mathbf a\mathbf p\mathbf b\in\pi\mathcal H$ implies $\mathbf a\mathbf p\in\pi\mathcal H$ or $\mathbf p\mathbf b\in\mathcal H$.
\end{lemma}
\begin{proof}
    It suffices to prove that if $\mathbf a\mathbf p\mathbf b\in\pi\mathcal H$ but $\mathbf p\mathbf b\notin\pi\mathcal H$, then $\mathbf a\mathbf p\in\pi\mathcal H$. Consider
    \[
        I~:= \lcol{\pi\mathcal H}{\mathbf p\mathbf b\mathcal H}\cap\mathcal H = \set{\mathbf x\in\mathcal H\mid \mathbf x\mathbf p\mathbf b\in\pi\mathcal H}.
    \]
    This is a left ideal of $\mathcal H$, so by $\mathcal H$ being a PID, we get $I = \mathcal H\mathbf q$ for some $\mathbf q$.

    Now on one hand, $1\notin I$ since $\mathbf p\mathbf b\notin\pi\mathcal H$, so $\nrd(\mathbf q)\notin\OK^\times$. On the other hand, clearly $\overline{\mathbf p}\in I$, so $\nrd(\mathbf q)\mid\nrd(\overline{\mathbf p}) = \pi$. This forces $\nrd(\mathbf q)$ to be associated to $\pi$, and hence $\overline{\mathbf p} = \mathbf u{\mathbf q}$ for some $\mathbf u\in\mathcal H$ with $\nrd(\mathbf u)\in\OK^\times$, meaning $\mathbf u\in\mathcal H^\times$. Thus $\mathbf q\in\mathcal H\overline{\mathbf p}$ and $\mathbf a\in I$ implies $\mathbf a\mathbf p\in \mathcal H \overline{\mathbf p}\mathbf p = \mathcal H\pi$ as we wished to prove.
\end{proof}

\begin{prop}
    \label{prp:prime-power-factorization}
    Let $\pi\in\OK^+$ be a prime and $\mathbf q\in\mathcal H$ a primitive quaternion of norm $\pi^r$. Then:
    \begin{enumerate}
        \item There is a factorization $\mathbf q = \mathbf p_1\cdots\mathbf p_r$ with $\nrd(\mathbf p_i) = \pi$.
        \item This factorization is unique up to unit migration.
    \end{enumerate}
\end{prop}
\begin{proof}
    For existence, $\mathbf q$ must factorize into some product $\mathbf p_1\cdots \mathbf p_n$ of irreducible quaternions. These each have prime reduced norms by Proposition~\ref{prp:irreducibles}, so due to $\nrd(\mathbf q) = \pi^r$, we must have $n=r$ and each $\nrd(\mathbf p_i)$ must be associated to $\pi$. Since all totally positive units in $\OK^\times$ are squares, we just multiply each $\mathbf p_i$ by a suitable unit to achieve $\nrd(\mathbf p_i)=\pi$, thus obtaining the desired factorization.

    For the uniqueness, let $\mathbf q = \mathbf p_1\cdots \mathbf p_r = \mathbf q_1\cdots\mathbf q_r$ be two such factorizations. Multiplying by $\overline{\mathbf q_r}$ from the right, we get
    \[
        (\mathbf p_1\cdots\mathbf p_{r-1})\mathbf p_r\overline{\mathbf q_r} = \mathbf q_1\cdots(\mathbf q_r\overline{\mathbf q_r}) \in \pi\mathcal H,
    \]
    so applying Lemma~\ref{lem:pseudo-euclid} with $\mathbf p:= \mathbf p_r$ and $\mathbf a:= \mathbf p_1\cdots\mathbf p_{r-1}$, $\mathbf b:=\overline{\mathbf q_r}$ we get that at least one of
    $(\mathbf p_1\cdots\mathbf p_{r-1})\mathbf p_r=\mathbf q$ and $\mathbf p_r\overline{\mathbf q_r}$ is divisible by $\pi$. But have presumed $\mathbf q$ to be primitive, hence we conclude that $\mathbf p\overline{\mathbf q_r} \in \mathcal H\pi = \mathcal H\mathbf q_r\overline{\mathbf q_r}$. Thus $\mathbf p_r \in\mathcal H\mathbf q_r$, considering reduced norms forces $\mathbf p_r = \mathbf u_{r-1} \mathbf q_r$ for some $\mathbf u_{r-1}\in\mathcal H^1$ and we proceed by induction as in the proof of Proposition~\ref{prp:coprime-factorization}.
\end{proof}

\begin{prop}
    \label{prp:primitivity-in-factorization}
    If $\mathbf q = \mathbf p_1\cdots \mathbf p_r$ is a factorization into a product of irreducible quaternions of reduced norm $\pi$, then $\mathbf q\in\pi\mathcal H$ if and only if $\mathbf p_i \in \mathcal H^1\overline{\mathbf p_{i+1}}$ for some $i=1,\dots,r-1$.
\end{prop}
\begin{proof}
    Sufficiency is obvious, so let us prove necessity. Let us choose an inclusion-minimal interval $\set{i,i+1,\dots,i+\ell-1}\subseteq\set{1,2,\dots,r}$ such that $\mathbf p_i\cdots \mathbf p_{i+\ell-1}\in\pi\mathcal H$. By considering reduced norms, clearly $\ell\geq2$. On the other hand, if it were the case that $\ell\geq 3$, then we could use Lemma~\ref{lem:pseudo-euclid} with $\mathbf a := \mathbf p_i$, $\mathbf p:=\mathbf p_{i+1}$, $\mathbf b:= \mathbf p_{i+2}\cdots\mathbf p_{i+\ell-1}$ to find a strictly smaller interval which still gives a product divisible by $\pi$.

    This forces $\ell=2$, i.e. $\mathbf p_i\mathbf p_{i+1}\in\pi\mathcal H = \mathcal H \overline{\mathbf p_{i+1}}\mathbf p_{i+1}$ for some $i$, hence $\mathbf p_i\in\mathcal H\overline{\mathbf p_{i+1}}$. Considering reduced norms then forces $\mathbf p_i$ to be a left multiple of $\overline{\mathbf p_{i+1}}$ by a quaternion of reduced norm $1$.
\end{proof}

\subsection{Counting irreducible elements}
\label{sec:irredcount}

\begin{lemma}
\label{lem:goodquotient}
    For any prime ideal $\mathfrak p$ of $\OK$ such that $\mathfrak p\nmid \discrd\mathcal H$, the quotient ring $\mathcal H/\mathfrak p\mathcal H$ is isomorphic to $\Mat_2(\OK/\mathfrak p)$.
\end{lemma}
\begin{proof}
    $\mathcal H/\mathfrak p\mathcal H$ is unchanged by taking the completion at $\mathfrak p$, and then the Lemma follows from \cite[23.2.3]{voight}.
\end{proof}
\begin{lemma}
\label{lem:badquotient}
For a prime ideal $\mathfrak p$ of $\OK$ such that $\mathfrak p \mid \discrd\mathcal H$, the quotient ring $\mathcal H/\mathfrak p\mathcal H$ has a unique non-trivial left ideal.
\end{lemma}
\begin{proof}
    Again, it suffices to argue for $\mathcal H_{\mathfrak p}/\mathfrak p\mathcal H_{\mathfrak p}$. Then the Lemma is essentially contained in \cite[13.3.7 and 13.3.10]{voight}: all ideals of $\mathcal H_p$ are two-sided and powers of the unique maximal ideal $P$ which satisfies $P^2=\mathfrak p\mathcal H_p$. Thus only $P$ remains as a non-trivial ideal in $\mathcal H_{\mathfrak p}/\mathfrak p\mathcal H_{\mathfrak p}$.
\end{proof}

These two lemmata hold even without our conditions on the class number of $\mathcal H$ (and the narrow class number of $K$). With these conditions however, we will leverage them to count irreducible quaternions of prime reduced norms.

\begin{prop}
\label{prp:irredcounts}
Let $\pi$ be a totally positive prime element in $\OK$.
\begin{enumerate}
\item If $\pi\OK\nmid \discrd\mathcal H$, there are $\#\mathcal H^1\cdot(1+\Nm_{K/\kve}(\pi))$ quaternions of reduced norm $\pi$ in $\mathcal H$.
\item If $\pi\OK\mid \discrd\mathcal H$, there are $\#\mathcal H^1$ quaternions of reduced norm $\pi$ in $\mathcal H$.
\end{enumerate}
\end{prop}
\begin{proof}
First, let us argue that it suffices to count left ideals of reduced norm $\pi\OK$. Firstly, any quaternion of reduced norm $\pi$ generates such a left ideal. On the other hand, any such ideal is principal and we may multiply any generator by a suitable unit (utilizing the narrow class number being $1$ again) to get a generator of reduced norm $\pi$. Any two such generators differ only by multiplication by a unit of reduced norm $1$, hence any ideal of reduced norm $\pi\OK$ yields exactly $\#\mathcal H^1$ quaternions of reduced norm $\pi$.

Then, a left ideal of reduced norm $\pi$ must be contained in $\pi\mathcal H$, so we may count these ideals in $\mathcal H/\pi\mathcal H$. Now (ii) is immediate, since by Lemma~\ref{lem:badquotient}, the quotient ring has a unique non-trivial (left) ideal.

For (i), let us denote the finite field $k:=\OK/\pi\OK$ and $q:=\# k = \Nm_{K/\kve}(\pi)$. Thanks to Lemma~\ref{lem:goodquotient}, we may count the non-zero left ideals of reduced norm $0\in k$ in $\Mat_2(k)$. As we noted in Section~\ref{chap:prelims}, this is a (left) principal ideal ring and there is a bijection
\[
    \set{\text{non-trivial left ideals of $\Mat_2(k)$}} \longleftrightarrow \set{\text{one-dimensional subspaces of $k^2$}} = \pe^1 k,
\]
where the last set is the projective line over $k$, which has $q+1$ elements, so we are done.
\end{proof}

\subsection{Counting factorizations}
\label{sec:counting-factorizations}

All that remains now is to use the results of Subsection~\ref{sec:factorizations} to count quaternions of a given reduced norm. As it turns out, expressions akin to sums of divisors appear in these calculations, so to simplify them, let us introduce the following notation: if $\mathfrak D$ is an ideal of $\OK$ which factorizes into a product of prime ideals $\mathfrak p_1\cdots \mathfrak p_r$, let us denote
\[
    \sigma_{\mathfrak D}(\alpha) := \sum_{\mathfrak p_1,\dots,\mathfrak p_r \nmid \delta\OK \mid\alpha\OK} \Nm(\delta\OK)
\]
where $\alpha\in\OK$ and the sum runs over all the (principal) ideals of $\OK$ which divide $\alpha\OK$ but are not divided by any $\mathfrak p_i$. The condition $\mathfrak p_1,\dots,\mathfrak p_r \nmid \delta\OK$ can also be equivalently stated as $\mathfrak D+\delta\OK=\OK$, so informally speaking, this is a \uv{sum of divisors of $\alpha$ coprime to $\mathfrak D$}.

By the Chinese remainder theorem, if $\alpha,\beta\in\OK$ are coprime, ideal divisors of $\alpha\beta\OK$ are exactly products of a divisor of $\alpha\OK$ and a divisor of $\beta\OK$, whence it follows that $\sigma_{\mathfrak D}(\alpha\beta) = \sigma_{\mathfrak D}(\alpha)\sigma_{\mathfrak D}(\beta)$. We say that $\sigma_{\mathfrak D}$ is \emph{multiplicative}. Lastly, we may notice that $\sigma_{\mathfrak D}(\epsilon\alpha)=\sigma_{\mathfrak D}(\alpha)$ for any unit $\epsilon\in\OK^\times$, since $\alpha\OK=\epsilon\alpha\OK$.

\begin{defn}
For $\alpha\in \OK^+$ and an order $\mathcal H$, let us denote
\[
    r_{\mathcal H}(\alpha) := \#\set{\mathbf q\in\mathcal H\mid \nrd(\mathbf q)=\alpha}.
\]
\end{defn}

\begin{thm}
\label{thrm:kind-max-formula}
Let $\mathcal H$ be a maximal order of class number $1$ in a definite quaternion algebra $\mathcal A$ of discriminant $\mathfrak D$ over a totally real number field $K$ of narrow class number $1$. Then for every $\alpha\in\OK^+$, there exist exactly
\[
    r_{\mathcal H}(\alpha)=\#\mathcal H^1\cdot \sigma_{\mathfrak D}(\alpha)
\]
quaternions in $\mathcal H$ of reduced norm $\alpha$.
\end{thm}
\begin{proof}
First, notice that if we multiply $\alpha$ by a totally positive unit $\epsilon$, such a totally positive unit is a square of some $\epsilon_0\in\OK^\times$. Hence this change can be realized on any quaternion $\mathbf q$ of reduced norm $\alpha$ by considering $\epsilon_0\mathbf q$ instead. Thus, both $r_{\mathcal H}(\alpha)$ and the quantity $\#\mathcal H^1\cdot\sigma_{\mathfrak D}(\alpha)$ are unchanged when replacing $\alpha$ by $\epsilon\alpha$ for some $\epsilon\in\OK^{\times,+}$, so we may change $\alpha$ by totally positive units without loss of generality.

With this in mind, we have $\sigma_{\mathfrak D}(1)=1$, so the theorem clearly holds when $\alpha$ is a unit. By Proposition~\ref{prp:irredcounts}, it also holds for a prime element $\alpha$.

Next, let $\alpha=\pi^k$ be a prime power. Suppose first that $\pi\OK\nmid\discrd \mathcal H$. Let us denote $q:=\Nm_{K/\kve}(\pi)$. We claim that there are $\#\mathcal H^1\cdot (q^k+q^{k-1})$ \emph{primitive} quaternions of reduced norm $\alpha$. For this, we use Propositions \ref{prp:prime-power-factorization} and \ref{prp:primitivity-in-factorization}. This tells us that we should count factorizations using irreducible quaternions of reduced norm $\pi$ where conjugates (up to multiplication by units) do not appear in adjacent positions, all the while managing the overcounting stemming from unit migration.

Choosing the irreducible $\mathbf p_i$, $i=1,\dots,k$ sequentially, the first one offers $\#\mathcal H^1\cdot(1+q)$ options, but for all subsequent positions, the last chosen factor forbids $\#\mathcal H^1$ of these. So we have $\#\mathcal H^1\cdot(1+q) \cdot \zav{\#\mathcal H^1\cdot q}^{k-1}$ suitable factorizations. From any such factorization, we perform $(\# \mathcal H^1)^{k-1}$ unit migrations and clearly, each results in a different factorization. So we have overcounted by a factor of $(\# \mathcal H^1)^{k-1}$, which yields
\[
    \frac{(\# \mathcal H^1)^{k}(q+1)q^{k-1}}{(\# \mathcal H^1)^{k-1}} = \#\mathcal H^1\cdot q^{k-1}(q+1)
\]
primitive quaternions of reduced norm $m$.

Dropping the primitivity condition, we need to additionally count quaternions divisible by $\pi$, $\pi^2$, etc., leading to a total count of
\[
r_{\mathcal H}(\alpha)=\#\mathcal H^1\cdot \zav{ (q^k+q^{k-1})+ (q^{k-2}+q^{k-3})+\cdots+1} = \#\mathcal H^1\cdot\sigma_{\mathfrak D}(\pi^k).
\]

If on the other hand $\pi\OK\mid \discrd \mathcal H$, let us prove that there only $\#\mathcal H^1$ quaternions of reduced norm $\pi^k$ by induction on $k\geq0$. We already have this claim for $k=0$ and $k=1$. Since there are only $\#\mathcal H^1$ quaternions of reduced norm $\pi$, if we let $\mathbf q$ be any of them, then $\mathcal H^1\overline{\mathbf q}$ must recover all $\#\mathcal H^1$ of them again. Hence in view of Proposition~\ref{prp:primitivity-in-factorization}, there can be no primitive quaternions of reduced norm $\pi^k$ for $k\geq2$. So their count is the same as for $\pi^{k-2}$ and our claim follows by induction.

Finally, we consider an arbitrary $\alpha\in\OK^+$ and we fix a factorization $\alpha=\pi_1^{e_1}\cdots \pi_k^{e_k}$ into powers of totally positive primes (as before, multiplying by a totally positive units does not affect the count). By Proposition~\ref{prp:coprime-factorization}, we multiply the quantities of quaternions of reduced norms $\pi_i^{e_i}$ and then correct for the overcount stemming from unit migrations by a factor of $(\#\mathcal H^1)^{k-1}$. Since we have already counted quaternions with prime power reduced norms, we thus obtain
\begin{multline*}
    r_{\mathcal H}(\alpha)=\frac{r_{\mathcal H}(\pi_1^{e_1})\cdots r_{\mathcal H}(\pi_k^{e_k})}{(\#\mathcal H^1)^{k-1}}=\frac{(\#\mathcal H^1 \sigma_{\mathfrak D}(\pi_1^{e_1}))\cdots(\#\mathcal H^1 \sigma_{\mathfrak D}(\pi_k^{e_k}))}{(\#\mathcal H^1)^{k-1}} =\\= \#\mathcal H^1\cdot (\sigma_{\mathfrak D}(\pi_1^{e_1})\cdots \sigma_{\mathfrak D}(\pi_k^{e_k})) = \#\mathcal H^1\cdot \sigma_{\mathfrak D}(\pi_1^{e_1}\cdots\pi_k^{e_k}) = \#\mathcal H^1 \cdot \sigma_{\mathfrak D}(\alpha)
\end{multline*}
quaternions of reduced norm $\alpha$, finishing the proof.
\end{proof}

\section{Intersecting unit orbits with suborders}
\label{chap:orbits}

Mimicking Hurwitz's proofs of Lagrange's and Jacobi's four-square theorem, we will concern ourselves with the following situation in this section: we are given an $\OK$-order $\mathcal H$ in a definite quaternion algebra $\mathcal A$ over a number field $K$ (mostly, $\mathcal H$ will be a maximal order of class number $1$). We let the (finite) group $\mathcal H^1$ of units of reduced norm $1$ act on $\mathcal H$ by multiplication from the left. Then we wish to examine suborders $\mathcal G\subseteq \mathcal H$ which intersect every orbit of this action -- in the classical case of Hurwitz, this corresponds to the Lipschitz order $\mathcal G$ intersecting all orbits in the Hurwitz order $\mathcal H$ with respect to the (left) action of the group $\mathcal H^1$.

\begin{defn}
\label{def:perceptivity}
Let $\mathcal G\subseteq\mathcal H$ be orders in a definite quaternion algebra. We say that $\mathcal G$ is \emph{(left) $\mathcal H$-perceptive} (or that the pair $\mathcal G$, $\mathcal H$ is \emph{(left) perceptive}, or that $\mathcal G$ is \emph{(left) perceptive in $\mathcal H$}), if every orbit $\mathcal H^1\mathbf q$ for $\mathbf q\in\mathcal H$ has a non-empty intersection with $\mathcal G$.
\end{defn}

We may notice that since the standard involution preserves orders and maps left orbits to right orbits and vice versa,
the notion of a right perceptive suborder would be equivalent to that of a left perceptive one, so we drop the directional distinction.

As in the Hurwitz case, once $\mathcal H$ is definite, maximal and of class number $1$ and $\mathcal G$ is $\mathcal H$-perceptive, the universality of the quadratic form $(\mathcal H,\nrd)$ may immediately be carried over to obtain universality of $(\mathcal G,\nrd)$. Further, with an examination of the exacts sizes of the intersections $\mathcal H^1\mathbf q\cap \mathcal G$, one may deduce precise formulas for the number of quaternions of a given norm in $\mathcal G$ from those in $\mathcal H$.

Throughout this section, \uv{orders} are always $\OK$-orders in a definite quaternion algebra.

\subsection{Chains of orders}

\begin{prop}
\label{prp:composing-perceptivity}
Let $\mathcal F\subseteq \mathcal G\subseteq\mathcal H$ be orders.
\begin{enumerate}
    \item If both $\mathcal G$, $\mathcal H$ and $\mathcal F$, $\mathcal G$ are perceptive pairs, then so is $\mathcal F$, $\mathcal H$.
    \item If $\mathcal F$ is $\mathcal H$-perceptive, then $\mathcal G$ is $\mathcal H$-perceptive.
    \item If $\mathcal F$ is $\mathcal H$-perceptive and additionally $[\mathcal H:\mathcal G]_{\OK}$, $[\mathcal G:\mathcal F]_{\OK}$ are comaximal ideals of $\OK$, then $\mathcal F$ is $\mathcal G$-perceptive.
\end{enumerate}
\end{prop}
\begin{proof}
\begin{enumerate}
\item Starting with a $\mathbf q\in \mathcal H$, we first find an $\mathbf r\in \mathcal H^1\mathbf q\cap \mathcal G$ and then an $\mathbf s\in \mathcal G^1\mathbf r \cap\mathcal F \subseteq \mathcal H^1\mathbf q\cap\mathcal F$.
\item Clearly $\mathcal H^1\mathbf q\cap\mathcal F \subseteq \mathcal H^1\mathbf q\cap\mathcal G$.

\item
Let us denote $\mathfrak a:=[\mathcal G:\mathcal F]_{\OK}$ and $\mathfrak b:=[\mathcal H:\mathcal G]_{\OK}$.
Suppose for the sake of contradiction that $\mathcal G^1\mathbf q\cap\mathcal F$ is empty for some $\mathbf q\in\mathcal G$.

\begin{claim}[A]
We can choose this $\mathbf q$ with the added condition that $\nrd(\mathbf q)\OK$ is comaximal to $\mathfrak b$.
\end{claim}

Since $\mathfrak a$, $\mathfrak b$ are comaximal, we may find an $\alpha\in\OK$ satisfying $\alpha\equiv0\pmod{\mathfrak a}$, $\alpha\equiv1\pmod{\mathfrak b}$ by the Chinese remainder theorem. Let us then start our search with an arbitrary $\mathbf r\in\mathcal G$ such that $\mathcal G^1\mathbf r\cap\mathcal F$ is empty, and then look for a suitable $\mathbf q$ in the form $t\alpha + (\alpha+1)\mathbf r$ for some $t\in\OK$ yet to be determined. Since $\mathfrak a \mathcal G\subseteq \mathcal F$ and $\alpha\in\mathfrak a$, we have $\alpha\mathcal G\subseteq\mathcal F$. Thus because $\mathbf q = \mathbf r+\alpha(t+\mathbf r)$, any such $\mathbf q$ will still satisfy $\mathcal G^1\mathbf q\cap\mathcal F=\emptyset$. So we only need to deal with the desired condition on $\nrd(\mathbf q)$.

Let $\mathfrak p_1,\dots,\mathfrak p_r$ be the prime ideals dividing $\mathfrak b$. Then it suffices to determine $t$ modulo each $\mathfrak p_i$ so that $\nrd(\mathbf q)\nequiv0\pmod{\mathfrak p_i}$ and then compose this data using the Chinese remainder theorem. So let us investigate $\nrd(\mathbf q)\pmod{\mathfrak p_i}$. We have
\begin{align}
\nonumber
    \nrd(t\alpha+(\alpha+1)\mathbf r) &= \alpha^2 t^2+\alpha(\alpha+1)\trd(\mathbf r)t + \nrd(\mathbf r) \equiv\\
    \tag{$\ast$}&\equiv t^2+2\trd(\mathbf r)t+\nrd(\mathbf r) \pmod{\mathfrak p_i}
\end{align}
since $\alpha\equiv1\pmod{\mathfrak b}$.
Thus this is a polynomial of degree $2$ over the finite field $\OK/\mathfrak p_i$. The only way for it to never attain a non-zero value is for all elements of $\OK/\mathfrak p_i$ to be its roots. It can have no more than two roots, so this may only happen if $\OK/\mathfrak p_i$ is the field with two elements. Hence in all other cases we find a suitable $t\pmod{\mathfrak p_i}$.

In the case when $\OK/\mathfrak p_i=\set{0,1}$ is the two-element field, ($\ast$) simplifies to
\[
    \nrd(\mathbf q) \equiv t^2 + \nrd(\mathbf r) \pmod{\mathfrak p_i}.
\]
Since $t^2\equiv t$ holds in the two-element field, we then choose $t\equiv \nrd(\mathbf r)+1$ and achieve $\nrd(\mathbf q)\nequiv 0\pmod{\mathfrak p_i}$ in this case as well.

Choosing a $t\in\OK$ that satisfies each of the chosen residues $t\pmod{\mathfrak p_i}$ via the Chinese remainder theorem, we achieve $\nrd(\mathbf q)\OK$ comaximal to $\mathfrak b$. This proves Claim (A).

Now with this $\mathbf q$, we know that there is some $\mathbf u\in\mathcal H^1$ such that $\mathbf u\mathbf q\in\mathcal F$; this ensures $\mathbf u\in\mathcal H^1\setminus\mathcal G^1$. Let us denote the order $\mathcal L:=\mathcal F+\mathcal F\mathbf u$ and the lattice $N:= \mathbf q\mathcal F$, then we observe that $\mathcal L N \subseteq \mathcal G$

\begin{claim}[B]
The ideal $[\mathcal L:\mathcal F]_{\OK}$ is divisible by some prime ideal $\mathfrak p$ of $\OK$ comaximal to $\mathfrak a$.
\end{claim}

Notice that $\mathbf u \in \mathcal L$ but $\mathbf u\notin\mathcal G$, hence $\mathcal L\nsubseteq\mathcal G$. By the local-global dictionary, this means there is some prime $\mathfrak p$ such that $\mathcal L_{\mathfrak p}\nsubseteq\mathcal G_{\mathfrak p}$. On one hand, it must be a prime $\mathfrak p\mid[\mathcal L:\mathcal F]_{\OK}$, since otherwise we would get $\mathcal L_{\mathfrak p}=\mathcal F_{\mathfrak p} \subseteq\mathcal G_{\mathfrak p}$. On the other hand, it must also be a prime $\mathfrak p\mid\mathfrak b = [\mathcal H:\mathcal G]_{\OK}$, since otherwise $\mathcal G_{\mathfrak p}=\mathcal H_{\mathfrak p} \supseteq\mathcal L_{\mathfrak p}$. But then $\mathfrak p$ is comaximal to $\mathfrak a$ because $\mathfrak b$ is, proving Claim (B).

Using this $\mathfrak p$, let us now pass to the local situation over $\mathcal O_{K,\mathfrak p}$. We obtain $\mathcal F_{\mathfrak p}\subseteq \mathcal G_{\mathfrak p}$ but also
\[
    [\mathcal G_{\mathfrak p}:\mathcal F_{\mathfrak p}]_{\mathcal O_{K,\mathfrak p}} = ([\mathcal G:\mathcal F]_{\OK})_{\mathfrak p} = \mathcal O_{K,\mathfrak p},
\]
hence $\mathcal F_{\mathfrak p} = \mathcal G_{\mathfrak p}$. Similarly, since $N_{\mathfrak p}\subseteq \mathcal G_{\mathfrak p}= \mathcal F_{\mathfrak p}$ and
\[
    [\mathcal F:N]_{\OK} = \nrd(\mathbf q)^2\OK,
\]
which is comaximal to $\mathfrak b$ by Claim (A), we get
\[
[\mathcal F_{\mathfrak p}:N_{\mathfrak p}]_{\mathcal O_{K,\mathfrak p}}  = (\nrd(\mathbf q)^2\OK)_{\mathfrak p} = \mathcal O_{K,\mathfrak p}
\]
and thus $N_{\mathfrak p} = \mathcal F_{\mathfrak p}$.

On the other hand, $\mathfrak p\mid [\mathcal L:\mathcal F]_{\OK}$ means $[\mathcal L_{\mathfrak p}:\mathcal F_{\mathfrak p}]_{\mathcal O_{K,\mathfrak p}}$ remains non-trivial and thus $\mathcal F_{\mathfrak p}\subsetneq \mathcal L_{\mathfrak p}$. But the inclusion $\mathcal LN\subseteq \mathcal G$ now turns into
\[
    \mathcal L_{\mathfrak p} N_{\mathfrak p}\subseteq \mathcal G_{\mathfrak p} = \mathcal F_{\mathfrak p},
\]
which is false due to $N_{\mathfrak p}=\mathcal F_{\mathfrak p}\ni 1$ and $\mathcal L_{\mathfrak p}\cdot 1\nsubseteq \mathcal F_{\mathfrak p}$.
Thus we have reached a contradiction, proving the proposition.
\qedhere
\end{enumerate}
\end{proof}

Once we've built up some tools to aid in the verification in the next subsection, we shall illustrate in Example~\ref{ex:perceptive-composing} that without the condition that $[\mathcal H:\mathcal G]_{\OK}$, $[\mathcal G:\mathcal F]_{\OK}$ be comaximal, the conclusion of part (iii) need not hold.

For now, the last proposition suggests it may be beneficial, when studying a suborder $\mathcal F\subseteq \mathcal H$, to insert intermediate orders such that the two resulting indices are comaximal. In its most severe form, this looks as follows:

\begin{prop}
\label{prp:prime-power-chain}
Let orders $\mathcal F\subseteq \mathcal H$ be given and let a factorization of $[\mathcal H:\mathcal F]_{\OK}$ into powers of distinct prime ideals $\mathfrak p_1^{e_1}\cdots \mathfrak p_r^{e_r}$ be given. Then there exists a chain of orders
\[
    \mathcal F =: \mathcal G_0\subsetneq \cdots \subsetneq\mathcal G_r := \mathcal H
\]
such that $[\mathcal G_i:\mathcal G_{i-1}]_{\OK}=\mathfrak p_i^{e_i}$ for each $i=1,\dots,r$.
\end{prop}
\begin{proof}
    Let us use the local-global dictionary to construct each $\mathcal G_i$ by prescribing
    \[
        \mathcal G_{i,\mathfrak p} := \begin{cases}
            \mathcal H_{\mathfrak p}, & \text{if $\mathfrak p = \mathfrak p_j$ for some $j\leq i$,}\\
            \mathcal F_{\mathfrak p}, & \text{if $\mathfrak p = \mathfrak p_j$ for some $j> i$,}\\
            \mathcal F_{\mathfrak p}=\mathcal H_{\mathfrak p}, & \text{if $\mathfrak p\notin\set{\mathfrak p_1,\dots,\mathfrak p_r}$;}
        \end{cases}
    \]
    the equality $\mathcal F_{\mathfrak p}=\mathcal H_{\mathfrak p}$ for $\mathfrak p\notin\set{\mathfrak p_1,\dots,\mathfrak p_r}$ is due to $\mathcal F\subseteq \mathcal H$ and $[\mathcal H_{\mathfrak p}:\mathcal F_{\mathfrak p}]_{\mathcal O_{K,\mathfrak p}} = ([\mathcal H:\mathcal F]_{\OK})_{\mathfrak p} = \mathcal O_{K,\mathfrak p}$.
    By construction, $\mathcal G_i$ is an order at all $\mathfrak p$, hence it is an order globally as well.

    For any $i=1,\dots,r$, we also see that $\mathcal G_i$ and $\mathcal G_{i-1}$ only differ at $\mathfrak p_i$, where
    \[
        \mathcal G_{i-1,\mathfrak p_i} = \mathcal F_{\mathfrak p_i}\subseteq \mathcal H_{\mathfrak p_i} = \mathcal G_{i,\mathfrak p_i},
    \]
    so $\mathcal G_{i-1}\subseteq \mathcal G_i$. It also follows that $([\mathcal G_i:\mathcal G_{i-1}]_{\OK})_{\mathfrak p} = \mathcal O_{K,\mathfrak p}$ for all $\mathfrak p\neq\mathfrak p_i$ and $([\mathcal G_i:\mathcal G_{i-1}]_{\OK})_{\mathfrak p_i} = ([\mathcal H:\mathcal F]_{\OK})_{\mathfrak p_i} = (\mathfrak p_i\mathcal O_{K,\mathfrak p_i})^{e_i}$ implies $[\mathcal G_i:\mathcal G_{i-1}]_{\OK} = \mathfrak p_i^{e_i}$.
\end{proof}

Naturally, we may ask for further refinements of a chain orders constructed like this. Assuming for notational ease that $[\mathcal H:\mathcal F]_{\OK}$ was a prime power already, whenever there exists an intermediate order $\mathcal G$, we may insert it to get $\mathcal F\subsetneq\mathcal G\subsetneq\mathcal H$. Then $[\mathcal H:\mathcal G]_{\OK}$ and $[\mathcal G:\mathcal F]_{\OK}$ will both be strictly smaller powers of the same prime, so this process of inserting intermediate orders must stop eventually. We now characterize maximal suborders of $\mathcal H$.

\begin{prop}
\label{prp:smallest-steps}
Let $\mathcal F\subsetneq \mathcal H$ be orders with $[\mathcal H:\mathcal F]_{\OK}=\mathfrak p^e$ for some prime $\mathfrak p$ of $\OK$. No order $\mathcal G$ with $\mathcal F\subsetneq\mathcal G\subsetneq\mathcal H$ exists if and only if either:
\begin{enumerate}
\item $e=1$ and thus $\mathcal H/\mathcal F\simeq\OK/\mathfrak p$ as $\OK$-modules; or
\item $e=2$, $\mathcal H/\mathcal F\simeq(\OK/\mathfrak p)^2$ as $\OK$-modules and within the $\OK/\mathfrak p$-algebra $\mathcal H/\mathfrak p\mathcal H$, the subalgebra $\mathcal F/\mathfrak p\mathcal H$ is a quadratic field extension of $\OK/\mathfrak p$.
\end{enumerate}
\end{prop}

\begin{proof}
Let us denote $k:=\OK/\mathfrak p$.
First, we verify that in both (i) and (ii), no intermediate order $\mathcal G$ exists. For (i), this is obvious. For (ii), we descend to $\mathcal H/\mathfrak p\mathcal H$. Any intermediate $\mathcal G$ would correspond to a three-dimensional subalgebra $\mathcal G/\mathfrak p\mathcal H$ containing $\mathcal F/\mathfrak p\mathcal H$. This would also be a vector space over $\mathcal F/\mathfrak p\mathcal H$, which is two-dimensional over $k$. Hence $\mathcal G/\mathfrak p\mathcal H$ would need to have an even dimension over $k$, a contradiction.

To prove that one of (i) or (ii) is necessary, we first note that by \cite[Corollary 1.11]{brzezinski83}, $[\mathcal H:\mathcal F]_{\OK}$ is either $\mathfrak p$ or $\mathfrak p^2$ and $\mathfrak p\mathcal F\subseteq\mathcal H$. Thus it only remains to prove that in the latter case, $\mathcal F/\mathfrak p\mathcal H$ is a field.

Let us work in $\mathcal H/\mathfrak p\mathcal H$ again. By its dimension, $\mathcal F/\mathfrak p\mathcal H$ is spanned by $1$ and some $\mathbf a$. Since $\mathbf a$ commutes with itself and $1$ commutes with everything, $\mathcal F/\mathfrak p\mathcal H$ is commutative. Let us suppose for the sake of contradiction it is not a field, then it contains a non-invertible element. Without loss of generality, we may assume $\mathbf a$ is this non-invertible element; this means $\nrd(\mathbf a)=0\in k$. Let us consider the multiplication map
\begin{align*}
    \mu:\mathcal H/\mathfrak p\mathcal H&\to \mathcal H/\mathfrak p\mathcal H,\\
    \mathbf q&\mapsto \mathbf q\mathbf a.
\end{align*}
Note that $\dim \ker\mu+\dim\im\mu=4$ (over $k$) and that $1$ does not belong to either $\ker\mu$ or $\im\mu$. Hence at least one of $\ker\mu\setminus(\mathcal F/\mathfrak p\mathcal H)$ and $\im\mu\setminus(\mathcal F/\mathfrak p\mathcal H)$ is non-empty.

If we can choose a $\mathbf q \in \ker\mu\setminus(\mathcal F/\mathfrak p\mathcal H)$, let us consider the subalgebra (by Lemma \ref{lem:switch-the-order})
\[
    k+k\mathbf q+k\mathbf a+k\mathbf q\mathbf a.
\]
Since $\mathbf q\mathbf a=0$, this is three-dimensional, so its preimage in $\mathcal H$ is an intermediate order.

If on the other hand we can choose a $\mathbf q\mathbf a \in \im\mu\setminus(\mathcal F/\mathfrak p\mathcal H)$, let us consider the subalgebra
\[
    k+k\mathbf q\mathbf a+k\mathbf a+k\mathbf q\mathbf a^2.
\]
Observe that $\mathbf a^2 = \trd(\mathbf a)\mathbf a$ due to $\nrd(\mathbf a)=0$, hence $k\mathbf q\mathbf a^2\subseteq k\mathbf q\mathbf a$, making the subalgebra three-dimensional, so its preimage in $\mathcal H$ is an intermediate order.

Overall, we have found a contradiction, so $\mathcal F/\mathfrak p\mathcal H$ must have been a field.
\end{proof}

\subsection{Module conditions for perceptivity}
\label{sec:module-conditions}

Orders are very specific lattices, so for orders $\mathcal G\subseteq\mathcal H$, we can look at $\mathcal H/\mathcal G$ as a quotient of $\OK$-modules. It must be a torsion module, so by the structure theorem for finitely generated modules over Dedekind domains (combined with the Chinese remainder theorem), it decomposes as a direct sum of several cyclic modules isomorphic to some $\OK/\mathfrak p^a$, $\mathfrak p$ being a prime ideal of $\OK$. In this subsection, we examine some implications this decomposition might have for whether $\mathcal G$ is $\mathcal H$-perceptive. The most fruitful of these will, in specific situations, yield equivalent conditions for this perceptivity using cardinalities of $\mathcal G^1$ and $\mathcal H^1$.

In view of Propositions~\ref{prp:composing-perceptivity} and \ref{prp:prime-power-chain}, we will restrict our attention to the case when the index of the two orders is a prime power. Hence, throughout this entire subsection, let $\mathcal G\subseteq\mathcal H$ be orders in a definite quaternion algebra $\mathcal A$ such that $[\mathcal H:\mathcal G]_{\OK}=\mathfrak p^e$.

\begin{prop}
\label{prp:at-most-two-dims}
Suppose that $\mathcal H/\mathcal G\simeq \OK/\mathfrak p^{a_1}\times\cdots\times \OK/\mathfrak p^{a_r}$ with $a_1,\dots,a_r\geq1$. If $\mathcal G$ is $\mathcal H$-perceptive, then $r\leq 2$.
\end{prop}
\begin{proof}
    Consider $\mathcal F := \mathcal G+\mathfrak p\mathcal H$. This is an order and it satisfies $\mathcal G\subseteq \mathcal F\subseteq H$ and $\mathcal H/\mathcal F\simeq (\OK/\mathfrak p)^r$. By Proposition~\ref{prp:composing-perceptivity}(ii), we see that $\mathcal F$ is $\mathcal H$-perceptive. We may quotient by $\mathfrak p\mathcal H$ to view $A:=\mathcal F/\mathfrak p\mathcal H$ as a subalgebra of the four-dimensional algebra $B:=\mathcal H/\mathfrak p\mathcal H$ over $k:=\OK/\mathfrak p$. Then $A$ is $(4-r)$-dimensional, which immediately forces $r\leq 3$ since at least $1\in A$.

    Let us suppose for the sake of contradiction that $r=3$, then $A=k$ is one-dimensional. Yet, because $\mathcal F$ is $\mathcal H$-perceptive, any element of $B$ may be multiplied by an element of $B^{\times}$ (in fact more strongly, by a residue class of some $\mathbf u\in\mathcal H^1$) with the result falling into $A$. This implies that every line (one-dimensional subspace) in $B$ contains an invertible element. Invertible elements must have non-zero reduced norms, so this would imply that the quadratic form $(B,\nrd)$ is anisotropic, contradicting Lemma~\ref{lem:finite-field-isotropy}. Hence it must have been the case that $r\leq 2$.
\end{proof}

Now we know that if the pair of orders is to be perceptive, it suffices to consider the case when $\mathcal H/\mathcal G\simeq (\OK/\mathfrak p ^a) \times (\OK/\mathfrak p^b)$. Our strategy will be to quotient the situation by such an $\OK$-submodule that the emptiness or non-emptiness of $\mathcal H^1\mathbf q\cap \mathcal G$ only depends on the class of $\mathbf q$ in the quotient. It turns out that
\[
    \rcol{\mathcal G}{\mathcal H} = \set{\mathbf x\in\mathcal A\mid \mathcal H\mathbf x\subseteq\mathcal G} = \set{\mathbf x\in\mathcal G\mid \mathcal H\mathbf x\subseteq\mathcal G},
\]
the so-called \emph{(right) conductor of $\mathcal H$ into $\mathcal G$}, is the suitable choice for such a submodule. We may notice that this is a right ideal of $\mathcal G$ and a left ideal of $\mathcal H$, since for $\mathbf x\in\rcol{\mathcal G}{\mathcal H}$ we obtain
$\mathcal H\mathbf x\mathcal G\subseteq\mathcal G\mathcal G\subseteq\mathcal G$, hence $\mathbf x\mathcal G\subseteq\rcol{\mathcal G}{\mathcal H}$,
and $\mathcal H\mathcal H\mathbf x\subseteq\mathcal H\mathbf x\subseteq\mathcal G$, hence $\mathcal H\mathbf x\subseteq\rcol{\mathcal G}{\mathcal H}$.

Due to $\mathcal H_0:=\rcol{\mathcal G}{\mathcal H}$ being a left ideal of $\mathcal H$, we specifically know that the left action of $\mathcal H^1$ by multiplication on $\mathcal H$ preserves $\mathcal H_0$, hence it makes sense and is well-defined to consider the action on the quotient $\mathcal H/\mathcal H_0$. Written explicitly, this is
\[
    \mathbf u(\mathbf q+\mathcal H_0) = \mathbf u\mathbf q + \mathcal H_0.
\]
Further, since $\mathbf u\mathcal H_0\subseteq\mathcal G$, if we have $\mathbf q_1+\mathcal H_0=\mathbf q_2+\mathcal H_0$, then $\mathbf u\mathbf q_1 \in\mathcal G$ if and only if $\mathbf u\mathbf q_2 \in\mathcal G$. This means that to decide whether $\mathcal G$ is $\mathcal H$-perceptive, it is enough to examine the action of $\mathcal H^1$ on $\mathcal H/\mathcal H_0$ -- namely to see whether $\bigcup_{\mathbf u\in\mathcal H^1}\mathbf u\mathcal G/\mathcal H_0 = \mathcal H/\mathcal H_0$.

\begin{prop}
\label{prp:perceptive-one-dim-step}
Suppose that $\mathcal G\subseteq\mathcal H$ are orders such that $\mathcal H/\mathcal G$ is a cyclic module, i.e. that $\mathcal H/\mathcal G\simeq \OK/\mathfrak p^e$ for some prime ideal $\mathfrak p$ of $\OK$ with $q:=\Nm(\mathfrak p)$ and $e\geq1$. Then $\frac{\#\mathcal H^1}{\#\mathcal G^1} \leq q^e+q^{e-1}$ and $\mathcal G$ is $\mathcal H$-perceptive if and only if equality occurs.
\end{prop}
\begin{proof}
First we investigate the right conductor $\mathcal H_0:= \rcol{\mathcal G}{\mathcal H}$. Since $\mathcal H/\mathcal G$ is a cyclic module, we may take an arbitrary lift $\mathbf z\in\mathcal H$ of some generator to get that $\mathcal H = \OK\mathbf z + \mathcal G$ as $\OK$-modules for some $\mathbf z\in\mathcal H$. Then for $\mathbf x\in\mathcal G$, we have $\mathcal H\mathbf x \subseteq \mathcal G$ if and only if $\mathbf z\mathbf x\in\mathcal G$, so we interpret $\mathcal H_0$ as the kernel of the $\OK$-linear map
\begin{align*}
    \mathcal G &\to \mathcal H/\mathcal G,\\
    \mathbf x &\mapsto \mathbf z\mathbf x+\mathcal G.
\end{align*}
Clearly, $1$ maps to $\mathbf z+\mathcal G$ which is a generator of $\mathcal H/\mathcal G$, so the map is surjective and hence we get
\[
    \mathcal G/\mathcal H_0 \simeq \mathcal H/\mathcal G \simeq \OK/\mathfrak p^e.
\]
Obviously we have $\mathfrak p^e\mathcal H\subseteq\mathcal G$ and thus also $\mathfrak p^e\mathcal H\subseteq \mathcal H_0$. This taken together with $\mathcal G/\mathcal H_0 \simeq \mathcal H/\mathcal G \simeq \OK/\mathfrak p^e$ means $\mathcal H/\mathcal H_0\simeq (\OK/\mathfrak p^e)^2$.

Now we view $\mathcal H/\mathcal H_0$ as a \uv{plane} (a rank $2$ free module) over the ring $R:=\OK/\mathfrak p^e$. Further, $\mathcal G/\mathcal H_0$ sits inside it as a \uv{line}, by which we mean it is a free cyclic $R$-submodule. Then for any $\mathbf u\in\mathcal H^1$, the set $\mathbf u\mathcal G/\mathcal H_0$ is again some line in the plane $\mathcal H/\mathcal H_0$, and $\mathcal G$ is $\mathcal H$-perceptive if and only if these lines collectively cover the entire plane. Finally, $\mathbf u_1\mathcal G/\mathcal H_0 = \mathbf u_2\mathcal G/\mathcal H_0$ if and only if $(\inv{\mathbf u_2}\mathbf u_1)\mathcal G/\mathcal H_0 = \mathcal G/\mathcal H_0$, which surely happens if and only if $\inv{\mathbf u_2}\mathbf u_1\in\mathcal G$. Thus we have constructed an injective map
\begin{align*}
    \set{\text{left cosets of $\mathcal G^1$ in $\mathcal H^1$}} & \injto \set{\text{lines in $\mathcal H/\mathcal H_0$}}\\
    \mathbf u\mathcal G^1 &\mapsto \mathbf u\mathcal G/\mathcal H_0.
\end{align*}

Next we notice that every element in $\mathcal H/\mathcal H_0$ lies on some line (possibly on multiple lines) and every line has a generator, which must lie only on the one line it generates. This means that $\mathcal G$ will be $\mathcal H$-perceptive if and only if the injective map above is surjective. We know there are $\frac{\#\mathcal H^1}{\#\mathcal G^1}$ cosets, so it remains to count the lines in $\mathcal H/\mathcal H_0$.

For that, we just notice that $(\alpha,\beta)\in R^2$ generates a free cyclic submodule if and only if at least of one $\alpha$, $\beta$ lies in $R^\times = R\setminus\mathfrak p R$. Thus after multiplying the generator by a suitable scalar from $R^\times$, we see that we may enumerate all the lines as those generated by $(1,\beta)$ for $\beta\in R$ and by $(\alpha,1)$ for $\alpha\in\mathfrak pR$, and that these are all distinct. Thus we have $\#R+\#(\mathfrak pR) = q^e+q^{e-1}$ lines in $\mathcal H/\mathcal H_0$. In view of the injective map established above, this gives us the inequality $\frac{\#\mathcal H^1}{\#\mathcal G^1}\leq q^e+q^{e-1}$ unconditionally and perceptivity if and only if equality holds.
\end{proof}

\begin{example}
\label{ex:perceptive-composing}
Let us provide an example illustrating the necessity of $[\mathcal H:\mathcal G]_{\OK}$ and $[\mathcal G:\mathcal F]_{\OK}$ being comaximal in Proposition~\ref{prp:composing-perceptivity}. In the quaternion algebra $\quatalg{-3,-1}{\kve} = \kve\oplus\kve\ii\oplus\kve\j\oplus\kve\k$ let us consider orders
\begin{align*}
    \mathcal H&:=\zet\oplus\zet\frac{1+\ii}2\oplus\zet\j\oplus\zet\frac{\j+\k}2, \\
    \mathcal G&:=\zet\oplus\zet\ii\oplus\zet\zav{\frac{1-\ii}2+\j}\oplus\zet\frac{1+\ii+\j+\k}2, \\
    \mathcal F&:=\zet\oplus\zet2\ii\oplus\zet\zav{\frac{1-\ii}2+\j}\oplus\zet\frac{1+\ii+\j+\k}2.
\end{align*}
One may verify that these are indeed $\zet$-orders and that $\mathcal H/\mathcal F\simeq\zet/4\zet$ and $\mathcal H/\mathcal G\simeq\mathcal G/\mathcal F\simeq\zet/2\zet$ as $\zet$-modules. Next, we have the following groups of units of reduced norm $1$:
\begin{align*}
    \mathcal H^1 &= \set{\pm1,\ \frac{\pm1\pm\ii}2,\ \pm\j,\ \frac{\pm\j\pm\k}2},  & \mathcal F^1&= \set{\pm1}.
\end{align*}
Applying Proposition~\ref{prp:perceptive-one-dim-step}, this immediately tells us that $\mathcal F$ is $\mathcal H$-perceptive. Even without computing $\mathcal G^1$ explicitly, we may then see that $\mathcal G$ is $\mathcal H$-perceptive (Proposition~\ref{prp:composing-perceptivity}(ii)) and thus $\#\mathcal G^1 = 4$ (Proposition~\ref{prp:perceptive-one-dim-step} again). Finally, since $\frac{\#\mathcal G^1}{\#\mathcal F^1} = 2 \lneq 2+1$, this means that $\mathcal F$ is \emph{not} $\mathcal G$-perceptive.
\end{example}

To close out this subsection, let us provide a partial analogue to Proposition~\ref{prp:perceptive-one-dim-step} in the following sense: if we consider specifically the case when $\mathcal H/\mathcal G\simeq\OK/\mathfrak p$ as $\OK$-modules, the Proposition gives a bound $\frac{\#\mathcal H^1}{\#\mathcal G^1}\leq \Nm(\mathfrak p)+1$ and says perceptivity happens if and only if equality occurs. Hence this gives a concise way to recognize a perceptive submodule when it is a \uv{maximal suborder} as described in Proposition~\ref{prp:smallest-steps}(i). Let us provide a similar answer for a maximal suborder described by Proposition~\ref{prp:smallest-steps}(ii).

\begin{prop}
\label{prp:perceptive-two-dim-step}
Suppose that $\mathcal G\subseteq\mathcal H$ are orders such that $\mathcal H/\mathcal G\simeq (\OK/\mathfrak p)^2$ as $\OK$-modules for some prime ideal $\mathfrak p$ of $\OK$ and $\mathcal G/\mathfrak p\mathcal H$ is a field. Then $\frac{\#\mathcal H^1}{\#\mathcal G^1} \leq \Nm(\mathfrak p)^2+1$ and $\mathcal G$ is $\mathcal H$-perceptive if and only if equality holds.
\end{prop}
\begin{proof}
We proceed similarly to the proof of Proposition~\ref{prp:perceptive-one-dim-step}.
Let us start by proving that the right conductor
\[
    \mathcal H_0 := \rcol{\mathcal G}{\mathcal H} = \set{\mathbf x\in\mathcal G\mid \mathcal H\mathbf x\subseteq\mathcal G}
\]
is in fact equal to $\mathfrak p\mathcal H$. We have $\mathfrak p\mathcal H\subseteq \mathcal H_0$, so we may quotient everything by this two-sided ideal. Then we are in the four-dimensional algebra $B := \mathcal H/\mathfrak p\mathcal H$ and have a two-dimensional subalgebra $A := \mathcal G/\mathfrak p\mathcal H \subset B$ that is in fact a field. Notice that $\mathcal H_0$ was a right ideal  of $\mathcal G$, so $\mathcal H_0/\mathfrak p\mathcal H$ is right ideal of $A$. So it is an $A$-vector subspace of the one-dimensional space $A$; additionally, we clearly have $1\notin \mathcal H_0$, so $\mathcal H_0\subsetneq \mathcal G$, and thus $\mathcal H_0/\mathfrak p\mathcal H$ must be a proper subspace of $A$, which forces $\mathcal H_0/\mathfrak p\mathcal H=0$.

Still viewing $B$ as a two-dimensional vector space over $A$ (the vector space structure given by multiplication from the right), we may see that for every $\mathbf u\in\mathcal H^1$, the set $\mathbf u \mathcal G/\mathfrak p\mathcal H = \mathbf uA$ is again a one-dimensional $A$-vector subspace. Hence, similar to the proof of Proposition~\ref{prp:perceptive-one-dim-step} working over $A$, we obtain an injective map
\begin{align*}
    \set{\text{left cosets of $\mathcal G^1$ in $\mathcal H^1$}} & \injto \set{\text{one-dimensional $A$-vector subspaces of $B$}}\\
    \mathbf u\mathcal G^1 &\mapsto \mathbf uA,
\end{align*}
since again $\mathbf u_1 A = \mathbf u_2 A$ if and only if $(\inv{\mathbf u_2}\mathbf u_2)A = A$ if and only if $(\inv{\mathbf u_2}\mathbf u_2)\in \mathcal G$. Perceptivity occurs if and only if the map is surjective, and since $A$ is a quadratic extension of $\OK/\mathfrak p$ and thus has $\Nm(\mathfrak p)^2$ elements, the two-dimensional $A$-vector space $B$ has $\Nm(\mathfrak p)^2+1$ one-dimensional subspaces. The conclusion of the Proposition then follows.
\end{proof}

\subsection{The case of a linear poset of orders}

In this subsection, we will provide a slight generalization to the results of Propositions~\ref{prp:perceptive-one-dim-step} and \ref{prp:perceptive-two-dim-step}, which will later coincidentally cover most of the perceptive suborders of maximal orders of class number $1$ that we find in the following subsection.

\begin{defn}
Let $\mathcal G\subseteq\mathcal H$ be orders. Let us say the pair $\mathcal G$, $\mathcal H$ \emph{has a linear poset of orders}, if the partially ordered set (poset) of intermediate orders $\mathcal M$ satisfying $\mathcal G\subseteq\mathcal M\subseteq\mathcal H$ ordered by inclusion is linear.
\end{defn}

In other words, $\mathcal G$, $\mathcal H$ has a linear poset of orders if there is only one chain
\[
    \mathcal G =: \mathcal M_1\subsetneq\mathcal M_2\subsetneq\cdots\subsetneq\mathcal M_{\ell-1}\subsetneq \mathcal M_{\ell}:=\mathcal H
\]
of orders between $\mathcal G$ and $\mathcal H$ that cannot be further refined. Note that the situation of Proposition~\ref{prp:perceptive-two-dim-step} satisfies this trivially and the situation of Proposition~\ref{prp:perceptive-one-dim-step} satisfies this because already the poset of intermediate $\OK$-modules $M$, $\mathcal G\subseteq M\subseteq\mathcal H$, being isomorphic to the poset of submodules of $\mathcal H/\mathcal G \simeq \OK/\mathfrak p^e$, is linear.

\begin{prop}
\label{prp:conductors}
Suppose that the pair $\mathcal G$, $\mathcal H$ has a linear poset of orders, that poset being \[\mathcal G=: \mathcal M_1\subsetneq\cdots\subsetneq \mathcal M_{\ell}:=\mathcal H.\]
Then
\begin{enumerate}
\item The collection of right conductors $\set{\rcol{\mathcal G}{\mathcal M_i}}$ forms an opposite poset whilst also satisfying $[\rcol{\mathcal G}{\mathcal M_i}:\rcol{\mathcal G}{\mathcal M_j}]_{\OK} = [\mathcal M_j:\mathcal M_i]_{\OK}$.
\item Let $\mathcal H_0 := \rcol{\mathcal G}{\mathcal H}$. Then for any $\mathcal M =\mathcal M_i$ and $\mathbf u_1,\mathbf u_2\in\mathcal H^1$, we have $\mathbf u_1\rcol{\mathcal G}{\mathcal M}/\mathcal H_0 = \mathbf u_2\rcol{\mathcal G}{\mathcal M}/\mathcal H_0$ if and only if $\mathbf u_1\mathcal M^1=\mathbf u_2\mathcal M^1$.
\item Let $\mathbf u_1,\mathbf u_2\in\mathcal H^1$ be given and let $\mathbf u:=\inv{\mathbf u_2}\mathbf u_1$ and $\mathcal M:=\mathcal G+\mathcal G\mathbf u$. Then \[(\mathbf u_1\mathcal G/\mathcal H_0) \cap (\mathbf u_2\mathcal G/\mathcal H_0) = \mathbf u_1\rcol{\mathcal G}{\mathcal M}/\mathcal H_0 = \mathbf u_2\rcol{\mathcal G}{\mathcal M}/\mathcal H_0.\]
\end{enumerate}
\end{prop}
\begin{proof}
\begin{enumerate}
\item First we show that the map $\mathcal M\mapsto \rcol{\mathcal G}{\mathcal M}$ reverses (non-strict) inclusions.
For any $\mathbf x\in\rcol{\mathcal G}{\mathcal M}$ must have $1\mathbf x\in \mathcal G$ due to $1\in\mathcal M$, so $\rcol{\mathcal G}{\mathcal M}\subseteq \mathcal G$. Next, if $\mathcal L\subseteq \mathcal M$ and $\mathbf x\in \rcol{\mathcal G}{\mathcal M}$, then
\[
    \mathcal L\mathbf x\subseteq\mathcal M\mathbf x \subseteq \mathcal G,
\]
meaning $\mathbf x \in \rcol{\mathcal G}{\mathcal L}$ and $\rcol{\mathcal G}{\mathcal M}\subseteq \rcol{\mathcal G}{\mathcal L}$.

Next, we prove the statement on indices. Since indices behave multiplicatively on a chain of lattices and we obviously have $\rcol{\mathcal G}{\mathcal G}=\mathcal G$, it suffices to show $[\mathcal G:\rcol{\mathcal G}{\mathcal M}]_{\OK} = [\mathcal M:\mathcal G]_{\OK}$ for all $\mathcal M$. Here we use the condition that the poset of orders is linear: since this is the case, none of the orders may be covered by the union of its proper suborders. Thus we may choose an element $\mathbf z\in\mathcal M$ that does not belong to any proper suborder of $\mathcal M$ (that also contains $\mathcal G$). Then, since $\mathcal G+\mathcal G\mathbf z$ and $\mathcal G+\mathbf z\mathcal G$ are both orders by Lemma~\ref{lem:its-an-order} and contain $\mathbf z$, we must have $\mathcal G+\mathcal G\mathbf z=\mathcal G+\mathbf z\mathcal G=\mathcal M$. With this, we may proceed similarly to the proof of Proposition~\ref{prp:perceptive-one-dim-step}: we see that for $\mathbf x\in\mathcal G$, if $\mathcal M\mathbf x\subseteq\mathcal G$, then surely $\mathbf z\mathbf x\in\mathcal G$, and conversely, $\mathbf z\mathbf x\in\mathcal G$ implies
\[
    \mathcal M\mathbf x = (\mathcal G+\mathcal G\mathbf z)\mathbf x = \mathcal G\mathbf x+\mathcal G\mathbf z\mathbf x\subseteq\mathcal G\mathcal G+\mathcal G\mathcal G = \mathcal G.
\]
So we interpret $\rcol{\mathcal G}{\mathcal M}$ as the kernel of the $\OK$-linear map
\begin{align*}
    \mathcal G &\to \mathcal M/\mathcal G,\\
    \mathbf x&\mapsto\mathbf z\mathbf x+\mathcal G.
\end{align*}
The image must then be $(\mathbf z\mathcal G+\mathcal G)/\mathcal G = \mathcal M/\mathcal G$, so we obtain
\[
    \mathcal G/\rcol{\mathcal G}{\mathcal M} \simeq \mathcal M/\mathcal G,
\]
which implies $[\mathcal G:\rcol{\mathcal G}{\mathcal M}]_{\OK} = [\mathcal M:\mathcal G]_{\OK}$.

Because of this relation of indices, we can improve the reversal of non-strict inclusions to strict inclusions as well: if $\mathcal M_i\subsetneq\mathcal M_j$, then also $\rcol{\mathcal G}{\mathcal M_i}\supsetneq\rcol{\mathcal G}{\mathcal M_j}$. Since the original poset of orders was linear, any two $\mathcal M_i$ and $\mathcal M_j$ were comparable, which then yields an opposite comparison of $\rcol{\mathcal G}{\mathcal M_i}$ and $\rcol{\mathcal G}{\mathcal M_j}$, so we see that everything in $\set{\rcol{\mathcal G}{\mathcal M_i}}$ is also comparable, in reversed order.

\item Letting $\mathbf u:=\inv{\mathbf u_2}\mathbf u_1$, we equivalently want to prove
\[
\mathbf u\rcol{\mathcal G}{\mathcal M}/\mathcal H_0 = \rcol{\mathcal G}{\mathcal M}/\mathcal H_0 \qquad\text{if and only if}\qquad \mathbf u\in\mathcal M^1.
\]

First, suppose $\mathbf u\in\mathcal M^1$. We know that $\rcol{\mathcal G}{\mathcal M}$ is a left ideal of $\mathcal M$, so $\mathbf u\rcol{\mathcal G}{\mathcal M}\subseteq \rcol{\mathcal G}{\mathcal M}$. Similarly we get $\inv{\mathbf u}\rcol{\mathcal G}{\mathcal M}\subseteq \rcol{\mathcal G}{\mathcal M}$, so the equality $\mathbf u\rcol{\mathcal G}{\mathcal M}=\rcol{\mathcal G}{\mathcal M}$ follows. We finish proving this implication by just descending to the quotient by $\mathcal H_0$.

Second, suppose that $\mathbf u\rcol{\mathcal G}{\mathcal M}/\mathcal H_0=\rcol{\mathcal G}{\mathcal M}/\mathcal H_0$. Lifting from $\mathcal H/\mathcal H_0$ to $\mathcal H$, this means $\mathbf u\rcol{\mathcal G}{\mathcal M} = \rcol{\mathcal G}{\mathcal M}\subseteq\mathcal G$. Let us denote $\mathcal L := \mathcal G+\mathcal G\mathbf u$, by Lemma~\ref{lem:its-an-order}, this is an order. We also see
\[
    \mathcal L\rcol{\mathcal G}{\mathcal M} = \mathcal G\rcol{\mathcal G}{\mathcal M}+\mathcal G\mathbf u\rcol{\mathcal G}{\mathcal M} \subseteq \mathcal G\mathcal G+\mathcal G\mathcal G = \mathcal G,
\]
and thus $\rcol{\mathcal G}{\mathcal M}\subseteq \rcol{\mathcal G}{\mathcal L}$. By part (i), we know these right conductors form an opposite poset to the orders, with $\mathcal M\mapsto \rcol{\mathcal G}{\mathcal M}$ being the antiisomorphism, so $\rcol{\mathcal G}{\mathcal M}\subseteq \rcol{\mathcal G}{\mathcal L}$ implies $\mathcal M\supseteq\mathcal L\ni\mathbf u$.

\item Multiplying by $\inv{\mathbf u_2}$ from the left, we are equivalently proving
\[
    (\mathbf u\mathcal G/\mathcal H_0)\cap (\mathcal G/\mathcal H_0) = \mathbf u\rcol{\mathcal G}{\mathcal M}/\mathcal H_0 = \rcol{\mathcal G}{\mathcal M}/\mathcal H_0,
\]
where the last equality was already the content of part (ii). We have trivial inclusions $\rcol{\mathcal G}{\mathcal M}/\mathcal H_0\subseteq\mathcal G/\mathcal H_0$ and $\mathbf u\rcol{\mathcal G}{\mathcal M}/\mathcal H_0\subseteq\mathbf u\mathcal G/\mathcal H_0$, making one direction clear. For the other, if $\mathbf u\mathbf x+\mathcal H_0\in \mathcal G/\mathcal H_0$, it means $\mathbf u\mathbf x \in \mathcal G$, and through $\mathcal M = \mathcal G+\mathcal G\mathbf u$ it follows that $\mathcal M\mathbf x\subseteq\mathcal G$, so $\mathbf x\in\rcol{\mathcal G}{\mathcal M}$.
\qedhere
\end{enumerate}
\end{proof}

With this Proposition, we are sufficiently equipped to account for the overcounting stemming from the intersections of various $\mathbf u\mathcal G/\mathcal H_0$ in estimating the size of $\bigcup_{\mathbf u\in\mathcal H^1}\mathbf u\mathcal G/\mathcal H_0$.

\begin{prop}
\label{prp:perceptivity-count}
Suppose that the pair $\mathcal G$, $\mathcal H$ has a linear poset of orders, that poset being \[\mathcal G=: \mathcal M_1\subsetneq\mathcal M_2\subsetneq\cdots\subsetneq \mathcal M_{\ell}:=\mathcal H.\]
If $\mathcal M_2$ is $\mathcal H$-perceptive, then
\[
    \frac{\#\mathcal H^1}{\#\mathcal G^1} \leq \Nm([\mathcal H:\mathcal G]_{\OK})+\Nm([\mathcal H:{\mathcal M_2}]_{\OK})
\]
and equality holds if and only if $\mathcal G$ is $\mathcal H$-perceptive.
\end{prop}
\begin{proof}
We calculate $\#\zav{\bigcup_{\mathbf u\in\mathcal H^1}\mathbf u\mathcal G/\mathcal H_0}$. We split the set $\mathcal G/\mathcal H_0$ into \[\rcol{\mathcal G}{\mathcal M_i}/\mathcal H_0\setminus\rcol{\mathcal G}{\mathcal M_{i+1}}/\mathcal H_0\qquad \text{for $i=1,\dots,\ell-1$}\] and the singleton $\mathcal H_0/\mathcal H_0$. By Proposition~\ref{prp:conductors}, we then see that there are $\frac{\#\mathcal H^1}{\#\mathcal M_i^1}$ possible results of
\[
    \mathbf u\zav{\rcol{\mathcal G}{\mathcal M_i}/\mathcal H_0\setminus\rcol{\mathcal G}{\mathcal M_{i+1}}/\mathcal H_0}
\]
as $\mathbf u\in\mathcal H^1$ varies and that for distinct $i$, these sets are disjoint. Since
\begin{multline*}
    \#\Bigl({\rcol{\mathcal G}{\mathcal M_i}/\mathcal H_0\setminus\rcol{\mathcal G}{\mathcal M_{i+1}}/\mathcal H_0}\Bigr) = \#\rcol{\mathcal G}{\mathcal M_i}/\mathcal H_0 - \#\rcol{\mathcal G}{\mathcal M_{i+1}}/\mathcal H_0 =\\
    = \Nm([\rcol{\mathcal G}{\mathcal M_i}:\rcol{\mathcal G}{\mathcal H}]_{\OK}) - \Nm([\rcol{\mathcal G}{\mathcal M_{i+1}}:\rcol{\mathcal G}{\mathcal H}]_{\OK}) \stackrel{\text{\ref{prp:conductors}(i)}}{=}\\\stackrel{\text{\ref{prp:conductors}(i)}}{=}
    \Nm([\mathcal H:\mathcal M_i]_{\OK}) - \Nm([\mathcal H:\mathcal M_{i+1}]_{\OK}),
\end{multline*}
we then count
\begin{align*}
    \#\zav{\bigcup_{\mathbf u\in\mathcal H^1}\mathbf u\mathcal G/\mathcal H_0} = \sum_{i=1}^\ell \frac{\#\mathcal H^1}{\#\mathcal M_{i}^1}\Bigl({\Nm([\mathcal H:\mathcal M_i]_{\OK}) - \Nm([\mathcal H:\mathcal M_{i+1}]_{\OK})}\Bigr)+1.
\end{align*}
The left-hand side is bounded above by $\#(\mathcal H/\mathcal H_0) = \Nm([\mathcal H:\mathcal G]_{\OK})^2$ and equality is equivalent to perceptivity of $\mathcal G$, $\mathcal H$.

But we know that $\mathcal M_2$ is $\mathcal H$-perceptive, so we may take the right-hand side sum starting for $\mathcal M_2$ in place of $\mathcal G$. Since the terms of the sum individually do not depend on $\mathcal G$, this corresponds to omitting the first term, so altogether, we obtain the statement
\begin{multline*}
    \Nm([\mathcal H:\mathcal G]_{\OK})^2 \geq \#\zav{\bigcup_{\mathbf u\in\mathcal H^1}\mathbf u\mathcal G/\mathcal H_0} =\\= \frac{\#\mathcal H^1}{\#\mathcal G^1}\Bigl({\Nm([\mathcal H:\mathcal G]_{\OK}) - \Nm([\mathcal H:\mathcal M_{2}]_{\OK})}\Bigr) + \Nm([\mathcal H:\mathcal M_2]_{\OK})^2.
\end{multline*}
Moving the term $\Nm([\mathcal H:\mathcal M_2]_{\OK})^2$ to the left-hand side and dividing both sides by
\[
    {\Nm([\mathcal H:\mathcal G]_{\OK})} - \Nm([\mathcal H:\mathcal M_{2}]_{\OK}),
\]
we obtain the desired inequality and we see that through the manipulations performed, equality is still equivalent to perceptivity of $\mathcal G$, $\mathcal H$.
\end{proof}

\begin{rem}
\label{rmrk:nonlinear-poset}
The methods of this subsection might, perhaps, be used in some further situations other than the case of a linear poset of orders. Essentially, everything we derived here stemmed from Proposition~\ref{prp:conductors}(i). Anytime its analogue -- that is, the collection of right conductors $\set{\rcol{\mathcal G}{\mathcal M}}$ forming an opposite poset to $\set{\text{orders $\mathcal M$}\mid \mathcal G\subseteq\mathcal M\subseteq\mathcal H}$ along with the property about indices -- could be established, the rest of the subsection could follow with the only change being a potentially more involved combinatorial calculation (and a different inequality) in the analogue of Proposition~\ref{prp:perceptivity-count}, depending on the poset of orders.
\end{rem}

\subsection{Searching for perceptive suborders}

In this subsection, we will use the results of this section so far to find all perceptive suborders $\mathcal G$ of maximal orders $\mathcal H$ of class number $1$. These are (left and right) PIDs, and we saw in Section~\ref{chap:factorizations} that this implies that $(\mathcal H,\nrd)$ is a universal quadratic form. Because the action of $\mathcal H^1$ used to define perceptivity preserves reduced norms, such a perceptive suborder gives us automatically:

\begin{prop}
\label{prp:basically-lagrange}
If $\mathcal H$ is a maximal order of class number $1$ and $\mathcal G\subseteq\mathcal H$ is an $\mathcal H$-perceptive suborder, then $(\mathcal G,\nrd)$ is a universal quadratic form.
\end{prop}

For these starting orders $\mathcal H$, we appeal to a result of Kirschmer and Lorch:

\begin{thm}[{\cite[Theorems II and III]{kirschmer-lorch}}]
\label{thrm:kirschmer-list}
Up to isomorphism, there are $154$ orders of class number $1$ in definite quaternion algebras over totally real number fields. Of them, $49$ are maximal orders.
\end{thm}
We may note that these maximal orders of class number $1$ occur over $15$ different number fields, of degree up to $5$. Some of these orders may only differ by a map induced from an automorphism of the number field though.

Now we need an algorithm to find all perceptive suborders of a given order, which we will then perform on all $49$ maximal orders of class number $1$. Though we will still rely on brute force to some extent, let us sum up results that can help in this search:

\begin{lemma}
\label{lem:search-tools}
Let $\mathcal G\subsetneq\mathcal H$ be an $\mathcal H$-perceptive suborder. Then:
\begin{enumerate}
\item For any prime ideal factor $\mathfrak p$ of $[\mathcal H:\mathcal G]_{\OK}$, one of $\Nm(\mathfrak p)+1$ or $\Nm(\mathfrak p)^2+1$ divides $\frac{\#\mathcal H^1}{\#\mathcal G^1}$, which in turn divides $\frac{\#\mathcal H^1}2$.
\item $\Nm([\mathcal H:\mathcal G]_{\OK}) =[\mathcal H:\mathcal G]_{\zet}< \frac{\#\mathcal H^1}{\#\mathcal G^1} \leq \frac{\#\mathcal H^1}2$.
\item If $\mathfrak a$ is an ideal of $\OK$ such that $\mathfrak a\mathcal H\subseteq\mathcal G$, then for verifying that $\mathcal G$ is $\mathcal H$-perceptive, it suffices to check $\bigcup_{\mathbf u\in\mathcal H^1}\mathbf u\mathcal G/\mathfrak a\mathcal H = \mathcal H/\mathfrak a\mathcal H$.
\item To determine whether the pair $\mathcal G$, $\mathcal H$ is perceptive, it suffices to determine whether each pair $\mathcal G_i$, $\mathcal G_{i+1}$ is perceptive in a chain
\[
    \mathcal G=:\mathcal G_0\subsetneq \cdots \subsetneq\mathcal G_r := \mathcal H
\]
where $[\mathcal G_{i+1}:\mathcal G_i]_{\OK}=\mathfrak p_i^{e_i}$ are the individual powers of prime ideals from the factorization of $[\mathcal H:\mathcal G]_{\OK}$.
\end{enumerate}
\end{lemma}
\begin{proof}
\begin{enumerate}
\item Combining Propositions~\ref{prp:prime-power-chain} and \ref{prp:smallest-steps}, we can find an order $\mathcal F$ with $\mathcal G\subseteq\mathcal F \subseteq\mathcal H$ and either $\mathcal H/\mathcal F\simeq\OK/\mathfrak p$ or $\mathcal H/\mathcal F\simeq(\OK/\mathfrak p)^2$ with $\mathcal F/\mathfrak p\mathcal H$ being a field. Then, we apply Proposition~\ref{prp:perceptive-one-dim-step} or Proposition~\ref{prp:perceptive-two-dim-step}. Finally, we note that $\set{\pm1}$ is a subgroup of $\mathcal G^1$, hence $2\mid \#\mathcal G^1$ and so $\frac{\#\mathcal H^1}{\#\mathcal G^1} \mid \frac{\#\mathcal H^1}{2}$.

\item We may denote $\mathcal H_0 := \rcol{\mathcal G}{\mathcal H}$ and view the action of $\mathcal H^1$ on $\mathcal H/\mathcal H_0$ as in Subsection~\ref{sec:module-conditions}. Whenever $\mathbf u_1\mathcal G^1 = \mathbf u_2\mathcal G^1$, we have $\mathbf u_1\mathcal G/\mathcal H_0 = \mathbf u_2\mathcal G/\mathcal H_0$, so we may view the union
\[
    \bigcup_{\mathbf u\in\mathcal H^1} \mathbf u\mathcal G/\mathcal H_0
\]
as indexed by left cosets $\mathbf u\mathcal G^1$ instead of the individual $\mathbf u$. Each $\mathbf u \mathcal G/\mathcal H_0$ has the same number of elements, so we bound
\[
    \#\zav{\bigcup_{\mathbf u~G^1} \mathbf u\mathcal G/\mathcal H_0} \leq \frac{\#\mathcal H^1}{\#\mathcal G^1} \cdot \#(\mathcal G/\mathcal H_0).
\]
Further, as long as $\frac{\#\mathcal H^1}{\#\mathcal G^1}>1$, the inequality is strict, since any two $\mathbf u \mathcal G/\mathcal H_0$ intersect in $0+\mathcal H_0$. Since $\mathcal G\subsetneq\mathcal H$, perceptivity requires that indeed $\frac{\#\mathcal H^1}{\#\mathcal G^1}>1$. Thus $\mathcal H$-perceptivity of $\mathcal G$ implies
\begin{align*}
    [\mathcal H:\mathcal G]_{\zet}\cdot \#(\mathcal G/\mathcal H_0) &= \#(\mathcal H/\mathcal H_0) = \#\zav{\bigcup_{\mathbf u~G^1} \mathbf u\mathcal G/\mathcal H_0} <\frac{\#\mathcal H^1}{\#\mathcal G^1} \cdot \#(\mathcal G/\mathcal H_0),
\end{align*}
whence the conclusion follows.
\item Notice that $\mathfrak a\mathcal H$ is a two-sided ideal of $\mathcal H$ which satisfies
\[
    \mathcal H(\mathfrak a\mathcal H) = \mathfrak a(\mathcal H\mathcal H) = \mathfrak a\mathcal H \subseteq\mathcal G,
\]
i.e. $\mathfrak a\mathcal H\subseteq\rcol{\mathcal G}{\mathcal H}$. Being a two-sided ideal (so especially a left ideal), the left action of $\mathcal H^1$ on $\mathcal H$ preserves $\mathfrak a\mathcal H$. So as we did with $\mathcal H/\rcol{\mathcal G}{\mathcal H}$, we may consider the left $\mathcal H^1$-action of $\mathcal H/\mathfrak a\mathcal H$ and check that all orbits of this action intersect $\mathcal G/\mathfrak a\mathcal H$ there.
\item This is just Proposition~\ref{prp:prime-power-chain} combined with Proposition~\ref{prp:composing-perceptivity}.
\qedhere
\end{enumerate}
\end{proof}

Now we may state the algorithm to find perceptive suborders.

\begin{algo}
\label{algo}
\null\par\noindent
Input: an $\OK$-order $\mathcal H$ in a definite quaternion algebra $\mathcal A$ over a number field $K$.
\par\noindent
Output: the set of all $\mathcal H$-perceptive suborders of $\mathcal H$.

\newcount\algcount
\def\algitem{\global\advance\algcount1\relax\item[(\the\algcount)]}
\newdimen\nestincrement\nestincrement1.75em
\let\nestin\relax\let\nestout\relax
\def\nestin{\begin{itemize}[leftmargin=\nestincrement, topsep=\smallskipamount, itemsep=\smallskipamount]}
\def\nestout{\end{itemize}}
\nestin
\algitem Compute $\mathcal H^1$ and determine all the prime ideals $\mathfrak p_1,\dots,\mathfrak p_r$ such that $\Nm(\mathfrak p_i)+1$ or $\Nm(\mathfrak p_i)^2+1$ divides $\frac{\#\mathcal H^1}2$.
\algitem Initialize $\Omega := \set{\mathcal H}$.
\algitem For each $i=1,\dots, r$, pick $\mathfrak p:=\mathfrak p_i$, $q:=\Nm(\mathfrak p)$ and do:
\nestin
\algitem Initialize $\Omega_{\text{new}} := \set{}$.
\algitem For each $\mathcal G\in\Omega$, do:
\nestin
\algitem Initialize $\Gamma:=\set{\mathcal G}$.
\algitem As long as $\Gamma$ is non-empty:
\nestin
\algitem Pick $\mathcal F\in \Gamma$, and update $\Gamma:=\Gamma\setminus\set{\mathcal F}$, $\Omega_{\text{new}}:=\Omega_{\text{new}}\cup\set{\mathcal F}$.
\hfuzz2pt
\algitem Compute $S_1 := \set{\text{suborders $\mathcal L\subset\mathcal F$ with $\mathcal F/\mathcal L\simeq\OK/\mathfrak p$}}$ by checking all three-dimensional subspaces of $\mathcal F/\mathfrak p\mathcal F$ containing $1$ to see whether they are closed under multiplication.
\algitem Compute $S_2 := \{\text{suborders $\mathcal L\subset\mathcal F$ where $\mathcal F/\mathcal L\simeq(\OK/\mathfrak p)^2$ and }\allowbreak\text{$\mathcal L/\mathfrak p\mathcal F$    is a field}\}$ by checking all two-dimensional subspaces of $\mathcal F/\mathfrak p\mathcal F$ containing $1$ to see whether they are fields.
\algitem For each $\mathcal L\in S_1\cup S_2$, if $[\mathcal G:\mathcal L]_{\zet}<\frac{\#\mathcal G^1}{\#\mathcal L^1}$, then we check whether $\mathcal L$ is $\mathcal G$-perceptive:
\nestin
\algitem If already $\mathcal L\in\Omega_{\text{new}}$, take no action.
\algitem Else, if $\mathcal G/\mathcal L\simeq \OK/\mathfrak p^e$ for some $e$, check perceptivity by checking whether $\frac{\#\mathcal G^1}{\#\mathcal L^1} = q^e+q^{e-1}$.
\algitem Else, if $\mathcal G/\mathcal L\simeq(\OK/\mathfrak p)^2$ and $\mathcal L/\mathfrak p\mathcal G$ is a field, check perceptivity by checking whether $\frac{\#\mathcal G^1}{\#\mathcal L^1} = q^2+1$.
\algitem Otherwise, find $e$ such that $\mathfrak p^e\mathcal G\subseteq\mathcal L$ and check perceptivity in $\mathcal G/\mathfrak p^e\mathcal G$.
\algitem In either case, if $\mathcal L$ is $\mathcal G$-perceptive and $\mathcal L\notin\Omega_{\text{new}}$ yet, then update $\Gamma := \Gamma\cup\set{\mathcal L}$.
\nestout
\nestout
\nestout
\algitem Update $\Omega := \Omega_{\text{new}}$.
\nestout
\algitem Return $\Omega$.
\nestout
\end{algo}

In a practical implementation, some further small optimizations may be taken. For example, within one iteration of the loop (5), we may also keep a set of orders that have already been found not to be $\mathcal G$-perceptive, so that especially the brute-force check of (15) is not performed unnecessarily. Similarly, we may along with each order $\mathcal L$ keep its $\mathcal L^1$ so that it is not computed several times. We have left these small details out of the statement of the algorithm above for the sake of conciseness.

\begin{prop}
The Algorithm~\ref{algo} is correct.
\end{prop}
\begin{proof}
Let $\mathcal G$ be a perceptive suborder of $\mathcal H$, let $[\mathcal H:\mathcal G]_{\OK} = \mathfrak p_1^{e_1}\cdots\mathfrak p_r^{e_r}$ and let
\[
    \mathcal G=:\mathcal G_0\subsetneq \cdots \subsetneq\mathcal G_r := \mathcal H
\]
be a chain of orders such that $[\mathcal G_{r+1-i}:\mathcal G_{r-i}]=\mathfrak p_i^{e_i}$.
To prove that $\mathcal G$ will be in the set returned by the algorithm, since $\mathcal G_r=\mathcal H$ is in the initial $\Omega$, it suffices prove that at the $i$-th iteration of the loop (3), after choosing $\mathcal G_{r+1-i}$ in (5), the order $\mathcal G_{r-i}$ will be found and added to $\Omega_{\text{new}}$ in this inner loop.

Between $\mathcal G_{r-i}$ and $\mathcal G_{r+1-i}$, we may construct a chain
\[
    \mathcal G_{r-i}:= \mathcal L_1\subsetneq\cdots\subsetneq \mathcal L_{\ell}:=\mathcal G_{r+1-i},
\]
where each pair of consecutive orders is as described in Proposition~\ref{prp:smallest-steps}. But then we see that starting with $\mathcal L_{\ell}=\mathcal G_{r+1-i}$, which is initially put in $\Gamma$, every time $\mathcal L_{j}$ is picked as $\mathcal F$ in (8), $\mathcal L_{j-1}$ is found in either (9) or (10), depending on whether $\mathcal L_{j-1}\subsetneq\mathcal L_j$ is as described by (i) or (ii) in Proposition~\ref{prp:smallest-steps}. Using Proposition~\ref{prp:composing-perceptivity}(ii) and (iii) repeatedly, we know $\mathcal G_{r-i}$ is $\mathcal G_{r+1-i}$-perceptive and so $\mathcal L_{j-1}\supseteq\mathcal G_{r-i}$ is $\mathcal G_{r+1-i}$-perceptive as well.
If step (12) takes place, that means $\mathcal H$-perceptivity of $\mathcal L_{j-1}$ has already been verified, so it is $\mathcal G_{r+1-i}$-perceptive as well by Proposition~\ref{prp:composing-perceptivity}. If (13) or (14) takes place, perceptivity is ascertained correctly due to Proposition~\ref{prp:perceptive-one-dim-step} or \ref{prp:perceptive-two-dim-step} respectively. Lastly, if (15) takes place, perceptivity is checked via Lemma~\ref{lem:search-tools}(iii).
So overall, $\mathcal L_{j-1}$ passes whichever test of perceptivity it is steered into, so it is added into $\Gamma$ in (16) and hence later into $\Omega_{\text{new}}$, provided it is not there already.

This finishes the proof that all perceptive suborders of $\mathcal H$ will be found by the algorithm. Conversely, any order added into $\Omega$ is checked to be a perceptive suborder of some order previously in $\Omega$ by one of (13), (14) or (15), so by induction and due to $\set{\mathcal H}$ being the initial value of $\Omega$, any order in the returned set is $\mathcal H$-perceptive by Proposition~\ref{prp:composing-perceptivity}.
\end{proof}

Running Algorithm~\ref{algo} on each of the $49$ maximal orders of class number $1$ from Theorem~\ref{thrm:kirschmer-list} a removing isomorphic copies, we obtain:
\begin{thm}
\label{thrm:our-list}
Up to isomorphism, there are $111$ orders in definite quaternion algebras over totally real number fields that are perceptive in some maximal order. The list is available from \orderlisturl.
\end{thm}

We used an implementation of Algorithm~\ref{algo} in Magma, which is also available at the page linked above.
\bigskip

In the following subsection, it will be useful to catalogue these orders based on the factorization of their indices in their respective maximal orders:
\begin{defn}
\label{def:kinds}
Let us say a suborder $\mathcal G$ of a maximal order $\mathcal H$ is \emph{of the kind $\mathfrak p_1^{e_1}\cdots \mathfrak p_a^{e_a}\mathfrak q_1^{f_1}\cdots\mathfrak q_b^{f_b}$}, if this is the factorization of $[\mathcal H:\mathcal G]_{\OK}$ into prime ideals with $\mathfrak p_1,\dots,\mathfrak p_a\mid \discrd\mathcal H$ and $\mathfrak q_1,\dots,\mathfrak q_b\nmid\discrd\mathcal H$. If $a=1$ or $b=1$, we will permit ourselves to omit the subscripts of $\mathfrak p$ and $\mathfrak q$ respectively.
\end{defn}
With this definition, we catalogue the orders whilst also checking their class numbers against the list of Theorem~\ref{thrm:kirschmer-list}.
\begin{prop}
\label{prp:kind-distribution}
Of the $111$ chosen representative orders $\mathcal G$ from Theorem~\ref{thrm:our-list} perceptive in a maximal order $\mathcal H$:
\begin{itemize}
\item $49$ are maximal,
\item $36$ are of the kind $\mathfrak q$,
\item $5$ are of the kind $\mathfrak p$,
\item $5$ are of the kind $\mathfrak q_1\mathfrak q_2$,
\item $11$ are of the kind $\mathfrak q^2$ and
\item $1$ is of each of the kinds $\mathfrak p\mathfrak q$, $\mathfrak p^2$, $\mathfrak q^3$, $\mathfrak p^3$, $\mathfrak q^4$.
\end{itemize}
Further, each of these orders has class number $1$. For those $\mathcal G$ of the kind $\mathfrak q^e$, the quotient $\mathcal H/\mathcal G$ is cyclic, that is $\mathcal H/\mathcal G\simeq\OK/\mathfrak q^e$; while all of those $\mathcal G$ of the kind $\mathfrak p^e$ have a linear poset of orders with $\mathcal H$.
\end{prop}

\subsection{Sizes of orbit intersections}
\label{sec:jacobi}

In this subsection, we shall go through the catalogue of Proposition~\ref{prp:kind-distribution} and provide a formula for the number of quaternions of a given reduced norm in each of the orders found. Let recall the notations $r_{\mathcal H}(\alpha)$ and $\sigma_{\mathfrak D}(\alpha)$ from Subsection~\ref{sec:counting-factorizations} and the fact that we proved the formula
\(
    r_{\mathcal H}(\alpha)=\#\mathcal H^1 \cdot \sigma_{\mathfrak D}(\alpha)
\)
when $\mathcal H$ is a maximal order of class number $1$.

The basic principle for all the proofs to come is in essence a restatement of Proposition~\ref{prp:conductors}(iii):
\begin{lemma}
\label{lem:orbit-intersections}
Suppose that the pair of orders $\mathcal G\subseteq\mathcal H$ is perceptive and has a linear poset of orders. Then for $\mathbf q\in\mathcal G$ such that $\mathcal M$ is the largest intermediate order $\mathcal G\subseteq\mathcal M\subseteq\mathcal H$ satisfying $\mathbf q\in\rcol{\mathcal G}{\mathcal M}$, we have $\#(\mathcal H^1\mathbf q\cap \mathcal G) = \#\mathcal M^1$.
\end{lemma}
\begin{proof}
    We know that for $\mathbf u\in\mathcal H^1$ and $\mathcal M:=\mathcal G+\mathbf u\mathcal G$, it holds that $\mathbf u\mathbf q\in \mathcal G$ if and only if $\mathbf q\in\rcol{\mathcal G}{\mathcal M}$. Thus belonging to the respective right conductors determines which units take $\mathbf q$ to an element of $\mathcal G$. But since the poset of orders is linear, the poset of the right conductors is also linear by Proposition~\ref{prp:conductors}(i), so the maximal one of them that contains $\mathbf q$ determines the order whose units land $\mathbf q$ back in $\mathcal G$.
\end{proof}

\begin{lemma}
\label{lem:orbit-and-conductor}
Let $\mathcal M$ be an $\mathcal H$-perceptive suborder and consider a left principal ideal $\mathcal M\mathbf a$ for some $\mathbf a\in\mathcal M$. Then for $\mathbf q\in\mathcal H$, the orbit $\mathcal H^1\mathbf q$ intersects $\mathcal M\mathbf a$ if and only if $\mathbf q\in\mathcal H\mathbf a$.
\end{lemma}
\begin{proof}
In one direction, if $\mathbf u \mathbf q = \mathbf m\mathbf a$ for some $\mathbf u\in\mathcal H^1$, $\mathbf m\in\mathcal M$, then $\mathbf q=(\inv{\mathbf u}\mathbf m)\mathbf a \in\mathcal H\mathbf a$. In the other direction, if $\mathbf q = \mathbf h\mathbf a$  for some $\mathbf h\in\mathcal H$, then by perceptivity, we find a $\mathbf u\in\mathcal H^1$ such that $\mathbf u\mathbf h\in\mathcal M$, so then $\mathbf u\mathbf q = (\mathbf u\mathbf h)\mathbf a \in \mathcal M\mathbf a$.
\end{proof}

Combining the two lemmata, we may say: if $\mathcal G$, $\mathcal H$ is a perceptive pair with a linear poset of orders and $\rcol{\mathcal G}{\mathcal M} =\mathcal M\mathbf a_{\mathcal M}$ for each of the intermediate orders, then $\#(\mathcal H^1\mathbf q\cap \mathcal G) = \#\mathcal M^1$, where $\mathcal M$ is the largest intermediate order such that $\mathbf q\in\mathcal H\mathbf a_{\mathcal M}$. We may also notice in this situation that if $\mathcal M_1\subseteq\mathcal M_2$, then $\mathbf a_{\mathcal M_1}$ divides $\mathbf a_{\mathcal M_2}$ from the right as elements of $\mathcal H$.

In the following theorem and others like it, we of course take only totally positive $\alpha$. Since some ideas in the proofs will repeat themselves, we will present these proofs more thoroughly in the beginning and then gradually more shortly, referring to repetitions of ideas from previous proofs. We will also illustrate some of the theorems with examples of particular quadratic forms over $\zet$ that such orders correspond to.
\begin{thm}
\label{thrm:kind-q-formula}
Let $\mathcal G$ be a perceptive suborder of the kind $\mathfrak q$ in a maximal order $\mathcal H$ with class number $1$ and $\discrd \mathcal H = \mathfrak D = \mathfrak p_1\cdots\mathfrak p_k$. Then
\begin{align}
\nonumber
    r_{\mathcal G}(\alpha) &= 2\#\mathcal G^1\cdot \sum_{\mathfrak p_1,\dots,\mathfrak p_k\nmid \delta\OK\mid \alpha\OK}\Nm(\delta\OK) - \#\mathcal G^1 \sum_{\mathfrak q,\mathfrak p_1,\dots,\mathfrak p_k\nmid \delta\OK\mid \alpha\OK}\Nm(\delta\OK) =\\
\label{eq:kind-q}
    &= 2\#\mathcal G^1 \sigma_{\mathfrak D}(\alpha) - \#\mathcal G^1\sigma_{\mathfrak q\mathfrak D}(\alpha).
\end{align}
\end{thm}
\begin{proof}
We have $[\mathcal H:\mathcal G]_{\OK}=\mathfrak q$, hence $\mathcal H/\mathcal G\simeq\OK/\mathfrak q$, so $\mathcal G$, $\mathcal H$ trivially has a linear poset of orders. Further, perceptivity implies that $\#\mathcal H^1 = \#\mathcal G^1\cdot(\Nm(\mathfrak q)+1)$.

We have that $\rcol{\mathcal G}{\mathcal H}$ is a left ideal of $\mathcal H$, which is a principal ideal domain (being a maximal order of class number $1$), so $\rcol{\mathcal G}{\mathcal H} = \mathcal H\mathbf a $ for some $\mathbf a\in \rcol{\mathcal G}{\mathcal H}$. Further, we know that $[\mathcal H:\rcol{\mathcal G}{\mathcal H}]_{\OK} = \mathfrak q^2$ and $[\mathcal H:\mathcal H\mathbf a ]_{\OK} = \nrd(\mathbf a)^2\OK$, hence we must have $\nrd(\mathbf a)\OK = \mathfrak q$.
Additionally, $\rcol{\mathcal G}{\mathcal G} = \mathcal G = \mathcal G\cdot 1$ trivially. So applying the consequence of the two lemmata above, we see that
\[
    \#(\mathcal H^1\cap \mathcal G) = \begin{cases}
        \#\mathcal H^1, & \text{if $\mathbf q\in\mathcal H\mathbf a$,}\\
        \#\mathcal G^1, & \text{otherwise}
    \end{cases}
\]
for $\mathbf q\in\mathcal H$.

Denote $\pi:=\nrd(\mathbf a)\in\OK^+$, a generator of $\mathfrak q$, and $q:=\Nm(\pi)=\Nm(\mathfrak q)$.
We will consider an $\alpha\in\OK^+$ and factorize it as $\alpha=\beta\pi^e$ for some $e\geq0$ and $\beta\in\OK^+$ with $\pi\nmid\beta$. Let us count how many of the $\sigma_{\mathfrak D}(\alpha) = \sigma_{\mathfrak D}(\beta)\sigma_{\mathfrak D}(\pi^e)$ orbits with reduced norm $\alpha$ that exist in $\mathcal H$ are contained in $\mathcal H\mathbf a$.
By Proposition~\ref{prp:coprime-factorization}, we may consider quaternions of reduced norm $\alpha$ in $\mathcal H$ factorized in the form $\mathbf b\mathbf q$ with $\nrd(\mathbf b)=\beta$, $\nrd(\mathbf q)=\pi^e$. Due to the coprime reduced norms, we have $\mathbf b\mathbf q\in\mathcal H\mathbf a$ if and only if $\mathbf q\in\mathcal H\mathbf a$ (e.g. because $\mathbf b$ will be invertible is some $\mathcal H/\pi^a\mathcal H$ such that $\pi^a\mathcal H\subseteq\mathcal H\mathbf a$), so it suffices to count $r_{\mathcal G}(\pi^e)$ and then multiply from the left by $\sigma_{\mathfrak D}(\beta)$.

Now, orbits $\mathcal H^1\mathbf q$ of quaternions of reduced norm $\pi^e$ that lie in $\mathcal H\mathbf a$ are exactly of the form $\mathcal H\mathbf q_0\mathbf a$ for $\nrd(\mathbf q_0)=\pi^{e-1}$, so there are as many of them as orbits of reduced norm $\pi^{e-1}$, that is $(1+q+\cdots+q^{e-1})$  for $e\geq 1$ and $0$ otherwise. The remaining $q^e$ orbits then do not lie in $\mathcal H\mathbf a$. Thus we calculate for $e\geq 1$ that
\begin{align*}
    r_{\mathcal G}(\pi^e) &= (1+q+\cdots+q^{e-1})\#\mathcal H^1 + q^e\#\mathcal G^1 =\\&= \#\mathcal G^1\zav{(1+q+\cdots+q^{e-1})(q+1)+q^e} =\\&= \#\mathcal G^1\zav{ 2(1+q+\cdots+q^e)- 1} = 2\#\mathcal G^1\sigma_{\mathfrak D}(\pi^e) - \#\mathcal G^1\sigma_{\mathfrak q\mathfrak D}(\pi^e),
\end{align*}
and a posteriori we observe that this also happens to be true for $e=0$.
Hence multiplying by $\sigma_{\mathfrak D}(\beta)$ and using fact that $\sigma_{\mathfrak D}(\beta) = \sigma_{\mathfrak q\mathfrak D}(\beta)$ due to $\beta\notin\mathfrak q$, we obtain $r_{\mathcal G}(\alpha) = \#\mathcal G^1\cdot(2\sigma_{\mathfrak D}(\alpha) - \sigma_{\mathfrak q\mathfrak D}(\alpha))$ as desired.
\end{proof}

Note that just by rearranging the right-hand side of \eqref{eq:kind-q}, it may also be written as
\[
    r_{\mathcal G}(\alpha) = \#\mathcal G^1\sum_{\mathfrak p_1,\dots,\mathfrak p_k\nmid \delta\OK\mid \alpha\OK}\Nm(\delta\OK) + \#\mathcal G^1 \sum_{\substack{\mathfrak p_1,\dots,\mathfrak p_k\nmid \delta\OK\mid \alpha\OK\\\mathfrak q\mid \delta\OK}}\Nm(\delta\OK).
\]

\begin{example}
    Inside the Hurwitz order $\mathcal H:=\zet\oplus\zet\ii\oplus\zet\j\oplus\zet\frac{1+\ii+\j+\k}2$, a maximal order of class number $1$ and reduced discriminant $2\zet$, the suborder $\mathcal G:= \zet\oplus\zet11\ii\oplus\zet(\j-\ii)\oplus\zet\frac{1+7\ii+\j+\k}2$ is perceptive and of the kind $11\zet$, whilst having $\#\mathcal G^1=2$. So writing out the quadratic form $(\mathcal G,\nrd)$ and applying the proposition, we see that for a positive integer $n$, the equation
    \[
        t^2 + 121x^2+2y^2+ 13z^2 + tz - 22xy + 77xz - 6yz = n
    \]
    has exactly
    \(\displaystyle
        4\sum_{2\nmid d\mid n}d - 2\sum_{11,2\nmid d\mid n}d
    \)
    solutions.
\end{example}

\begin{thm}
\label{thrm:kind-p-formula}
Let $\mathcal G$ be a perceptive suborder of the kind $\mathfrak p$ in a maximal order $\mathcal H$ with class number $1$ and $\discrd \mathcal H = \mathfrak D = \mathfrak p\mathfrak p_2\cdots\mathfrak p_k$. Then
\begin{equation}
\label{eq:kind-p}
    r_{\mathcal G}(\alpha) = \#\mathcal G^1 \sum_{\mathfrak p^2,\mathfrak p_2,\dots,\mathfrak p_k\nmid \delta\OK\mid \alpha\OK}\Nm(\delta\OK).
\end{equation}
\end{thm}
\begin{proof}
The proof of this proposition follows mostly similarly to the previous one, so let us simply treat the few differences and leave out what would be merely a repetition.

The characterization of the sizes of intersections of orbits with $\mathcal G$ is analogous: we have $\rcol{\mathcal G}{\mathcal H}=\mathcal H\mathbf a$ with $\nrd(\mathbf a) = \pi$ being a generator of $\mathfrak p$. The discussion now turns rather trivial though, because for all $e$, there is only one orbit of quaternions of reduced norm $\pi^e$, hence $\mathbf q\in \mathcal H\mathbf a$ if and only if $\pi\mid\nrd(\mathbf q)$.

So, splitting the orbits based on the size of their intersections with $\mathcal G$ again, we just get
\begin{align*}
r_{\mathcal G}(\alpha) = \begin{cases}
    \#\mathcal H^1 \sigma_{\mathfrak D}(\alpha), & \text{if $\pi\mid\alpha$},\\
    \#\mathcal G^1 \sigma_{\mathfrak D}(\alpha), & \text{otherwise}.
\end{cases}
\end{align*}
Then we finish by noticing that this agrees with \eqref{eq:kind-p} simply by the virtue of $\#\mathcal H^1 = (\Nm(\mathfrak p)+1)\#\mathcal G^1$, the looser condition $\mathfrak p^2\nmid\delta\OK$ exactly adding a factor of $\Nm(\mathfrak p)+1$ for $\alpha$'s divisible by $\pi$ compared to the stricter $\mathfrak p\nmid\delta\OK$.
\end{proof}

\begin{example}
 Inside the Hurwitz order $\mathcal H:=\zet\oplus\zet\ii\oplus\zet\j\oplus\zet\frac{1+\ii+\j+\k}2$, a maximal order of class number $1$ and reduced discriminant $2\zet$, the Lipschitz order $\mathcal G:= \zet\oplus\zet\ii\oplus\zet\j\oplus\zet\k$ is perceptive and of the kind $2\zet$, having $\#\mathcal G^1=8$. Hence we see that
    \[
        t^2 + x^2 + y^2+ z^2= n
    \]
    has exactly
    \(\displaystyle
        8\sum_{4\nmid d\mid n}d
    \)
    solutions, i.e. we have recovered Jacobi's four-square theorem.
\end{example}

\begin{thm}
\label{thrm:kind-q2-formula}
    Let $\mathcal G$ be a perceptive suborder of class number $1$ in a maximal order $\mathcal H$ such that $\mathcal H/\mathcal G\simeq \OK/\mathfrak q^2$ and let $\mathcal H$ have $\mathfrak q\nmid\discrd \mathcal H = \mathfrak D = \mathfrak p_1\cdots\mathfrak p_k$. If $q:=\Nm(\mathfrak q)$ and $\alpha\in\mathfrak q^e$ but $\alpha\notin\mathfrak q^{e+1}$ for some $e\geq0$, then
\begin{equation}
\label{eq:kind-q2}
    r_{\mathcal G}(\alpha) = \#\mathcal G^1\sum_{\mathfrak q,\mathfrak p_1,\dots,\mathfrak p_k\nmid\delta\OK\mid\alpha\OK}\Nm(\delta\OK) \cdot \begin{cases}
        1, & \text{$e=0$,}\\
        2q, & \text{$e=1$,}\\
        2(q+\cdots+q^e)+q^e-q,
        &\text{$e\geq2$.}
    \end{cases}
\end{equation}
\end{thm}
\begin{proof}
Due to the cyclic quotient, we have the linear poset of orders $\mathcal G\subset\mathcal M\subset\mathcal H$, where $\mathcal M := \mathcal G+\mathfrak q\mathcal H$, and perceptivity implies $\#\mathcal M^1 = q \#\mathcal G^1$ and $\#\mathcal H^1 = {q(q+1)\#\mathcal G^1}$. Let us start by identifying the conductors.

Let $\pi\in\OK^+$ be a generator of $\mathfrak q$.
Trivially we have $\rcol{\mathcal G}{\mathcal G} = \mathcal G$ and $\rcol{\mathcal G}{\mathcal H} = \mathcal H\mathbf a_2$, where considering indices and reduced norms yields $\nrd(\mathbf a_2)\OK = \mathfrak q^2$.
Next we claim that $\rcol{\mathcal G}{\mathcal M} = \mathcal M \mathbf a_1$ for some $\mathbf a_1\in \rcol{\mathcal G}{\mathcal M}$. For that, denote $\lord(\rcol{\mathcal G}{\mathcal M}) =: \mathcal L$, this must contain $\mathcal M$ since $\rcol{\mathcal G}{\mathcal M}$ is a left ideal of $\mathcal M$. By Proposition~\ref{prp:cls-implications}, we have $\#\Cls\mathcal L = 1$, so then $\rcol{\mathcal G}{\mathcal M} = \mathcal L \mathbf a_1$ for some $\mathbf a_1\in\rcol{\mathcal G}{\mathcal M}$. Since $\mathcal M$ is an index-$\mathfrak q$ suborder of a maximal order, either $\mathcal L=\mathcal M$ or $[\mathcal L:\mathcal M]_{\OK}=\mathfrak q$. But the latter would lead to
\[
\mathfrak q^2 = [\mathcal M:\rcol{\mathcal G}{\mathcal M}]_{\OK} = \inv{\mathfrak q} \cdot \nrd(\mathbf a_1)^2\OK,
\]
which is absurd, because the ideal on the left is a square and the one on the right is not. So we indeed have $\mathcal L=\mathcal M$ and $\rcol{\mathcal G}{\mathcal M}=\mathcal M\mathbf a_1$. Considering indices then gives $\nrd(\mathbf a_1) = \mathfrak q$.

Possibly multiplying by a unit from $\OK$ (here we use $\NCl K = 1$ due to Proposition~\ref{prp:cls-implications}), we may presume that $\nrd(\mathbf a_1)=\pi$, $\nrd(\mathbf a_2)=\pi^2$. As in the proof of Theorem~\ref{thrm:kind-q-formula}, we may consider quaternions of reduced norm $\alpha=\beta\pi^e$ as factorized into $\mathbf b\mathbf q$, $\nrd(\mathbf b)=\beta$, $\nrd(\mathbf q)=\pi^e$ and the choice of $\mathbf b$ is then irrelevant for whether the orbit of such a quaternion belongs to either of the two $\mathcal H\mathbf a_i$. So we only count $r_{\mathcal G}(\pi^e)$ and then multiply by $\sigma_{\mathfrak D}(\beta)$.

Now, counting orbits $\mathcal H^1\mathbf q$ of quaternions of reduced norm $\pi^e$ in $\mathcal H\mathbf a_i$ is equivalent to counting the orbits $\mathcal H^1\mathbf q\inv{\mathbf a_i}$ of reduced norm $\pi^{e-i}$ in $\mathcal H$, which yields $1+q+\cdots+q^{e-i}$ for $e\geq i$. Let us deal with small cases of $e$ first; if $e=0$, tautologically we obtain $r_{\mathcal G}(1)=\#\mathcal G^1$. Then, for $e=1$, of the $1+q$ orbits of reduced norm $\pi$, one lies in $\mathcal H\mathbf a_1$ (that is, $\mathcal H^1\mathbf a_1$ itself), none of them lies in $\mathcal H\mathbf a_2$ and the remaining $q$ are only  in $\mathcal H$. So we count
\[
    r_{\mathcal G}(\pi) = \#\mathcal M^1+ q\#\mathcal G^1 = \#\mathcal G^1(q+q)
\]
as desired for \eqref{eq:kind-q2}.

Next we deal with the general case $e\geq2$. Then $1+q+\cdots+q^{e-2}$ orbits lie in $\mathcal H\mathbf a_2$, leaving $(1+q+\cdots+q^{e-1})-(1+q+\cdots+q^{e-2}) = q^{e-1}$ of them in $\mathcal H\mathbf a_1$ and $(1+q+\cdots+q^{e}) - (1+q+\cdots+q^{e-1})=q^e$ merely in $\mathcal H$. Thus we calculate
\begin{align*}
    r_{\mathcal G}(\pi^e) &= (1+q+\cdots+q^{e-2})\#\mathcal H^1 + q^{e-1}\#\mathcal M^1 + q^e\#\mathcal G^1 = \\
    &= \#\mathcal G^1 \zav{(1+q+\cdots+q^{e-2})\cdot q(q+1)  + q^{e-1}\cdot q + q^e} =\\
    &= \#\mathcal G^1 \zav{q+2q^2+\cdots+2q^{e-1}+3q^e} =\\
    &= \#\mathcal G^1\zav{2(q+\cdots+q^e)+q^e-q}
\end{align*}
as desired. Multiplying by $\sigma_{\mathfrak D}(\beta) = \sigma_{\mathfrak q\mathfrak D}(\alpha)$ as justified above, we then obtain \eqref{eq:kind-q2}.
\end{proof}

\begin{example}
    Inside the maximal order $\mathcal H:=\zet\oplus\zet\frac{1+\ii}2\oplus\zet\j\oplus\zet\frac{\j+\k}2$ (which has class number $1$) of $\quatalg{-3,-1}\kve$, the suborder $\mathcal G:=\zet\oplus\zet2\ii\oplus\zet\zav{\frac{1-\ii}2+\j}\oplus\zet\frac{1+\ii+\j+\k}2$ is perceptive and of the kind $(2\zet)^2$, having $\#\mathcal G^1=2$. Hence for any positive integer $n$ with $2$-adic valuation $e$, the equation
    \[
        t^2+12x^2+2y^2+2z^2 +ty+tz-6xy+6xz = n
    \]
    has its number of solutions given exactly by $\displaystyle2\sum_{2,3\nmid d\mid n} d \cdot\begin{cases}
        1, & \text{$e=0$,}\\
        4, & \text{$e=1$,}\\
        2^{e+2}+2^e-6
        &\text{$e\geq2$.}
    \end{cases}$
\end{example}

    \begin{thm}
\label{thrm:kind-p2-formula}
    Let $\mathcal G$ be a perceptive suborder of the kind $\mathfrak p^2$ and class number $1$ in a maximal order $\mathcal H$ with a linear poset of orders $\mathcal G\subsetneq\mathcal M\subsetneq\mathcal H$ and $\discrd \mathcal H = \mathfrak D = \mathfrak p\mathfrak p_2\cdots\mathfrak p_k$. If $q:=\Nm(\mathfrak p)$ and $\alpha\in\mathfrak p^e$ but $\alpha\notin\mathfrak p^{e+1}$ for some $e\geq0$, then
\begin{equation}
\label{eq:kind-p2}
    r_{\mathcal G}(\alpha) = \#\mathcal G^1\sum_{\mathfrak p,\mathfrak p_2,\dots,\mathfrak p_k\nmid\delta\OK\mid\alpha\OK}\Nm(\delta\OK) \cdot \begin{cases}
        1, & \text{$e=0$,}\\
        q, & \text{$e=1$,}\\
        q^2+q, &\text{$e\geq2$.}
    \end{cases}
\end{equation}
\end{thm}
\begin{proof}
Again, the proof is similar to that of Theorem~\ref{thrm:kind-q2-formula}, but simpler due to the nature of quaternions of reduced norm $\mathfrak p^e$, so we just highlight the differences.

The proof that $\rcol{\mathcal G}{\mathcal M}$ is a principal left ideal of $\mathcal M$ goes through in the same way as before, so then $\#(\mathcal H^1\mathbf q\cap \mathcal G)$ is determined by the belonging or non-belonging of $\mathbf q$ to some $\mathcal H\mathbf a_2$, $\mathcal H\mathbf a_1$ or just $\mathcal H$ with $\nrd(\mathbf a_i)\OK = \mathfrak p^2$. Since $\NCl K= 1$, we may choose a totally positive generator $\pi$ of $\mathfrak p$ and take $\nrd(\mathbf a_i)=\pi^i$. But since there is only one orbit of reduced norm $\pi^e$ for each $e\geq0$, we have $\mathbf q\in\mathcal H\mathbf a_i$ if and only if $\pi^i\mid\nrd(\mathbf q)$.

Thus we obtain
\[
    r_{\mathcal G}(\pi^e) = \left\{\begin{array}{ll}
        \#\mathcal G^1, & \text{$e=0$,}\\
        \#\mathcal M^1, & \text{$e=1$,}\\
        \#\mathcal H^1, & \text{$e\geq2$}
    \end{array}\right\} = \#\mathcal G^1\cdot\begin{cases}
        1, & \text{$e=0$,}\\
        q, & \text{$e=1$,}\\
        q(q+1), & \text{$e\geq2$}
    \end{cases}
\]
and multiplying by $\sigma_{\mathfrak D}(\beta)$ with $\beta$ from the factorization $\alpha=\beta\pi^e$, $\pi\nmid \beta$ yields \eqref{eq:kind-p2}.
\end{proof}

\begin{thm}
\label{thrm:kind-qq-formula}
Let $\mathcal G$ be a perceptive suborder of the kind $\mathfrak q_1\mathfrak q_2$ and class number $1$ in a maximal order $\mathcal H$ with $\discrd \mathcal H=\mathfrak D=\mathfrak p_1\cdots\mathfrak p_k$. Then
\begin{multline}
\label{eq:kind-qq}
r_{\mathcal G}(\alpha) = 4\#\mathcal G^1\sum_{\mathfrak p_1,\dots,\mathfrak p_k\nmid\delta\OK\mid\alpha\OK}\Nm(\delta\OK) - 2\#\mathcal G^1\sum_{\mathfrak q_1,\mathfrak p_1,\dots,\mathfrak p_k\nmid\delta\OK\mid\alpha\OK}\Nm(\delta\OK) -{}\\{}- 2\#\mathcal G^1\sum_{\mathfrak q_2,\mathfrak p_1,\dots,\mathfrak p_k\nmid\delta\OK\mid\alpha\OK}\Nm(\delta\OK) + \#\mathcal G^1\sum_{\mathfrak q_1,\mathfrak q_2,\mathfrak p_1,\dots,\mathfrak p_k\nmid\delta\OK\mid\alpha\OK}\Nm(\delta\OK) =\\
= 4\#\mathcal G^1\sigma_{\mathfrak D}(\alpha) - 2\#\mathcal G^1\sigma_{\mathfrak q_1\mathfrak D}(\alpha) - 2\#\mathcal G^1\sigma_{\mathfrak q_2\mathfrak D}(\alpha) + \#\mathcal G^1\sigma_{\mathfrak q_1\mathfrak q_2\mathfrak D}(\alpha).
\end{multline}
\end{thm}
\begin{proof}
First, let us investigate the poset of orders of $\mathcal G$, $\mathcal H$. Since $[\mathcal H:\mathcal G]_{\OK}=\mathfrak q_1\mathfrak q_2$, the only quotient module $\mathcal H/\mathcal G$ for such an index is the cyclic module $\OK/\mathfrak q_1\mathfrak q_2$, which decomposes as $\OK/\mathfrak q_1\times \OK/\mathfrak q_2$. Thus there are only two intermediate modules (corresponding to $\OK/\mathfrak q_1\times\set0$ and $\set0\times\OK/\mathfrak q_2$), which are also modules because we express them as $\mathcal M_1 := \mathcal G + \mathfrak q_2 \mathcal H$, $\mathcal M_2 := \mathcal G + \mathfrak q_1\mathcal H$; this is so that $[\mathcal M_i:\mathcal G]_{\OK}=\mathfrak q_i$.

Since both pairs $\mathcal G$, $\mathcal M_1$ and $\mathcal G$, $\mathcal M_2$ have a linear poset of orders, it follows that $[\mathcal G:\rcol{\mathcal G}{\mathcal M_i}]_{\OK} = [\mathcal M_i:\mathcal G]_{\OK}$. Further, $(1,1)$ in $\OK/\mathfrak q_1\times\OK/\mathfrak q_2$ is not contained in either of the submodules $\OK/\mathfrak q_1\times\set0$ or $\set0\times\OK/\mathfrak q_2$, so this corresponds to some $\mathbf z\in\mathcal H$ such that $\mathcal H=\mathcal G+\mathcal G\mathbf z = \mathcal G+\mathbf z\mathcal G$, with which one easily shows that $[\mathcal G:\rcol{\mathcal G}{\mathcal H}]_{\OK} = [\mathcal H:\mathcal G]_{\OK}$ just as in the proof of Proposition~\ref{prp:conductors}(i). As in that proof, we also have that $\mathcal M\mapsto \rcol{\mathcal G}{\mathcal M}$ reverses inclusions, so thanks to the indices, it also preserves strict inclusions. Finally, due to $[\mathcal G:\rcol{\mathcal G}{\mathcal M_i}]=\mathfrak q_i$ being comaximal, the two conductors $\rcol{\mathcal G}{\mathcal M_i}$ must be incomparable in inclusion. So altogether, we have shown that the collection of conductors $\set{\rcol{\mathcal G}{\mathcal M}\mid\mathcal G\subseteq\mathcal M\subseteq\mathcal H}$ forms an opposite poset to the poset of orders. In other words, we have established the conclusion of Proposition~\ref{prp:conductors}(i) for the poset of intermediate orders between $\mathcal G$ and $\mathcal H$. Since parts (ii) and (iii) of that Proposition only relied on the conclusion of (i) and not other conditions of the Proposition directly, we may now use (ii) and (iii) in our situation.

Now, arguing with indices $[\mathcal M_i:\rcol{\mathcal G}{\mathcal M_i}]_{\OK} = \mathfrak q_i^2$ and parity of exponents as in the proof of Theorem~\ref{thrm:kind-q2-formula}, we obtain $\rcol{\mathcal G}{\mathcal M_i} = \mathcal M_i\mathbf a_i$ with some $\nrd(\mathbf a_i)\OK = \mathfrak q_i\OK$. Further, we obtain $\rcol{\mathcal G}{\mathcal H} = \mathcal H\mathbf a$ for some $\mathbf a$ of reduced norm $\nrd(\mathbf a)\OK = \mathfrak q_1\mathfrak q_2$. If we choose totally positive generators $\pi_1$, $\pi_2$ of $\mathfrak q_1$, $\mathfrak q_2$ respectively, then we may without loss of generality take $\nrd(\mathbf a_i)=\pi_i$, $\nrd(\mathbf a)=\pi_1\pi_2$. Additionally, let us denote $q_i:=\Nm(\mathfrak q_i)$.

The size of $\mathcal H^1\mathbf q\cap\mathcal G$ is now determined by which of the conductors
\[
\rcol{\mathcal G}{\mathcal G} = \mathcal G,
\quad
\rcol{\mathcal G}{\mathcal M_1} = \mathcal M_1\mathbf a_1,
\quad
\rcol{\mathcal G}{\mathcal M_2} = \mathcal M_2\mathbf a_2,
\quad
\rcol{\mathcal G}{\mathcal H} = \mathcal H\mathbf a,
\]
the orbit intersects, hence we obtain
\[
    \#(\mathcal H^1\mathbf q\cap\mathcal G) = \begin{cases}
        \#\mathcal H^1, & \text{if $\mathbf q\in\mathcal H\mathbf a$,}\\
        \#\mathcal M_1^1, & \text{if $\mathbf q\in\mathcal H\mathbf a_1\setminus\mathcal H\mathbf a$,}\\
        \#\mathcal M_2^1, & \text{if $\mathbf q\in\mathcal H\mathbf a_2\setminus\mathcal H\mathbf a$,}\\
        \#\mathcal G^1, & \text{if $\mathbf q\in\mathcal H\setminus(H\mathbf a_1\cup H\mathbf a_2)$.}
    \end{cases}
\]
All of these conditions are unchanged when multiplying from the left by some $\mathbf b$ with $\nrd(\mathbf b)\notin\mathfrak q_1,\mathfrak q_2$, so as we have seen a number of times so far, it suffices that we count $r_{\mathcal G}(\pi_1^{e_1}\pi_2^{e_2})$ and then multiply from the left by the number of orbits of some reduced norm $\beta\notin\mathfrak q_1,\mathfrak q_2$ taken from a factorization $\alpha=\beta\pi_1^{e_1}\pi_2^{e_2}$, i.e. by $\sigma_{\mathfrak D}(\beta)$.

Now, considering quaternions of reduced norm $\pi_1^{e_1}\pi_2^{e_2}$ in $\mathcal H$, we see that $\mathcal H\mathbf a$ contains exactly $\sigma_{\mathfrak D}(\pi_1^{e_1-1}\pi_2^{e_2-1})$ of their orbits. Similarly, $\mathcal H\mathbf a_1$ contains $\sigma_{\mathfrak D}(\pi_1^{e_1-1}\pi_2^{e_2})$, so
\[
  \sigma_{\mathfrak D}(\pi_1^{e_1-1})\sigma_{\mathfrak D}\pi_2^{e_2}) -
  \sigma_{\mathfrak D}(\pi_1^{e_1-1})\sigma_{\mathfrak D}\pi_2^{e_2-1}) =   \sigma_{\mathfrak D}(\pi_1^{e_1-1})\cdot q_2^{e_2}
\]
are in $\mathcal H\mathbf a_1 \setminus\mathcal H\mathbf a$. Similarly, $\mathcal H\mathbf a_2 \setminus\mathcal H\mathbf a$ contains $q_1^{e_1}\sigma_{\mathfrak D}(\pi_2^{e_2-1})$ orbits. Subtracting from $\sigma_{\mathfrak D}(\pi_1^{e_1}\pi_2^{e_2})$ all orbits accounted for so far, we get $q_1^{e_1}q_2^{e_2}$ orbits that only lie in $\mathcal H$. Note that all this counting is sensible even for $e_1=0$ or $e_2=0$ if we interpret
\[
\sigma_{\mathfrak D}(\pi_i^{e_i}) = 1+q_i+\cdots+q_i^{e_i} = \frac{q_i^{e_i+1}-1}{q_i-1}
\]
and hence $\sigma_{\mathfrak D}(\pi_i^{0-1}) = \frac{q_i^{-1+1}-1}{q_i-1} = 0$, which we do.

Weighing this and considering that $\#\mathcal M_i^1 = (q_i+1)\#\mathcal G^1$ and $\#\mathcal H^1={(q_1+1)}\allowbreak(q_2+1)\#\mathcal G^1$, we obtain
\begin{align*}
r_{\mathcal G}(\pi_1^{e_1}\pi_2^{e_2}) &= \sigma_{\mathfrak D}(\pi_1^{e_1-1}\pi_2^{e_2-1})\#\mathcal H^1 + \sigma_{\mathfrak D}(\pi_1^{e_1-1}) q_2^{e_2}\#\mathcal M_1^1 + q_1^{e_1}\sigma_{\mathfrak D}(\pi_2^{e_2-1})\#\mathcal M_2^1 +{}\\&\hskip4em{}+ q_1^{e_1}q_2^{e_2}\#\mathcal G^1 =\\
&= \#\mathcal G^1\Bigl(
    \sigma_{\mathfrak D}(\pi_1^{e_1-1}\pi_2^{e_2-1})(q_1+1)(q_2+1) + \sigma_{\mathfrak D}(\pi_1^{e_1-1}) q_2^{e_2}(q_1+1) +{}\\&\hskip4em{}+ q_1^{e_1}\sigma_{\mathfrak D}(\pi_2^{e_2-1})(q_2+1) + q_1^{e_1}q_2^{e_2}\Bigr) =\\
&= \#\mathcal G^1\zav{\sigma_{\mathfrak D}(\pi_1^{e_1-1})(q_1+1) + q_1^{e_1}}\zav{\sigma_{\mathfrak D}(\pi_2^{e_2-1})(q_2+1) + q_2^{e_2}} =\\
&= \#\mathcal G^1\zav{2\sigma_{\mathfrak D}(\pi_1^{e_1})-1}\zav{2\sigma_{\mathfrak D}(\pi_2^{e_2})-1}.
\end{align*}
If we now interpret the solitary $-1$'s in the two parentheses as $\sigma_{\mathfrak q_i\mathfrak D}(\pi_i^{e_i})$ and multiply out, we get an expression corresponding to \eqref{eq:kind-qq} if $\alpha=\pi_1^{e_1}\pi_2^{e_2}$.
Multiplying by $\sigma_{\mathfrak D}(\beta) $ for $\alpha=\beta\pi_1^{e_1}\pi_2^{e_2}$ then yields the Proposition in its full statement.
\end{proof}

\begin{thm}
\label{thrm:kind-pq-formula}
Let $\mathcal G$ be a perceptive suborder of the kind $\mathfrak p\mathfrak q$ and class number $1$ in a maximal order $\mathcal H$ with $\discrd \mathcal H=\mathfrak D=\mathfrak p\mathfrak p_2\cdots\mathfrak p_k$. Then
\begin{equation}
\label{eq:kind-pq}
r_{\mathcal G}(\alpha) = 2\#\mathcal G^1\cdot \sum_{\mathfrak p^2,\mathfrak p_2,\dots,\mathfrak p_k\nmid \delta\OK\mid \alpha\OK}\Nm(\delta\OK) - \#\mathcal G^1 \sum_{\mathfrak q,\mathfrak p^2,\mathfrak p_2,\dots,\mathfrak p_k\nmid \delta\OK\mid \alpha\OK}\Nm(\delta\OK).
\end{equation}
\end{thm}
\begin{proof}
We argue in the same fashion as in Theorem~\ref{thrm:kind-qq-formula}, but just like in Theorem~\ref{thrm:kind-p-formula} the prime $\mathfrak p\mid\mathfrak D$ only contributes a factor of $\Nm(\mathfrak p)+1$ to those $\alpha$'s contained in $\mathfrak p$, which then may be interpreted as loosening a condition $\mathfrak p\nmid \delta\OK$ to $\mathfrak p^2\nmid\delta\OK$.
\end{proof}
\begin{example}
    Inside the Hurwitz order, we found a perceptive suborder $\mathcal G:=\zet\oplus\zet3\ii\oplus\zet(\ii+\j)\oplus\zet(\ii+\k)$ of the kind $(2\zet)(3\zet)$, having $\#\mathcal G^1=2$. Hence we see that for a positive integer $n$, the equation
    \[
        t^2 + (3x+y+z)^2+y^2+z^2 = n
    \]
    has exactly $\displaystyle 4\sum_{4\nmid d\mid n}d - 2\sum_{3,4\nmid d\mid n}d$ solutions.
\end{example}

Reviewing Proposition~\ref{prp:kind-distribution}, the only orders not covered by our results so far are
\[
    \mathcal G_{\mathfrak p^3} = \zet\oplus\zet2\ii\oplus\zet2\j\oplus\zet(\ii+\k)
\]
inside the Hurwitz order in the algebra $\quatalg{-1,-1}\kve$ and
\begin{align*}
    \mathcal G_{\mathfrak q^3} &=\OK\oplus\OK2\ii\oplus\OK\frac{\ii+\j}{\sqrt2}\oplus\OK\frac{(1+\sqrt2)+(\sqrt2-1)\ii+\j+\k}2,\\
    \mathcal G_{\mathfrak q^4} &= \OK\oplus\OK2\sqrt2\ii\oplus\OK\zav{2\ii+\frac{\ii+\j}{\sqrt2}}\oplus\OK\frac{(1+\sqrt2)+(3+\sqrt2)\ii+\j+\k}2,
\end{align*}
both inside the so-called \emph{cubian} order \[\KK = \OK\oplus\OK\frac1{\sqrt2}(1+\ii)\oplus\OK\frac1{\sqrt2}(1+\j)\oplus\OK\frac{1+\ii+\j+\k}2\] in the algebra $\quatalg{-1,-1}K$ over $K=\kve(\sqrt2)$. Note that cubians were used by Deutsch \cite{deutsch2} to prove a theorem on sums of four squares over $\kve(\sqrt2)$ originally due to Cohn.

Let us first focus on $\mathcal G:=\mathcal G_{\mathfrak p^3}$ inside the Hurwitz order. Its maximal order is the Hurwitz order $\mathcal H=\zet\oplus\zet\ii\oplus\zet\j\oplus\zet\frac{1+\ii+\j+\k}2$ and with it, it has a linear poset of orders consisting of the Lipschitz order $\mathcal L=\zet\oplus\zet\ii\oplus\zet\j\oplus\zet\k$ and the order $\mathcal M=\zet\oplus2\ii\oplus\zet\j\oplus\zet(\ii+\k)$. To mimic the proof of Theorem~\ref{thrm:kind-p2-formula}, let us show that the right conductor of each of these orders in $\mathcal G$ is a left principal ideal of its respective order. For $\mathcal H$ and $\mathcal L$, we can do this in the same way as in previous theorems, since these orders are a maximal order and an index-$2$ suborder of a maximal order respectively. For $\mathcal M$, we do it explicitly, claiming that
\[
    \rcol{\mathcal G}{\mathcal M} = \mathcal M(\ii+\k).
\]
On one hand, since $\mathcal M = \mathcal G+\mathcal G\j$ and $\j(\ii+\k) = -\k+\ii\in\mathcal G$, we see that $\mathcal M(\ii+\k)\subseteq\mathcal G$, hence $\ii+\k\in\rcol{\mathcal G}{\mathcal M}$ and so $\mathcal M(\ii+\k)\subseteq\rcol{\mathcal G}{\mathcal M}$ because the latter is a left ideal of $\mathcal M$. On the other hand we know that
\[
[\mathcal M:\rcol{\mathcal G}{\mathcal M}]_{\zet} = [\mathcal M:\mathcal G]_{\zet}^2 = 4\zet,
\]
so since $[\mathcal M:\mathcal M(\ii+\k)]_{\zet} = \nrd(\ii+\k)^2\zet = 4\zet$, equality must occur in $\mathcal M(\ii+\k)\subseteq\rcol{\mathcal G}{\mathcal M}$.

Now the rest of the discussion of orbit intersections goes exactly as in Theorem~\ref{thrm:kind-p2-formula}: the intersection size of an orbit depends on which of the conductors it intersects, that in turn depends on which corresponding left ideal of $\mathcal H$ it lies in, but because there is only one orbit of irreducible quaternions of reduced norm $2^e$ in $\mathcal H$ for all $e\geq0$, this only depend on the $2$-adic valuation of the reduced norm. Depending on this, the intersection sizes may be either
\[
    \#\mathcal G^1 = 2,\qquad\#\mathcal M^1 = 4,\qquad\#\mathcal L^1=8\qquad\text{or}\qquad\#\mathcal H^1=24.
\]
Thus we obtain:
\begin{thm}
\label{thrm:kind-p3-formula}
Let $\mathcal G=\mathcal G_{\mathfrak p^3}$ be as above and let $n\in\zet^+$ have $2$-adic valuation $e$. Then
\begin{equation}
\label{eq:kind-p3}
r_{\mathcal G}(n) = 2\sum_{2\nmid d\mid n} d \cdot\begin{cases}
    1, & \text{$e=0$,}\\
    2, & \text{$e=1$,}\\
    4, & \text{$e=2$,}\\
    12, & \text{$e\geq3$.}
\end{cases}
\end{equation}
\end{thm}

Now we focus on the two suborders of the cubians. Note that $\mathcal G_{\mathfrak q^4}$ is a suborder of $\mathcal G_{\mathfrak q^3}$ and that $\KK/\mathcal G_{\mathfrak q^4}$ is a cyclic module, so the poset of orders is just
\[
    \underbrace{\mathcal G_{\mathfrak q^4}}_{=:\mathcal M_4} \subsetneq \underbrace{\mathcal G_{\mathfrak q^4}+\mathfrak q^3\KK = \mathcal G_{\mathfrak q^3}}_{=:\mathcal M_3} \subsetneq \underbrace{\mathcal G_{\mathfrak q^4}+\mathfrak q^2\KK}_{=:\mathcal M_2} \subsetneq \underbrace{\mathcal G_{\mathfrak q^4}+\mathfrak q\KK}_{=:\mathcal M_1}\subsetneq\underbrace{\KK}_{=:\mathcal M_0},
\]
where $\mathfrak q = \sqrt2\OK = (2-\sqrt2)\OK$. Explicitly, the cyclic module $\mathcal M_0/\mathcal M_4$ is generated by (the class of) $\frac{1+\ii}{\sqrt2}$. Denoting $\mathbf a:=\frac{(1-\sqrt2)+\ii+(\sqrt2-1)\j+\k}2\in\mathcal M_4$, we calculate explicitly that
\[
    \frac{1+\ii}{\sqrt2}\cdot\mathbf a \in \mathcal M_1,
\]
hence $\mathcal M_0\mathbf a\subseteq\mathcal M_1$.
Since $\sqrt2\mathcal M_i\subseteq\mathcal M_{i+1}$, we also analogously obtain \[(\sqrt2)^i\OK\frac{1+\ii}{\sqrt2}\mathbf a \subseteq \mathcal M_{i+1}.\] Then since $\sqrt2\frac{1+\ii}{\sqrt2}\mathbf a\in\mathcal M_2$, we get $\mathcal M_0\mathbf a^2 \subseteq\mathcal M_1\mathbf a =\mathcal M_2+\mathcal M_2\sqrt2\frac{1+\ii}{\sqrt2}\mathbf a \subseteq\mathcal M_2$ etc. -- in general, $\mathcal M_i\mathbf a^j\subseteq \mathcal M_{i+j}$ for all $i$ and $j$ that makes sense.
Then since $\mathcal M_i\mathbf a^j \subseteq\rcol{\mathcal M_{i+j}}{\mathcal M_i}$ and
\[
[\mathcal M_i:\rcol{\mathcal M_{i+j}}{\mathcal M_i}]_{\OK} = (\sqrt2)^{2j}\OK = \nrd(\mathbf a^j)^2\OK,
\]
we get that $\rcol{\mathcal M_{i+j}}{\mathcal M_i} = \mathcal M_i\mathbf a^j$.

Thus when investigating the size of the intersection for each orbit, we will have ${\#(\mathcal H^1\mathbf q\cap \mathcal M_i)} = \#\mathcal M_{i-j}^1$ if and only if $j$ is the largest such that $\mathbf q\in\mathcal H\mathbf a^j$. In other words, for both $\mathcal G_{\mathfrak q^3}=\mathcal M_3$ and $\mathcal G_{\mathfrak q^4}=\mathcal M_4$, we obtain a calculation akin to that of Theorem~\ref{thrm:kind-q2-formula}.
Within this calculation, we only need to find $r_{\mathcal M_i}((2-\sqrt2)^e)$ (here we chose $\pi:=2-\sqrt2$ as a totally positive generator of $\mathfrak q=\sqrt2\OK$) and subsequently multiply by $\sigma_{\mathfrak q\mathfrak D}(\alpha) = \sigma_{\sqrt2\OK}(\alpha)$. When counting $r_{\mathcal M_i}((2-\sqrt2)^e)$, we get the term corresponding to $\mathcal H\mathbf a^j$ contributing
\[
\#\mathcal M_{i-j}^1\zav{\sigma_{\mathfrak D}((2-\sqrt2)^{e-j})-\sigma_{\mathfrak D}((2-\sqrt2)^{e-j-1})} = \#\mathcal M_i^1 \cdot 2^j\cdot 2^{e-j} = 2^e\#\mathcal M_{i}^1
\]
for all $j<i$ and then
\begin{multline*}
\#\mathcal M_0^1 \sigma_{\mathfrak D}((2-\sqrt2)^{e-i}) = \#\mathcal M_i^1 (2^{i}+2^{i-1})(1+2+\cdots+2^{e-i}) =\\= \#\mathcal M_i^1\zav{2(2^{i-1}+\cdots+2^e) -2^{i-1}-2^e},
\end{multline*}
with the contribution from $\mathcal H\mathbf a^j$ only happening if $e\geq j$. Thus, for the small cases, we get $\#\mathcal M_i$ times $1\cdot 2^0$, $2\cdot 2^1$, $\dots$, $i2^{i-1}$, until finally for the general case $e\geq i$ we get
\begin{multline*}
\#\mathcal M_i^i\cdot\zav{ 2(2^{i-1}+\cdots+2^e) -2^{i-1}-2^e + i2^e } =\\= \#\mathcal M_i^i\cdot\zav{ 2(2^{i-1}+\cdots+2^e) + (i-1)2^e -2^{i-1} }
\end{multline*}

Hence, specializing these calculations to $\mathcal M_3 = \mathcal G_{\mathfrak q^3}$ and $\mathcal M_4 = \mathcal G_{\mathfrak q^4}$, which have $\#\mathcal G_{\mathfrak q^3}^1=4$ and $\#\mathcal G_{\mathfrak q^4}^1 = 2$ respectively, we obtain the last of the Jacobi-like formulas:

\begin{thm}
\label{thrm:kind-q3-formula}
Let $\mathcal G=\mathcal G_{\mathfrak q^3}$ be as above and let $\alpha\in\OK^+$ satisfy $\alpha\in(\sqrt2)^e\OK$ but $\alpha\notin(\sqrt2)^{e+1}\OK$. Then
\begin{equation}
\label{eq:kind-q3}
r_{\mathcal G}(\alpha) = 4\sum_{\sqrt2\OK\nmid\delta\OK\mid\alpha\OK}\Nm(\delta\OK) \cdot\begin{cases}
    1, & e=0,\\
    2\cdot 2, & e=1,\\
    3\cdot 2^2, & e=2,\\
    2(2^2+\cdots+2^e) + 2\cdot 2^e - 2^2, & e\geq3.
\end{cases}
\end{equation}
\end{thm}

\begin{thm}
\label{thrm:kind-q4-formula}
Let $\mathcal G=\mathcal G_{\mathfrak q^4}$ be as above and let $\alpha\in\OK^+$ satisfy $\alpha\in(\sqrt2)^e\OK$ but $\alpha\notin(\sqrt2)^{e+1}\OK$. Then
\begin{equation}
\label{eq:kind-q4}
r_{\mathcal G}(\alpha) = 2\sum_{\sqrt2\OK\nmid\delta\OK\mid\alpha\OK}\Nm(\delta\OK) \cdot\begin{cases}
    1, & e=0,\\
    2\cdot2, & e=1,\\
    3\cdot2^2, & e=2,\\
    4\cdot2^3, & e=3,\\
    2(2^3+\cdots+2^e) + 3\cdot 2^e - 2^3, & e\geq4.
\end{cases}
\end{equation}
\end{thm}

Let us remark that in these results, we sometimes relied on certain \uv{happy coincidences} observed on concrete data obtained algorithmically: namely, that all the perceptive suborders from Theorem~\ref{thrm:our-list} happened to have class number $1$, that those of them that were of the kinds $\mathfrak p^a$ and $\mathfrak q^b$ happened to have linear posets of orders, and that the relevant right conductors $\rcol{\mathcal G}{\mathcal M}$ were principal, of the form $\mathcal M\mathbf a$ (which is just equivalent to $\lord{\rcol{\mathcal G}{\mathcal M}}=\mathcal M$, since then we could leverage $\#\Cls\mathcal M \leq\#\Cls\mathcal G = 1$). One avenue of further research might thus be to investigate whether these are indeed coincidences, or rather if they are provable consequences of perceptivity combined with the maximal order having class number $1$. As we alluded to in Remark~\ref{rmrk:nonlinear-poset} as well as in the proof of Theorem~\ref{thrm:kind-qq-formula}, the condition of having a linear poset of orders does not seem to be tightly necessary in our endeavor, because as long as the conductors $\rcol{\mathcal G}{\mathcal M}$ form an opposite poset to the intermediate orders $\mathcal G\subseteq\mathcal M\subseteq\mathcal H$ and satisfy the condition on indices (Proposition~\ref{prp:conductors}(i)), the rest of the subsequent theory can be carried out without major alterations. However, we have not been able to establish Proposition~\ref{prp:conductors}(i) in a more general situation -- the problem is when an order would be covered by the union of its proper suborders.

\section{A~quaternionic proof of G\"otzky's four-square theorem}
\label{chap:gotzky}

In this short section, we will illustrate that perceptivity of a suborder may not be necessary for a Hurwitz-like method -- examining the norm form of a quaternion order through a maximal superorder with class number $1$ -- to succeed. This will be achieved by proving an analogue of Jacobi's four-square theorem in $\kve(\sqrt5)$ originally due to G\"otzky. Throughout this entire section, let us fix $K:=\kve(\sqrt5)$ and its ring of integers $\OK = \zet[\phi]$, where $\phi = \frac{1+\sqrt5}2$ is the golden ratio. Note that $\OK$ has narrow class number $1$.

\begin{thm}[G\"otzky]
\label{thrm:gotzky}
For any $\alpha\in\OK^+$, the equation $\alpha= t^2+x^2+y^2+z^2$ has exactly
\[
    8\sum_{\delta\OK\mid\alpha\OK}\Nm(\delta\OK)-4\sum_{2\OK\mid\delta\OK\mid\alpha\OK}\Nm(\delta\OK)+8\sum_{4\OK\mid\delta\OK\mid\alpha\OK}\Nm(\delta\OK)
\]
solutions $t,x,y,z\in\OK$. In particular, the quadratic form $t^2+x^2+y^2+z^2$ is universal over $K$.
\end{thm}

G\"otzky originally derived this result through an analytic approach \cite{gotzky}, although Kirmse had already studied sums of four squares in $\kve(\sqrt5)$ using quaternions before that \cite{kirmse}. More recently, Deutsch used quaternions and geometry of numbers to prove universality of $t^2+x^2+y^2+z^2$ over $K$, but did not extract the full formula for the number of representations \cite{deutsch-on-gotzky}. Here, we will provide a quaternionic proof of the full theorem. The analytic point of view of Götzky's theorem was also recently examined by Thompson \cite{thompson}.

Of course, G\"otzky's theorem is easily restated as a formula for $r_{\mathcal G}$ with the order $\mathcal G:= \OK\oplus\OK\ii\oplus\OK\j\oplus\OK\k$ in $\quatalg{-1,-1}K$. Within the framework established in Section~\ref{chap:orbits}, the role of the maximal superorder will be played by the \emph{icosian} order
\[
    \II := \OK\oplus\OK\ii\oplus\OK\h\oplus\OK\ii\h,
\]
where $\h = \frac12(\phi+(\phi-1)\ii+\j)$. This is a maximal order in $\quatalg{-1,-1}K$, it has reduced discriminant $\OK$, class number $1$ and its group $\II^1$, the so-called \emph{binary icosahedral group}, has $120$ elements. Both the group $\II^1$ and the icosian order have many remarkable properties, for which we refer the reader to \cite[§8.2]{conway-sloane}.

Let us verify that $\mathcal G$, $\II$ has a linear poset of orders, identifying the poset in the process. We have $[\II:\mathcal G]_{\OK} = 4\OK$ and $2\II\subseteq\mathcal G$; further, in $\OK$, the rational prime $2$ is inert, so $\II/2\II$ is a four-dimensional algebra over the four-element field $k:=\OK/2\OK$ in which $\mathcal G/2\II$ is present as a two-dimensional subalgebra spanned by (the residue classes of) $1$ and $\ii$. Additionally, since $2\OK\nmid \discrd \II$, we have $\II/2\II\simeq\Mat_2(k)$ by Lemma~\ref{lem:goodquotient}.

Suppose $\mathcal H$ is an order with $\mathcal G\subsetneq\mathcal H\subsetneq\mathcal \II$, then $\mathcal H/2\II$ will manifest as a three-dimensional subalgebra of $\II/2\II$ containing $\mathcal G/2\II$. We may thus take the basis of $\mathcal H/2\II$ to be $1$, $\ii$, $\mathbf q$ for some $\mathbf q\in \mathcal H$. Then by Lemma~\ref{lem:threedim-subalg}, $(\ii+a)(\mathbf q+b)= 0$ in $\II/2\II$ for some $a,b\in k$, and $\nrd(\ii+a)\equiv \nrd(\mathbf q+b)\equiv 0\pmod{2\OK}$. The only elements of reduced norm divisible by $2$ in the two-dimensional subalgebra $\mathcal G/2\II$ are scalar multiples of $1+\ii$, which forces $a=1$. All non-trivial ideals in $\Mat_2(k)$ are two-dimensional, the set of elements that annihilate $1+\ii$ from the right forms a non-trivial right ideal and it contains the ideal $\overline{(1+\ii)}\II/2\II$, so these two must coincide because they are both two-dimensional. Thus we see that $\mathbf q+b$ must be chosen from $\overline{(1+\ii)}\II/2\II$. But this two-dimensional ideal intersects the two-dimensional subalgebra $\mathcal G/2\II$ in a one-dimensional subspace, so together, they span just a three-dimensional subspace. Thus any choice of $\mathbf q+b$ from $\overline{(1+\ii)}\II/2\II \setminus \mathcal G/2\II$ in fact gives the same three-dimensional subspace.

This means there may be at most one three-dimensional algebra $\mathcal H/2\II$ of the desired properties, and we easily see that $k\oplus k\ii\oplus k\j$ is a such an algebra, since $\ii\j\equiv \j\ii\equiv (\phi-1)+\phi\ii\pmod{2\II}$. This corresponds to the order
\[
    \mathcal H = \OK\oplus\OK\ii\oplus\OK\j\oplus\OK\frac{1+\ii+\j+\k}2,
\]
essentially an analogue over $\OK$ of the Hurwitz order over $\zet$. Thus we have identified the full poset of orders between $\mathcal G$ and $\II$, this poset being $\mathcal G\subset\mathcal H\subset\II$.

Straightforwardly, one calculates that $\#\mathcal H^1=24$ and $\#\mathcal G^1 = 8$. Alongside $\#\II^1=120$, applying Proposition~\ref{prp:perceptivity-count}, we obtain that $\mathcal H$ is $\II$-perceptive but $\mathcal G$ is not.

Next, we may wish to examine the right conductors arising in this poset. $\rcol{\mathcal G}{\II}$ must have $[\mathcal G:\rcol{\mathcal G}{\II}]_{\OK}=4\OK$ by Proposition~\ref{prp:conductors}, but since $2\II\subseteq\mathcal G$, it follows that $2\II\subseteq\rcol{\mathcal G}{\II}$, so just by considering indices, we obtain $\rcol{\mathcal G}{\II}=2\II$. Next we show $\rcol{\mathcal G}{\mathcal H} = \mathcal H(1+\ii)$. On one hand, we must have $[\mathcal H:\rcol{\mathcal G}{\mathcal H}]_{\OK} = 4\OK$, on the other, we see $\mathcal H = \mathcal G+\mathcal G\frac{1+\ii+\j+\k}2$, so
\[
    \frac{1+\ii+\j+\k}2\cdot (1+\ii) = \ii+\j\in\mathcal G
\]
implies $1+\ii\in\rcol{\mathcal G}{\mathcal H}$ and so $\mathcal H(1+\ii)\subseteq\rcol{\mathcal G}{\mathcal H}$. Due to indices, we then get $\rcol{\mathcal G}{\mathcal H}=\mathcal H(1+\ii)$.

\begin{lemma}
    \label{lem:units-on-odd-lines}
    For each $\tilde{\mathbf u}\in \II/2\II$ with $\nrd(\tilde{\mathbf u})\equiv 1\pmod{2\OK}$, there is a $\mathbf u\in\II^1$ such that $\mathbf u\equiv \tilde{\mathbf u}\pmod{2\II}$. As a consequence, for any $\mathbf q\in\II$ with $2\nmid\nrd(\mathbf q)$, the set $\II^1\mathbf q\cap \mathcal G$ is non-empty.
\end{lemma}
\begin{proof}
First, we claim that for $\mathbf u_1,\mathbf u_2\in \II^1$, we have $\mathbf u_1\equiv \mathbf u_2\pmod{2\II}$ if and only if $\mathbf u_1=\pm\mathbf u_2$. The \uv{if} part is obvious, so to prove the \uv{only if} part, let us presume that $\mathbf u_1\neq\pm\mathbf u_2$ and show that $\mathbf u_1\nequiv \mathbf u_2\pmod{2\II}$.

Since $\mathbf u_1+\mathbf u_2\neq 0$, we have
\[
    0\prec \nrd(\mathbf u_1+\mathbf u_2) = 1+\trd(\mathbf u_1\overline{\mathbf u_2})+1,
\]
hence $\trd(\mathbf u_1\overline{\mathbf u_2})\succ -2$. Using this bound, we then have
\[
    0\prec\nrd(\mathbf u_1-\mathbf u_2) = 2-\trd(\mathbf u_1\overline{\mathbf u_2}) \prec 4
\]
and taking the field norm yields $0<\Nm(\nrd(\mathbf u_1-\mathbf u_2)) < 16$.
Now if it were the case that $\mathbf u_1\equiv \mathbf u_2\pmod{2\II}$, it would imply $4\mid \nrd(\mathbf u_1-\mathbf u_2)$, hence $16\mid \Nm(\nrd(\mathbf u_1-\mathbf u_2))$, which is a contradiction with the previous bound.

Now, since $\#\II^1 = 120$, these units must occupy $60$ distinct residue classes in $\II/2\II\simeq\Mat_2(\OK/2\OK)$. This algebra has $256$ elements, namely the zero element and $255$ non-zero ones, which we may group in $255/3=85$ lines (one-dimensional $\OK/2\OK$-subspaces) they generate. On a line, either all four elements have reduced norm $0\in\OK/2\OK$, or they are non-zero, and then $\nrd$ takes all four values from $\OK/2\OK$ (this is because all elements are squares in this finite field). Let us call the former an \emph{even line} and the latter an \emph{odd line}.

Non-zero elements from even lines must generate a non-trivial left ideal, which is then two-dimensional, so it has $15$ non-zero elements or equivalently $5$ even lines. Different left ideals have trivial intersections, and since $\II/2\II\simeq\Mat_2(\OK/2\OK)$, there are $\Nm(2)+1 = 5$ of these non-trivial ideals. Thus, we count that there are $5\cdot 5=25$ even lines in all of $\II/2\II$, leaving $85-25 = 60$ odd lines.

Now, we know units from $\II^1$ occupy $60$ distinct residue classes. Each such class has reduced norm $1$, and on each odd line, there is only one such residue class. So we see that there are only $60$ classes with $\nrd(\tilde{\mathbf u})=1$, hence each must be represented by exactly two units from $\II^1$.

To prove that $\II^1\mathbf q\cap \mathcal G\neq\emptyset$ for $2\nmid\nrd(\mathbf q)$, we just take the line of $\overline{\mathbf q}+2\II$, which must be an odd line, find its element of reduced norm $1$ and represent it by a $\mathbf u\in\II^1$. This then ensures that $\mathbf u\mathbf q$ lies on the same line as $\overline{\mathbf q}\mathbf q = \nrd(\mathbf q)$, that is, the line of $1$, which lies in $\mathcal G/2\II$.
\end{proof}

\begin{lemma}
\label{lem:which-orbits-hit}
    For $\mathbf q\in\II$ with $2\mid\nrd(\mathbf q)$, the set $\II^1\mathbf q\cap \mathcal G$ is non-empty if and only if $\mathbf q\in\II(1+\ii)$.
\end{lemma}
\begin{proof}
Let us denote $k:=\OK/2\OK$.
First, suppose that $\II^1\mathbf q\cap \mathcal G$ is non-empty, then we may without loss of generality presume $\mathbf q$ already lies in $\mathcal G^1$. If $\mathbf q\in 2\II = (\II(1+\ii))(1+\ii)$, then the conclusion holds, so we may presume $\mathbf q+2\II$ is non-zero in $\II/2\II$. In $\II/2\II$, the residue class of $\mathbf q$ then lies in the subalgebra spanned by $1$ and $\ii$. The norm form in this two-dimensional subalgebra is
\[
    \nrd(x+y\ii) = x^2+y^2 = (x+y)^2
\]
for $x,y\in k$, so the only elements with reduced norm zero are located on the line generated by $1+\ii$. This means that in $\II/2\II$, the left ideals generated by $\mathbf q$ and $1+\ii$ intersect non-trivially, so they must in fact coincide, hence $\mathbf q+2\II\in(\II/2\II)(1+\ii)$. Lifting back from $\II/2\II$ to $\II$, which we may do due to $2\II\subset\II(1+\ii)$, we obtain $\mathbf q\in \II(1+\ii)$.

Second, suppose that $\mathbf q\in \II(1+\ii)$. If actually $\mathbf q\in 2\II$, then trivially $\mathbf q\in\mathcal G$, so we may presume $\mathbf q\notin2\II$, which then means $(\II/2\II)\mathbf q = (\II/2\II)(1+\ii)$ and thus $\mathbf a\mathbf q \equiv 1+\ii\pmod{2\II}$ for some $\mathbf a\in\II$. We will show that $\mathbf a$ may in fact be chosen from $\II^1$.

Notice that the polynomial $x^2+x+\phi$ has no root in $k$, so after homogenizing, $x^2+xy+\phi y^2$ is only zero if $x\equiv y\equiv 0$. In view of the isomorphism $\II/2\II\simeq\Mat_2(k)$, we have the matrix $\mathbf b:=\mtrx{1&\phi\\1&0}$ with $\trd(\mathbf b)=1$ and $\nrd(\mathbf b) = \phi$, so
\[
    \nrd(x+y\mathbf b) = x^2+xy+\phi y^2
\]
for $x,y\in k$. This implies that $1$ and $\mathbf b$ span a two-dimensional subalgebra $B$ of $\II/2\II$ that is a quadratic field extension of $k$. In particular, each of its non-zero elements has a scalar multiple that may be represented by a unit from $\II^1$ by Lemma~\ref{lem:units-on-odd-lines}

Let us prescribe a $k$-linear map
\begin{align*}
    \mu: B & \to (\II/2\II)\mathbf q\\
    \mathbf x &\mapsto \mathbf x\mathbf q.
\end{align*}
Since all non-zero elements of $B$ have non-zero reduced norms, they are invertible in $\II/2\II$, so $\mathbf x\mathbf q \equiv 0$ would imply $\mathbf q\equiv0$ for $\mathbf x\in B\setminus\set0$, which is absurd, and thus $\ker\mu = 0$. Now both $B$ and $(\II/2\II)\mathbf q$ are two-dimensional spaces, so injectivity of $\mu$ implies its surjectivity. Hence we may take $1+\ii\equiv \mathbf a\mathbf q\pmod{2\II}$ for some $\mathbf a\in B$. But now, possibly after taking a scalar multiple, $\mathbf a$ is represented by some $\mathbf u\in\II^1$. Thus we get $\mathbf u\mathbf q \in k(1+\ii)\subseteq \mathcal G/2\II$, hence $\mathbf u\mathbf q\in \mathcal G$.
\end{proof}

\begin{prop}
\label{prp:intersection-sizes}
For $\mathbf q\in\II$, we have
\[
    \#(\II^1\mathbf q \cap \mathcal G) = \begin{cases}
        8, & \text{if $2\nmid\nrd(\mathbf q)$},\\
        0, & \text{if $2\mid\nrd(\mathbf q)$ but $\mathbf q\notin\II(1+\ii)$},\\
        24, & \text{if $\mathbf q\in\II(1+\ii)$ but $\mathbf q\notin2\II$},\\
        120, & \text{if $\mathbf q\in2\II$.}
    \end{cases}
\]
\end{prop}
\begin{proof}
The case when $\mathbf q\in2\II$ is obvious and the case of $2\mid\nrd(\mathbf q)$ but $\mathbf q\notin\II(1+\ii)$ stems from Lemma~\ref{lem:which-orbits-hit}. When $2\nmid \nrd(\mathbf q)$, then surely $\mathbf q\notin\mathcal H(1+\ii) = \rcol{\mathcal G}{\mathcal H}$, so by Lemma~\ref{lem:orbit-intersections}, we obtain $\#(\II^1\mathbf q\cap \mathcal G) = \#\mathcal G^1 = 8$.

For the case of $\mathbf q\in\II(1+\ii)$ but $\mathbf q\notin2\II$, note that since $\rcol{\mathcal G}{\mathcal H}=\mathcal H(1+\ii)$, the fact that $\mathbf q\in\II(1+\ii)$ implies that $\II^1\mathbf q$ intersects $\rcol{\mathcal G}{\mathcal H}$ by Lemma~\ref{lem:orbit-and-conductor}, whence the conclusion follows by Lemma~\ref{lem:orbit-intersections} due to $\mathbf q\notin2\II = \rcol{\mathcal G}{\II}$.
\end{proof}

With this, we are ready to prove G\"otzky's theorem:
\begin{proof}[Proof of Theorem~\ref{thrm:gotzky}]
Recall that $\II$ has reduced discriminant $\OK$, hence $r_{\II}(\alpha) = 120 \sigma_{\OK}(\alpha)$.
Let us consider an $\alpha\in\OK^+$ and derive a formula for $r_{\mathcal G}(\alpha)$. For this, let us write $\alpha=2^e\beta$ for some $e\geq0$ and $2\nmid \beta\in\OK^+$; this is valid because $2$ is a prime element in $\OK$. Any quaternion in $\II$ of reduced norm $\alpha$ may then be written as $\mathbf b\mathbf q$, where $\nrd(\mathbf b) = \beta$ and $\nrd(\mathbf q) = 2^e$. Then $\mathbf b$ is invertible in $\II/2\II$; we may further notice that each of the four conditions in Proposition~\ref{prp:intersection-sizes} may be recognized by just looking at residue classes in $\II/2\II$ and that each is unchanged when multiplying by an invertible element from the left. Hence $\#(\II^1\mathbf b\mathbf q\cap\mathcal G) = \#(\II^1\mathbf q\cap\mathcal G)$. So, to count elements of reduced norm $\alpha$ in $\mathcal G$, it suffices to count those of reduced norm $2^e$ and multiply the result by $\sigma_{\OK}(\beta)$.

First let us deal with the cases of small $e$. If $e=0$, there is only one orbit and it intersects $\mathcal G$ in $8$ elements, which is consistent with the desired formula. If $e=1$, there are $5$ orbits, but only of them intersects $\mathcal G$, namely in $24$ elements. We express this as
\[
    8(1+4) - 4\cdot 4,
\]
so it is again consistent with the desired formula.

Now, we may presume $e\geq2$. In $\II$ there are $1+4+\cdots+4^{e-1}+4^e$ orbits quaternions of reduced norm $2^e$ with respect to the action of $\II^1$ acting by multiplication from the left. Of these, $1+4+\cdots+4^{e-2}$ lie in $2\II$, a further
\[
    \zav{1+4+\cdots+4^{e-1}} - \zav{1+4+\cdots+4^{e-2}} = 4^{e-1}
\]
lie in $\II(1+\ii)\setminus2\II$ while the remaining $4^e$ lie outside of $\II(1+\ii)$. By Proposition~\ref{prp:intersection-sizes}, the orbits in these three groups contribute $120$, $24$ and $0$ quaternions each.

So with these weights, we get the total number of these quaternions in $\mathcal G$ as
\begin{align*}
    r_{\mathcal G}(2^e) &= 120\zav{1+4+\cdots+4^{e-2}} + 24\cdot 4^{e-1} =\\
    &= 120\cdot \frac{4^{e-1}-1}{4-1} + 24\cdot 4^{e-1} = 40\cdot 4^{e-1}-40+24\cdot 4^{e-1} =\\
    &= 64\cdot 4^{e-1}-40 = 4^{e+2}-40.
\end{align*}
On the other hand,
\begin{multline*}
    8\sum_{\delta\OK\mid2^e\OK}\Nm(\delta\OK)-4\sum_{2\OK\mid\delta\OK\mid2^e\OK}\Nm(\delta\OK)+8\sum_{4\OK\mid\delta\OK\mid2^e\OK}\Nm(\delta\OK) =\\
    = 8\zav{1+4+\cdots+4^e} - 4\zav{4+\cdots+4^e} + 8\zav{4^2+\cdots+4^e} = \\
    = 8\cdot\frac{4^{e+1}-1}{4-1} - 4\cdot\frac{4^{e+1}-4}{4-1} + 8\cdot\frac{4^{e+1}-16}{4-1} = \frac{12\cdot4^{e+1}-8+16-128}3 = 4^{e+2}-40,
\end{multline*}
so indeed the two quantities agree. Thus we have shown that the formula of the theorem holds for $\alpha=2^e$. But then by $\beta$ being coprime to $2$, multiplying both quantities by $\displaystyle\sigma_{\OK}(\alpha) = \sum_{\delta\OK\mid\beta\OK}\Nm(\delta\OK)$ finishes the proof.
\end{proof}

To conclude, let us remark that this proof suggests that a Hurwitz-like technique on a pair of orders $\mathcal G\subseteq\mathcal H$ may be performed even with some weaker properties of the $\mathcal H^1$-action (with $\mathcal H=\II$ there) on $\mathcal H$, compared to perceptivity. If we denote $\mathfrak a$ some ideal such that $\mathfrak a\mathcal H\subseteq\mathcal G$, it seems hard to imagine that any Hurwitz-like technique could succeed without at least the orbits of elements invertible in $\mathcal H/\mathfrak a\mathcal H$ intersecting $\mathcal G$, but it is unclear where the exact limits are or what a tighter condition for the success of a Hurwitz-like method might look like.

\def\bibfont{\hfuzz=2pt}

\end{document}